\documentclass[reqno,10pt]{amsart}
\usepackage[margin=1in]{geometry}
\usepackage{psfrag}
\usepackage{graphicx}
\usepackage{epsfig} %For pictures: screened artwork should be set up with an 85 or 100 line screen
\usepackage{float}
\usepackage{wasysym}
\usepackage[normalem]{ulem}
\usepackage[active]{srcltx}
\usepackage{amsmath, amsfonts, amscd}%,amssymb,color}
\usepackage{amssymb,amssymb,latexsym}
\usepackage{mathrsfs}
\usepackage{subcaption}
\usepackage[numbers]{natbib}
\usepackage{graphics}
\usepackage{color}
\numberwithin{equation}{section}

\usepackage[colorlinks,bookmarksopen,bookmarksnumbered,citecolor=red,urlcolor=red]{hyperref}
\usepackage{amsaddr}
\usepackage{textcomp, eufrak}
\usepackage{verbatim}
\usepackage{fancyhdr}
\usepackage{BOONDOX-uprscr}

\usepackage{natbib}
\usepackage[dvipsnames]{xcolor}

\usepackage{enumitem}
\usepackage{soul}
\usepackage{centernot}
\usepackage{textgreek}
\usepackage{caption}
\usepackage{appendix}
\usepackage{footmisc}
\usepackage{scalerel}
\newtheorem{theorem}{Theorem}[section]
\newtheorem{proposition}[theorem]{Proposition}

\newtheorem{remark}[theorem]{Remark}

\newtheorem{corollary}[theorem]{Corollary}

\newcommand{\PKp}{P_{\K^\perp}}

\newcommand{\J}{\mathfrak{I}}

\newcommand{\bvl}{\big(\hspace{-.12cm}\big(}
\newcommand{\bvr}{\big)\hspace{-.12cm}\big)}

\def\what{\widehat}

\def\bm{\mathbf m}

\def\and{\hbox{and}}

\def\max{\text{\rm max}}

\newcommand{\R}{{\mathbb{R}}}

\newcommand{\N}{{\mathbb N}}

\newcommand{\I}{\mathcal{I}}

\def\k0{\kappa_0}

\def\leaderfill{\leaders\hbox to 1em{\hss-\hss}\hfill\ }
\newcommand{\Pb}{\mathbb{P}}
\newcommand{\Nc}{\mathcal{N}}
\newcommand{\Yc}{\mathcal{Y}}

\newcommand{\comments}[1]{}

\def\lf{\left} 
\def\rg{\right}

\renewcommand{\phi}{\varphi}
\newcommand{\nn}{\nonumber}
\newcommand{\D}{\displaystyle }
\renewcommand{\P}{\mathbb P}

\newcommand{\h}{\mathbb H}
\newcommand{\cal}{\mathcal }
\newcommand{\bu}{{\mathbf u}}
\newcommand{\bz}{{\mathbf z}}
\newcommand{\bv}{{\mathbf v}}
\newcommand{\bw}{{\mathbf w}}
\newcommand{\dt}{{\D\frac{\textrm{d}}{\textrm{d}t}}}
\newcommand{\ra}{\rightarrow}

\renewcommand{\l}{\langle}

\newcommand{\bdy}{\partial}

\newcommand{\Om}{\Omega}
\newcommand{\De}{\Delta}
\newcommand{\nab}{\nabla}

\newcommand{\p}{\mathbb P}
\newcommand{\K}{\cal K}

\newcommand{\bW}{{\bf W}}

\newcommand{\Sob}[2]{\lVert#1\rVert_{#2}}

\newcommand{\lra}{\longrightarrow}
\newcommand{\co}{{\cal O}}
\newcommand{\be}{\begin{equation}}
\newcommand{\ee}{\end{equation}}
\newcommand{\beas}{\begin{eqnarray*}}
	\newcommand{\eeas}{\end{eqnarray*}}
\newcommand{\bea}{\begin{eqnarray}}
\newcommand{\eea}{\end{eqnarray}}
\newcommand{\mbb}[1]{{\mathbb{#1}}}

\newcommand{\E}{\mbb{E}}

\newcommand{\om}{\omega}

\newcommand{\mbf}{\mathbf}

\newcommand{\bbr}{\mathbb{R}}

\newcommand{\pp}{\hspace*{.000cm}\mathord{\raisebox{-0.095em}{\scaleobj{.86}{\includegraphics[width=1em]{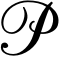}}}}\hspace*{.02cm}}

\newcommand{\dr}{\textrm{d}}

\newcommand{\bmk}{\mathbf{m}^{(k)}}
\newcommand{\hbmk}{\widehat{\mathbf{m}}^{(k)}}
\newcommand{\hbm}{\widehat{\mathbf{m}}}

\newcommand{\bek}{\mathbf{e}^{(k)}}
\newcommand{\beb}{\overline{\mathbf{e}}}
\newcommand{\bee}{\mathbf{e}}
\newcommand{\by}{\mathbf{y}}

\newcommand{\charfn}[1]{{\raisebox{1.2pt}{\mbox{$\chi
				_{\kern-1pt\lower3pt\hbox{{$\scriptstyle{#1}$}}}$}}}}

\title[Unified Approach to data Assimilation]{A unified framework for the analysis of accuracy and stability of a class of approximate Gaussian filters for the Navier-Stokes Equations }

\author{\vspace*{-.6cm} Animikh Biswas$^{\dagger}$ and Micha\l \;Branicki$^{\,\ddagger*}$ }

\address{\small \vspace*{-0.cm} $^\dagger$ Department of Mathematics and Statistics\\ % \hfill (Received 00 00 201?)\\
	University of Maryland Baltimore County  \\ %\hfill (Revised  00 00 201?)\\
	Baltimore, MD 21250, USA\\[.2cm]
	$^{\,\ddagger}$\small
Department of Mathematics, University of Edinburgh, Scotland, UK \\
\hspace*{-1.9cm} $^{\,*}$\!\! The Alan Turing Institute for Data Science, London, UK	}

\email{abiswas@umbc.edu}

\email{m.branicki@ed.ac.uk}

\begin{document}

\begin{abstract}
Bayesian state estimation of a dynamical system utilising  a stream of noisy measurements is important in many geophysical and engineering applications. In these cases,  nonlinearities, high (or infinite)  dimensionality of the state space, and sparse observations pose key challenges for deriving efficient and accurate data assimilation (DA) techniques.  A number  of DA algorithms used commonly in practice, such as the Ensemble Kalman Filter (EnKF) or Ensemble Square Root Kalman Filter (EnSRKF), suffer from serious drawbacks such as  {\em catastrophic filter divergence} and filter instability. The analysis of stability and accuracy of these DA schemes has thus far focused either on finite-dimensional dynamics, or on complete observations in case of infinite-dimensional dissipative systems.  We develop a unified framework for the analysis of several well-known and empirically efficient data assimilation techniques derived from various Gaussian approximations of the Bayesian filtering schemes for geophysical-type dissipative dynamics with quadratic nonlinearities. 
We establish  rigorous results on (time-asymptotic) accuracy and stability of these algorithms with  general  covariance and observation operators.    {\em The accuracy and stability} results for EnKF and EnSRKF for dissipative PDEs are, to the best of our knowledge, completely new in this general setting. It turns out that a hitherto unexploited cancellation property involving the ensemble covariance and observation operators and the concept of {\em covariance localization}  in conjunction with covariance inflation play a pivotal role in the accuracy and stability for EnKF and EnSRKF.  Our approach also elucidates the links, via {\em determining functionals},  between the approximate-Bayesian and control-theoretic approaches to data assimilation. We consider the `model'  dynamics governed by the two-dimensional incompressible Navier-Stokes equations and observations given by noisy measurements of averaged volume elements or spectral/modal observations  of the velocity field. In this setup, several continuous-time data assimilation techniques, namely the so-called 3DVar, EnKF  and EnSRKF  reduce to  a stochastically forced Navier-Stokes equations.  For the first time, we derive conditions for accuracy and stability  of  EnKF and EnSRKF.  The derived bounds are given for the limit supremum of the expected value of the $L^2$ norm and of the $\mathbb{H}^1$ Sobolev norm of the difference between the approximating solution and the actual solution  as the time tends to~infinity. Moreover, our analysis reveals an interplay between the resolution of the observations (roughly,  the `richness' of the observation space) associated with the observation operator underlying the data assimilation algorithms and  covariance inflation and localization which are  employed in practice for improved filter performance.

\end{abstract}

\maketitle

\setcounter{tocdepth}{2}
\vspace*{.5cm}\tableofcontents

\section{Introduction}\label{intro}
The goal of data assimilation (DA) is to  estimate and/or  predict the state  of a dynamical system,  given the (often approximate) dynamics and usually partial and noisy observations of the state in the absence of the knowledge of the initial condition. This setup is commonly encountered in applications, ranging from autonomous navigation to weather prediction, and it poses multiple practical and theoretical  challenges which most often stem from the sensitivity of the dynamics to the initial condition and the high dimensionality of the state to be estimated. There are two popular and competing methodologies for data assimilation: the Bayesian filtering approach (e.g., \cite{Kal60, LSZbook2015}) and the variational approach (e.g., \cite{TK97,LSZbook2015}).  Considering the data assimilation problem within the framework of Bayesian statistics stems from the fact that the underlying model and/or the observational data are uncertain, and Bayesian filtering aims to update the (filtering) distribution of the system state using generally noisy and partial empirical data. This process may be performed exactly for linear systems subject to Gaussian noise where the solution is given by the Kalman filter~\cite{Kal60}. For the case of nonlinear and non-Gaussian scenarios, the majority of Bayesian DA  algorithms implemented in practice rely on ad hoc simplifications which are driven by  both physical insights and computational expediency. Many of these  methods  invoke some form of ad hoc Gaussian approximation.    However,  theoretical understanding of the ability of such methods to accurately and reliably estimate the state variables is under-developed. Thus the development of practical filtering algorithms for high-dimensional dynamical systems is an active research area and for further insight into this subject we refer the reader to, for example,  \cite{BCJ11, ChMT10, Evens09, HM08, TK97,PVL09,PVL10}.

In this work, we consider two classes of approximate Gaussian DA algorithms, namely the so-called {\em 3DVar filter,} the ensemble based {\em Ensemble Kalman Filter} (EnKF) and its close variant, the {\em Ensemble Square Root Kalman Filter} (EnSRKF),  together with the purely control-theoretic method for state estimation of nonlinear, infinite-dimensional dissipative dynamics referred to as  {\em Nudging}. The simplest such algorithm 
referred to as 3DVar for historical reasons enforces a Gaussian filtering distribution with a constant covariance, effectively  neglecting the propagation of the uncertainty  in the produced estimates.   The 3DVar filter has its origin in weather forecasting \cite{Lor86} and is prototypical of more sophisticated filters used today in the nonlinear and non-Gaussian setting, including EnKF and EnSRKF.  Similar to the Kalman filter, which can be derived from the control-theoretic viewpoint as the exact solution of the associated linear-quadratic minimization problem,  3DVar  may also be viewed within the framework of nonlinear control theory and thereby be derived directly, without reference to the Bayesian probabilistic interpretation. In fact,  this is primarily how  approximate Gaussian filtering algorithms were conceived.  In particular, 3DVar  turns out to  share many common features with the recently developed approach  connected to ideas from control theory (e.g. Luenberger Observer Theory \cite{Luenberger1971,Nijmeijer2001,Thau1973}) which is  referred to hereafter as {\it nudging}. Nudging has been proposed in \cite{AOT}, and subsequently applied to many different settings including  localized and mobile data; see for instance \cite{AlbanezNussenzveigLopesTiti2016, BBJ2020, CLT, carlson2018,  CHLM, FGMW, FLT,FoiasMondainiTiti2016,FLV} and the extensive references therein.  The nudging  DA algorithm relies on  
the fact that dissipative dynamical systems possess finitely many {\em determining parameters}\footnote{\,In the infinite-dimensional or PDE case such as determining modes, nodes and local volume averages, see, for example, \cite{CJT97,FMRT01, FT84,FTT91, HT97, JT92, JT93} and references therein.} and   introduces a feedback control term that forces the approximating solution obtained by data assimilation towards the reference solution that is only accessible through finite-dimensional observations (determining parameters);  generalization of this framework to account for noisy observations was developed in \cite{BOT}.  While 3DVar and nudging are structurally similar, their derivation provides a complimentary point of view and insight into  the fundamental property of  observations that is necessary (though not sufficient) for the utility of all the aforementioned DA algorithms.

In  general, a Bayesian DA procedure (including 3DVar, EnKF, and EnSRKF)  computes an update which is a compromise between the model predictions and the data. In this framework, two techniques have been empirically known to improve the performance of Bayesian DA especially in the high-dimensional setting.   {\it  Covariance inflation} is a technique of adding stability to a Bayesian DA algorithm by increasing the size of the model covariance in order to weight the empirical data more heavily. This model-data compromise is weighted by the covariance on the model and the fixed noise covariance on the data. The underlying dynamics and its model typically allow for unstable divergence of trajectories and the potential amplification of errors in the model prediction, whilst the empirical data tends to stabilize the outcome be it at the expense of the corrupting observation noise.  {\it  Covariance localization} implies the reduction of the bandwidth and the rank of the model covariance operator. In practice covariance localization  relies on removing correlations involving components of the state that are deemed to introduce spurious instabilities in the DA estimates. This generally ad hoc procedure is intimately related and constrained by the need to retain correlations between the `unstable' degrees of freedom in the state space so that  {\it assimilation in the unstable subspace} is retained, while spurious correlations between the observational noise the observed unstable directions and the unobserved stable directions are removed.  For practical considerations and implementations of the inflation and localization procedures we refer the reader to, for example, \cite{bishop07,anderson12,menetrier15,roh15,bolin16,smith18,yoshida18,gharamti18,lopez21} but  we stress that there exists a vast literature devoted to these issues which we do not mention here. We provide an important theoretical insight into this issue in the sequel; see Section \ref{secn:cov_localisation_inflation} for more details.

Finally, in contrast to 3DVar and the nudging DA, the {\it ensemble} DA methods attempt to propagate the uncertainty in the estimates (at a low computational cost) based on the  empirical covariance calculated from the propagated ensemble of dynamically coupled state estimates.  Despite the questionable utility of the covariance information evolved by the ensemble-based filters for uncertainty quantification \cite{LS12}, such methods have been known to perform well in practice at providing reliable point estimates in high-dimensional systems even for relatively small ensembles.  However, from the  theoretical perspective, these ensemble-based methods are poorly understood, with the exception of a recently established but still incomplete nonlinear stability theory \cite{TMK15, TMK16,deVRS18,deVT20} that highlights, in particular, the importance of the nature of observations and appropriately dealing with the unobserved components of the underlying dynamics. Moreover, recent numerical and theoretical studies of catastrophic filter divergence \cite{MH2010, GM2013, KMT15} have indicated that stability and accuracy of ensemble based DA methods is a genuine mathematical concern and cannot be taken for granted in practical implementations.

\subsection{Main Contributions.} In what follows we study,  within a unified framework, (i) accuracy and stability of the 3DVar filter for various physically relevant types of observations, and  (ii) accuracy and stability of two ensemble-based DA methods, EnKF and EnSRKF. The term accuracy refers to establishing closeness of the filter estimates to the true signal underlying the observed data, and stability is concerned with the distance between two filter estimates, initialized differently, but driven by the same noisy data. We focus on dissipative dynamical systems, and take the two-dimensional  Navier-Stokes equations as a prototype model in this area which is known to be globally well-posed. To the best of our knowledge, for the first time, we derive accuracy and stability estimates for the EnKF and EnSRKF for general partial observations (i.e.,~when the observation operator has finite rank so that {\em finitely many} modes or volume elements are observed). In our work, we explicate the critical role played by a certain non-trivial `cancellation property' of the trace of a term involving  the ensemble covariance and observation operators. This cancellation property leads to the presence of a negative-definite term in the dynamics of the estimation error and it can be exploited to control the error in the estimates via {\it covariance localization} and {\it additive covariance inflation} provided that the observations are sufficiently `rich' (according to a prescribed criterion); this is achieved by  removing spurious correlations introduced by the ensemble covariance operator between the coarse (observed) and fine (unobserved but `stable') scales of the flow, and making the localized covariance sufficiently dinagonally-dominant. Although the use of covariance localization and inflation has been studied in applied literature (e.g., \cite{bishop07,anderson12,menetrier15,roh15,bolin16,smith18,yoshida18,gharamti18,lopez21}), their  theoretical importance in the stability and accuracy of EnKF and EnSRKF remained unclear to this point.

The most relevant papers from the point of view of our work are those concerned with the theoretical analysis of 3DVar and EnKF/EnSRKF filters. Papers \cite{BLSZ2013, BLLMSS11} describe a theoretical analysis of 3DVar applied to the Navier-Stokes equation when the data arrives in both discrete time \cite{BLLMSS11}, and  in the continuous time \cite{BLSZ2013}; both papers only allow for the possibility of partial observations in Fourier space.  In what follows we allow for much broader types of observations, including volume-averaged  observations,  and covariance operators which are better  justified and suitable in realistic applications. Moreover, this extended analysis paves the way to the core of our  study focussed on  the ensemble-based DA methods, EnKF and EnSRKF, in the infinite-dimensional setting. Both these filters, as well as numerous variations, have been extensively used and studied. The properties of EnKF in the finite-dimensional setting and based on  observations of the full state  were discussed in \cite{ deVRS18,deVT20}, while the well-posedness of EnKF  applied to the Navier-Stokes equation (and for observations in the spectral domain) was discussed in \cite{KLS14}. Nonlinear stability of EnKF and EnSRKF in the finite-dimensional setting was analysed in \cite{TMK15, TMK16}, where the role of covariance inflation was highlighted as crucial for attaining the stability.  However,   these works did not address the  crucial issue of accuracy of  EnKF/EnSRKF, except in \cite{KLS14}, where this issue is addressed only when the full system is observed in Fourier space. It is also worth noting that there is a potential inconsistency in the results of \cite{KLS14} due to assumptions which are incompatible in the infinite-dimensional setting, since the noise covariance in the observations is assumed to be the identity on the state space (i.e., it is not trace class) and  the filter estimates do not have finite energy. Moreover, the role of covariance inflation and localization that has been well known to improve the performance of these ensemble-based filters in practice, remained unclear until now. As a result of the analysis,  we find explicit conditions on the (spatial or spectral) resolution of the observations and the norm and structure of the relevant covariance operators  for the two-dimensional Navier-Stokes equations which guarantees that the resulting approximate  solution converges (in an appropriate sense) to the exact reference solution within an error that is determined by the observation density and the noise amplitude in the observations. In particular, for the first time, we show the importance of covariance localization for accuracy and stability of EnKF/EnSRKF which hitherto was unrecognized in the analysis of filter accuracy and stability; see Section \ref{secn:cov_localisation_inflation} for a detailed discussion.

Our analysis is motivated by the theory developed in \cite{OT11, HOT11, BOT}  which are the first papers to study data assimilation directly through PDE analysis, using ideas from the theory of determining modes in infinite-dimensional dynamical systems.  However, in contrast to those papers, we start from the DA algorithms derived within the Bayesian (and not control-theoretic) setting which allow us to consider a range of approximate Gaussian algorithms, including 3DVar and EnKF/EnSRKF, within the same framework and with systematically derived structure of feedback/control terms which turns out to be crucial in the subsequent analysis.   We make use of the squeezing property of many dissipative dynamical systems \cite{CF88, Tem97}, including the Navier-Stokes equation, which drives many theoretical results in this area.   The current paper provides a generalization of the previous results concerned with 3DVar in the infinite-dimensional setting derived in \cite{BLSZ2013} to allow for physically relevant observations and a general background covariance, and it presents completely new and long-overdue  results concerned with the accuracy and stability  of EnKF and EnSRKF. Moreover, by considering partial noisy observations we provide theoretical insight into  the importance of covariance localization and inflation in DA and the machanism through which these methods can induce accuracy and stability of the DA methods. 
 
 Following \cite{BLSZ2013, KLS14} and working in the limit of high frequency observations we formally introduce continuous-time filters for 3DVar and EnKF; we also derive an analogous continuous-time equation for EnSRKF.  The resulting  stochastic differential equations for state estimation in each of these distinct cases combine the original dynamics with extra terms including mean reversion to the noisily observed signal. 
 In the particular case of the Navier-Stokes equation our continuous time limits give rise to a stochastic partial differential equations (SPDE) with an additional mean-reversion term, driven by spatially-correlated time-white noise. These SPDEs corresponding to the different filters  are  central to our analysis and their properties are used to prove accuracy and stability results of the underlying filters.   We postpone the consideration of discrete-time setting and a rigorous passage to the continuum limit, as well as {\em nodal observations}, i.e., data obtained on finitely many points on a grid in the domain and referred to as Type\,-2 observations in \cite{AOT}, to a future work.

 Finally, we note that a similar analysis of the three-dimensional Navier-Stokes equations is obstructed by the problem of global regularity of solutions; however, it should be possible to treat three-dimensional turbulence models and other dissipative dynamical systems that are more regular using an analysis similar to what we present here. 

\section{Lay outline of main results}

In what follows we consider the accuracy of several DA algorithms whose estimates $\bm(t)$ are constructed from the model and noisy observations  to approximate, in a sense outlined below, the solutions of the underlying truth dynamics $\bu(t)$ contained in some appropriate Hilbert space. We consider the following necessary characteristics of a DA algorithm 
\be  \label{goodfilter_intro}
\left.
 \begin{array}{l}
\text{(i) Filter Accuracy:}\ \limsup_{t \ra \infty} \E \|\bm (t) - \bu(t)\|^2  \lesssim\mathfrak{C}\hspace{.03cm}(\sigma^2),\\[.5cm]
\text{(ii) Filter Stability:}\ 
\E \|\bm^{(1)}(t) - \bm^{(2)}(t)\|^2 \lesssim 
e^{-kt}\E\|\bm^{(1)}(0)-\bm^{(2)}(0)\|^2\\[.3cm]
\qquad \qquad \qquad \qquad \qquad \qquad \qquad \qquad \qquad + \sup_t \E \|\co(\bu^{(1)}(t)-\bu^{(2)}(t))\|^2+\mathfrak{C}\hspace{.03cm}(\sigma^2)
\end{array}
\right\},
\ee
where $\|\cdot\|$ is a norm on an appropriate Hilbert space,  $\bm^{(i)}(\cdot)$, $i=1,2$ are two DA estimates with initial data (distribution) $\bm^{(i)}(0)$,  $\bu^{(i)}(t)$ the reference solutions, and  $\mathfrak{C}\hspace{.03cm}(\sigma^2)\rightarrow 0$ as $\sigma\rightarrow 0$, where $\sigma>0$ denotes the amplitude of the observational noise.

We show that for the reference/truth dynamics given by the (incompressible) two-dimensional Navier-Stokes equation, which is a generic example of a dissipative dynamics of an incompressible fluid the three classes of approximate Gaussian filters (3DVar, EnKF and EnSRKF) are accurate and stable for sufficiently spatially resolved observations and appropriately inflated and localized model covariance.  

In Section \ref{acc_3dvar} we consider the properties of the general 3DVar algorithm in continuous time with observations given either by {\it volume} measurements (i.e., measurements in the physical domain averaged over finite-volume sets) or modal measurements (i.e., measurements of coefficients in an appropriate spectral basis). In Theorem \ref{3dv_acc_2}, Corollary~\ref{3dv_acc_3} and Theorem \ref{3dv_stab_1} we show that for a fairly general (finite-rank) observation operator  the 3DVar  algorithm is accurate and stable in the sense of (\ref{goodfilter_intro}). While 3DVar can be accurate for sufficiently well-resolved observations (Theorem~\ref{3dv_acc_2}), the bounds on its accuracy are improved with appropriate covariance inflation and/or localization Corollary~\ref{3dv_acc_3}. 

In Section \ref{enkf_sec} we consider the properties of continuous-time EnKF and EnSRKF algorithms with the same type of observations as those in 3DVar; namely observations given either by {\it volume} measurements or modal measurements. It turns out that EnKF and EnSRKF cannot be guaranteed to be accurate without an appropriate covariance inflation and localization. In Theorem \ref{enkf_acc_thm} and Theorem \ref{ensrkf_acc_thm} we show for the first time that, respectively, EnKF and EnSRKF are accurate for sufficiently well-resolved observations, sufficiently large additive inflation of the state-dependent covariance which is localized by projecting away terms involving the unresolved degrees of freedom. 

In all cases we point out the possibility of a potential improvement of the asymptotic bounds by a better informed choice of the prescribed `background' covariance (3DVar) or a bespoke covariance inflation in EnKF and EnSRKF; this could, in principle, be achieved via machine learning techniques and we postpone such considerations to future work. 

\section{Mathematical Preliminaries}
\subsection{Approximate Gaussian filtering algorithms in discrete time.}  \label{filtering}
We are given a continuous dynamical system $\big(\Psi_t,H\big)$, where the (flow) map $\Psi_t$ for  ${t\in  [t_0,\,T)}$  on a (Borel measurable) phase space $H$ is induced by the differential equation 
\begin{align}\label{gen_ode_ds}
\frac{\dr}{\dr t} \bu(t)=F(\bu),\qquad \bu(t_0)=\bu_{0} \overset{d}{\sim} \mu_{0}, 
\end{align}
where\,\footnote{\,Strictly speaking $F\,: H_0\subseteq H\rightarrow T_{(\cdot)}H$, where $T_\bu H$ is the tangent space at $\bu\in H_0$.} $F\,: H_0\subseteq H\rightarrow H$ is such that global solutions exist on $[t_0,\,T]$, and the probability distribution $\mu_0\in \pp(H)$ on the initial condition $\bu(t_0)=\bu_0\in H_0$ is unknown. We are provided a (possibly noisy) sequence of observations 
$\{\by(t_j)\}_{j \in \N}$, $t_1<t_2 <\cdots $ of the system. 
More specifically, the realization of (noisy) observations is  given by
 \begin{align}\label{obs}
 \by_j=\co \bu(t_j)+\pmb{\xi}_j,\quad \pmb{\xi}_j = \Xi_j(\omega), \quad {\Xi}_j\sim\Nc(0,\Gamma_j),
\end{align}
where $\by_j\equiv\by(t_j)$,  $\co:H \ra \cal H_\co$ is a linear `observation'  operator from the phase space to a finite dimensional Hilbert space $\cal H_\co$ termed {\it observation space}, and $\big\{\Xi_j\big\}_{j\in \mathbb{N}}$, $\Xi_j\in \cal H_\co$ is the sequence  of independent, Gaussian--distributed random variables that is  independent of the process ${\bu(t),\,t\in [t_0,\,T]}$. The goal in filtering is to obtain the conditional (or {\it filtering}) probability distribution $\Pb (\bu_j|\Yc_j)$, where $\bu_j=\bu(t_j)$, $ \Yc_j=\{\by_i\}_{i=1}^j$ is the given data, since the conditional expectation satisfies the optimality property
 \[
 \E[\bu_j|\Yc_j]=\int \bu_j \,\Pb (\dr\bu_j|\Yc_j)= \inf_{g}\,\E\|\bu_j - g(\Yc_j)\|^2,\qquad  g:\cal H_\co^{\otimes j} \ra H, \quad g \;\; \mbox{measurable}. 
 \]
 The existence of  a regular version of such a distribution, and its representation by an abstract version of Bayes' formula was derived early on in \cite{kalinpur}.
 Assume now that the dynamical system induced by (\ref{gen_ode_ds}) is \underline{autonomous} and \underline{\emph{linear}}, i.e.,  $\Psi_{t}=\Psi$ and  
 $\bu_{j+1}=\Psi(\bu_j)=M\bu_j$ for a bounded linear operator  $M: H \ra H$. An elementary, but fundamental observation in this case is the following: if $\bv \sim \Nc(\bar{\bv}, \Sigma)$, then $M\bv \sim \Nc (M\bar{\bv}, M\Sigma M^*)$, where $M^*$ denotes the adjoint of $M$.
 Thus, assuming  the initial distribution $\bu_0 \sim \Nc(\bm_0, C_0)$, it follows that $\bu_j|\Yc_j$ and $\bu_{j+1}|\Yc_j$ are both Gaussian, with respective means denoted by $\bm_j$ and $\what{\bm}_{j+1}$ and respective covariance matrices denoted by $C_j$ and $\widehat{C}_{j+1}$, i.e.,
 $\bu_j|\Yc_j \sim \Nc(\bm_j,C_j)$ and $\bu_{j+1}|\Yc_j \sim \Nc(\what{\bm}_{j+1}, \what{C}_{j+1})$. Using the facts that $\Xi_j \sim \Nc(\mbf 0, \Gamma)$ and $\bu_{j+1}$ is independent of $\{\Xi_j\}_{j\in \mathbb{N}}$ 
 one readily obtains through the Bayes' formula \cite{kalinpur}  that 
   \be  \label{opt}
 \bm_{j+1}= \underset{\bv \in H}{\rm arg min} \;
 \underbrace{\frac12 \|\Gamma^{-\frac12}(\by_{j+1}-\co\bv)\|^2}_{\text{fit to data}}
 + \ \underbrace{\frac12 \|\what{C}_{j+1}^{-\frac12}(\bv -  \what{\bm}_{j})\|^2}_{\text{fit to model}},
 \ee
 where 
 \begin{equation}
 \what{\bm}_{j+1}=M\bm_j,  \quad \what{C}_{j+1}=MC_jM^*,
 \end{equation}
and $\|C^{-\frac{1}{2}}(\,\cdot\,) \|^2= (\,\cdot\,)^*\,C^{-1}(\,\cdot\,)$. In this setup, {\it covariance inflation} corresponds to increasing the norm of $\widehat C$ which results weighting the estimate towards data; the simplest implementation of such a procedure relies either on an additive inflation (i.e., $\widehat C\mapsto \mu I+\widehat C$, $\mu>0$), or on multiplicative inflation (i.e., $\widehat C\mapsto \mu\widehat C$). Note that one could consider a more elaborate inflation which relies, for example, on replacing $\mu I$ with some bespoke positive definite operator; such a procedure could be carried out, for example,   via machine learning techniques often with a remarkable success \cite{PAS}. {\it Covariance localization} is achieved via `projecting out' correlations between some a priori  selected degrees of freedom; i.e., one takes $\widehat C\mapsto P\widehat CP$, where $P$ is an appropriate projection operator.  

The quadratic minimization problem in (\ref{opt}) has an analytical solution (see, e.g.,  \cite{astrom70, LSZbook2015}) which leads to the following iterative formula for updating $m_j$ and $C_{j}$:
  \begin{align}\label{kalman:state}
 \bm_{j+1}= \what{\bm}_{j+1}+\mathcal{K}_{j+1}(\by_{j+1}- \co \what{\bm}_{j+1}),
 \quad C_{j+1}=(I-\mathcal{K}_{j+1}\co)\what{C}_{j+1},
 \end{align}
 where the so-called {\it Kalman gain} is given by
\begin{align}
\mathcal{K}_{j+1}=\what{C}_{j+1}\co^*\left(\co \what{C}_{j+1}\co^*+ \Gamma\right)^{-1}.
 \end{align}
Note that the above solutions still hold for non-autonomous, affine dynamics and time-dependent observation operator $\mathcal{O}_j$ and $\Gamma_j$ with obvious modifications (as long as the unique, global solutions of (\ref{gen_ode_ds}) exist). Throughout,  we consider the autonomous setup in order to simplify the presentation.

\medskip
 When the evolution map $\Psi$ for $\bu$ is nonlinear, then the conditional  distribution $\bu_j|\Yc_j$ is no longer necessarily Gaussian and  cannot in general be obtained in this manner. In such a  case, \eqref{opt} is taken as a starting point with a suitable choice of $\widehat m_j$ and $\what{C}_j$ parameterizing a Gaussian approximation of the prior $\Pb(\bu_{j+1}|\Yc_{j})$, which  leads to a class of ad-hoc (approximate Gaussian) filters (see, e.g., \cite{BM14, LSZbook2015}) that no longer possess the optimality property. Most notable filters in this family are:
 \begin{itemize}[leftmargin = .9cm]
 \item[(i)] 
 The  3DVar filter which corresponds to setting 
 \begin{equation}
 \widehat \bm_{j+1} = \Psi(\bm_j), \qquad \what{C}_{j+1}=\widehat{C}(t_{j+1}),
 \end{equation}
 where $\Psi: H\rightarrow H$ is the flow map induced by (\ref{gen_ode_ds}) and $\what{C}( \cdot)$ is an a priori  \emph{prescribed} continuous function. This filter is referred to hereafter as the 3DVar for historical reasons that originate in atmospheric sciences.  As remarked above, the dynamics and the observations can be time dependent which implies $\Psi = \Psi_{j}$, $\Sigma = \Sigma_j$, $\co = \co_j$, $\Gamma = \Gamma_j$.

\smallskip 
\item[(ii)]
 The Extended Kalman Filter (ExKF) where 
 \begin{equation}
 \widehat \bm_{j+1} = \Psi(\bm_j), \qquad \what{C}_j=\nabla \Psi (\bm_j){C}_j(\nabla \Psi(\bm_j))^*+\Sigma.
 \end{equation}

 \smallskip
 \item[(iii)]  The {\it perturbed observations} Ensemble Kalman Filter (EnKF), where $\what{C}_{j+1}$ is computed empirically from an ensemble of solutions $\big\{\hbmk_{j+1}\big\}_{k=1}^K$ such that 
 \begin{align}
 \hbmk_{j+1} &= \Psi( \bmk_{j} ), \qquad  \bar{\hbm}_{j+1}= \frac{1}{K} {\textstyle \sum_{k=1}^K \hbmk_{j+1}},\\[.16cm]
 \widehat C_{j+1} &= \widehat C_{j+1}\Big(\{\hbmk_{j+1}\}_{k=1}^K\Big)=\frac{1}{K} {\textstyle \sum_{k=1}^K} (\hbmk_{j+1}) - \bar{\hbm}_{j+1})\otimes (\hbmk_{j+1} - \bar{\hbm}_{j+1}),\\[.5cm]
\bmk_{j+1} &=  \hbmk_{j+1} +\mathcal{K}_{j+1}\big({\bf y}^{(k)}_{j+1}-\mathcal{O} \,\hbmk_{j+1} \big), \mathcal{K}_{j+1} = \widehat C_{j+1}\mathcal{O}^*(\mathcal{O} \widehat C_{j+1}\mathcal{O}^*+\Gamma)^{-1},
 \end{align}
 where $\big\{{\bf y}_{j}^{(k)}\big\}_{k=1}^K$, ${\bf y}^{(k)}_{j }= {\bf y}_{j}+\Gamma^{1/2}\xi^{(k)}_j$, ${\bf y}_j = \co \bu_j+\pmb{\xi}_j$ with $\pmb{\xi}^{(k)}_j,\, \pmb{\xi}_j$ representing independent draws from $\Nc(0,\Gamma)$.

  \medskip
 \item[(iv)] The Ensemble Square-Root Kalman Filter (EnSRKF)
 \begin{align}
 \hbmk_{j+1} &= \Psi( \bmk_{j} ), \qquad  \bar{\hbm}_{j+1}= \frac{1}{K} {\textstyle \sum_{k=1}^K \hbmk_{j+1}}\label{ensrkf_j_beg},\\[.16cm]
 \widehat C_{j+1} &= \widehat C_{j+1}\Big(\{\hbmk_{j+1}\}_{k=1}^K\Big)=\frac{1}{K} {\textstyle \sum_{k=1}^K} (\hbmk_{j+1}) - \bar{\hbm}_{j+1})\otimes (\hbmk_{j+1} - \bar{\hbm}_{j+1}),\\[.5cm]
\bar \bm_{j+1} &=  \bar{\hbm}_{j+1} +\mathcal{K}_{j+1}\big(\by_{j+1}-\mathcal{O} \,\bar{\hbm}_{j+1} \big), \qquad \mathcal{K}_{j+1} = \widehat C_{j+1}\mathcal{O}^*(\mathcal{O} \widehat C_{j+1}\mathcal{O}^*+\Gamma)^{-1},\\[.2cm]
 \bmk_{j+1} &=  \bar{\bm}_{j+1}+S^{(k)}_{j+1},\label{ensrkf_j_end}
 \end{align}
 where $S^{(k)}_{j+1}$ is the $k$-th column of  $S_{j+1} = \widehat S_{j+1}T_{j+1}$ with 
 \begin{align}
 \widehat S_{j+1} = \big[ \hbm^{(1)}_{j+1}- \bar{\hbm}_{j+1}, \dots,  \hbm^{(K)}_{j+1}- \bar{\hbm}_{j+1}\big],
 \end{align}
 and $T_{j+1}$ satisfies 
 \begin{align}
 \widehat S_{j+1}T_{j+1}\big(\widehat S_{j+1}T_{j+1} \big)^*=  K\big(I-\widehat C_{j+1}\mathcal{O}^*(\mathcal{O}\widehat C_{j+1}\mathcal{O}^*+\Gamma)^{-1}\mathcal{O}\big)\widehat C_{j+1} = K(I-\mathcal{K}_{j+1}\co)\widehat C_{j+1}.
 \end{align}

 \noindent By construction, the posterior covariance $C_{j+1} = \frac{1}{K}S_{j+1}S_{j+1}^*$ has the same form as  that in  the Kalman filter.  
 The standard choice of $T_{j}$ is to take it  in  the `square-root' form 
 \begin{align}
 T_{j} = \left(I+\frac{1}{K}\widehat S_j^*\mathcal{O}^*\mathcal{O}\widehat S_j\right)^{-1/2} = \left(I-\frac{1}{K}\widehat S_j^*\mathcal{O}^*(\mathcal{O}\widehat C_{j}\mathcal{O}^*+\Gamma)^{-1}\mathcal{O}\widehat S_j\right)^{1/2}.
 \end{align}
 \end{itemize}
 As in the case of linear Gaussian dynamics, the above filters (i)--(v) can be considered  for time-dependent observation operator $\mathcal{O}_j$ and $\Gamma_j$ with straightforward modifications.

 \medskip
 \subsection{Continuous time limit of approximate Gaussian filter equations}\label{cont_approx_G_list}

 \smallskip
 In order to obtain a continuous time limit (i.e., $t_{j+1}-t_j \ra 0$) of 3DVar, EnKF and EnSRKF (see (i), (iii)--(v) in \S\ref{filtering}), so that the update formulas then correspond to time discretization of an appropriate SDE, 
 one chooses the scaling $\Gamma= h^{-1}\Gamma_0$, where $h = t_{j+1}-t_j$. 
 
 \smallskip
 In the case of 3DVar  choosing, {\em for simplicity}, $\Gamma_0=\sigma^2I$ (i.e., uncorrelated observational noise) leads to a stochastic ODE/PDE for the mean 
 $\bm(\cdot)\in H$ of the form \cite{BLSZ2013}
 
 \be  \label{3dvareq}
 \dr \bm = \big\{F(\bm)- \sigma^{-2}C \co^* \co (\bm - \bu)\big\}\dr t  + \sigma^{-1} C\co^* \dr \bW_t\,,
 \ee
where the covariance operator $C$ with the range  in $H$ is prescribed a priori and is continuous in time, $\bW_t$ is the 
%(cylindrical) 
Wiener process on the Hilbert space $H_\co$ (the observation space). Note that an analogous procedure for linear dynamics in $F$ leads to the Kalman-Bucy filter equations \cite{lipster01} where the evolution of $C= C(t)$ is systematically prescribed.
The stochastic calculus of Hilbert-space-valued semimartingales is developed pedagogically in \cite{daprato92}. 

\smallskip

In the case of EnKF, in each step one needs to generate an ensemble $\{\bmk(t)\}_{k=1}^K$, and proceeding similarly to \cite{KLS2014}, 
 one arrives at \eqref{3dvareq} for each ensemble member $\bmk$ (i.e., replace $\bm$ with $\bmk$ in \eqref{3dvareq}), where 
 \be  \label{enkfcov}
 C(\bm)\equiv C(\{\bmk\}_{k=1}^K)=\frac{1}{K} \sum_{k=1}^K (\bmk) - \bar{\bm})\otimes (\bmk - \bar{\bm}),\;\; \mbox{where}\; 
{  \bar{\bm}= \frac1K \sum_{k=1}^K \bmk}
 \ee
 so that we (formally) obtain 
 \begin{align}\label{enkf_t}
\dr \bmk = \big\{F(\bmk)- \sigma^{-2}\hspace{.04cm}C(\bm) \co^* \co (\bmk - \bu)\big\}\dr t+ \sigma^{-1}\hspace{.04cm}C(\bm)\co^* (\dr \bW_t+\dr {\bf B}^{(k)}_t), 
\end{align}
where  $\bmk(0)=\bmk_{0}, \;\;k=1, \cdots ,K$, and $\bW_t^{(k)}, {\bf B}_t$ are pairwise independent Wiener processes in the observation space $H_\co$. 
 
 \medskip
 In the case of EnSRKF a procedure analogous to the one above (and in  \cite{KLS2014})  yields 
 \begin{align}  \label{ensrkf_t}
\textstyle \dr \bmk &= \Big\{F(\bmk)- \frac{1}{2}\sigma^{-2}\hspace{.04cm}C(\bm) \co^* \co \big( \bmk-\bu\big)-\frac{1}{2}\sigma^{-2}\hspace{.04cm}C(\bm) \co^* \co \big(\bar\bm-\bu\big) )\Big\}\dr t\notag\\[.2cm]&\hspace{10cm}+ \sigma^{-1}\hspace{.04cm}C(\bm)\co^* \dr \bW_t
 \end{align}
with  ${\bf W}_t$ the Wiener processes in the observation space $\cal H_\co$.  Details of the formal derivation are included in Appendix \ref{ensqkf_app}.
 
 \begin{remark}\rm
 Note that, in contrast to the 3DVar algorithm, the covariance operator $C$ in (\ref{ensrkf_t}) is not arbitrarily prescribed and it is state-dependent.  This fact introduces non-trivial differences in the ensemble-based algorithms and the technical differences arising through this state-dependence eluded, until now, a rigorous treatment of the accuracy of such algorithms. We discuss the accuracy of the above ensemble-based algorithms in \S\ref{enkf_sec}.
 \end{remark}

 The standard `nudging'  method, without any observational noise,  yields equation for estimate of the state  $\bm(\cdot)\in H$ on a suitable Hilbert space $H$  in the control-theoretic form~\cite{AOT}
 \be  \label{nudg}
 \dr \bm = \big\{F(\bm)- \mu \,\mathcal{I}(\bm - \bu)\big\}\dr t , 
 \ee
where $F$ has the range in $H$ (as in (\ref{gen_ode_ds})), $\mu>0$ is called the {\it nudging parameter},  $\mathcal{I}= \J\circ\co$, $\co: H\rightarrow \cal H_\co$ has finite rank, and  $\mathfrak{I}:\, \cal H_\co\rightarrow H $ is a linear {\it interpolant} operator 
of the observations. In what  follows we will set $\mathfrak{I} = \co^*$ so that $\mathcal{I} = \co^*\co$.  The role of the control term aimed at nudging the estimates towards $\bu$ is clear. In what follows the observation operator will depend on the parameter $h>0$ denoting the `resolution' of the observations either in the physical or the spectral domain, and $\I_h$  will be associated with the $h$-dependent observations. 

\medskip
The nudging algorithm with observational noise is given by \cite{BOT}
 \be  \label{nudg_W}
 \dr \bm = \big\{F(\bm)- \mu \,\mathcal{I} (\bm - \bu)\big\}\dr t +\mu\hspace{.04cm}\sigma\hspace{.03cm}\dr \bW_t, \quad \mu>0,
 \ee
 where the $Q$-Wiener process $\bW_t$ is in the range of $\mathcal{I}$ and $\sigma\geqslant 0$ is the amplitude of the observation noise. 
 
 \smallskip
 As can be seen above, the equations of the (control-theoretic) nudging algorithms  (\ref{nudg_W}) can be written in the same form as the equations of the approximate Gaussian filters \eqref{3dvareq}, \eqref{enkf_t} and \eqref{ensrkf_t}, subject to some important scaling differences. Within this unified framework specific forms of the involved operators will lead to important differences between  these algorithms.

\begin{remark}\rm
Note that for $C = \mathfrak{c} I$, $\mathfrak{c}>0$ the drift and diffusion in (\ref{nudg_W}) and (\ref{3dvareq}) have the same scaling w.r.t., respectively,  the nudging parameter $\mu$ and the covariance  amplitude~$\mathfrak{c}$, and these two parameters play an equivalent role; in particular, the form of the nudging equations with noisy observations coincides with the restriction of the 3DVar equations to the case of a diagonal covariance.  For a general $C$ in 3DVar one has, in principle, more freedom in choosing the optimal control operator compared to the nudging.  While the ad-hoc structure of the control and noise terms in (\ref{nudg_W}) is consistent with the standard control-theoretic framework, the drift and diffusion in (\ref{3dvareq}) arise from the systematic Bayesian formulation of the filtering problem in continuous time~\cite{BLSZ2013}; this fact will have important consequences later on, especially on the properties of the Ensemble Kalman Filter.  It is interesting to note the scaling w.r.t.~the (prescribed a priori) noise amplitude $\sigma$ implies that the dynamics of continuous-time 3DVar is singular at $\sigma=0$; this behaviour resembles that of the Kalman-Bucy filter (see, e.g., \cite{lipster01}) and it is absent in the discrete time formulation~(\ref{kalman:state}). We postpone the analysis of this singularity to future work. 
\end{remark}

In the sequel we will be concerned with the  interplay between the performance of the filtering and the properties of both the observation operator $\co$ and the covariance $C$. In particular, the issue of the so-called covariance localization will be of specific interest. Moreover, later on we will consider configurations in which the observation and covariance operators are time dependent. We will argue later that the while time-dependence of the covariance in the EnKF case is imposed by the algorithm from the outset, the analysis can be reduced, under certain assumptions, to the 3DVar case. In applications, we will assume that $\co$ is of finite rank, in which case no special assumptions need be made on the observational noise process. In case the range of $\co$ is infinite dimensional, for instance if one observes all modes as in \cite{BLSZ2013, KLS2014}, we will impose the additional restriction that the noise covariance $Q$ is trace class. Note however, that the setup where all modes are observed is not a realistic one and we do not specifically focus on such a configuration.

 \subsection{Observation operators and interpolants}\label{sect:interpolantops}

In this section, we will give specific examples of observation operators occurring in Section \ref{filtering} which are important  in geophysical applications.  
\begin{itemize}[leftmargin = 0.8cm]
\item[(1)] \underline{\it Type\,-1 observation operator}\,: Abstractly, we will assume that  Type\,-1 observation operators~${\co}$ are bounded, linear operators from the (Hilbert) state  space $H$ to some  Hilbert space $H_\co$ ({\it observation space}). Thus,
$\co: H \ra H_\co$ satisfies $\|\co \bu\|_{H_\co} \le \mathfrak{C}\, \|\bu\|_H$.
 The corresponding {\em Type\,-1  interpolant} is given by  $\mathfrak{I} = \co^*$ and,  due to its occurrence and importance in (\ref{nudg}), \eqref{3dvareq}, (\ref{enkf_t}), and (\ref{ensrkf_t}), we will henceforth introduce the operator  $\I = \co^*\co$ to simplify the subsequent notation.
 We also denote the orthogonal projection  $\mathcal{P}_{\K^\perp}$ onto the complement of $\K:= \text{ker}\, \I=\text{ker}\, \co$, where $\I = \mathfrak{I}\circ \co = \co^*\co$.  Accuracy and stability of the approximate Gaussian filters with $\co$ of Type-\,1 is associated with  `richness' of observations is expressed through a version of an {\em approximation of identity property} of $\PKp$ in the form 
\begin{equation}
\|\mathcal{P}_{\K^\perp} \bu\|_{L^2}^2 \le   c_1^2\|\bu\|_{L^2}^2\quad  \mbox{and}\quad \|P_\K \bu\|_{L^2}=\Sob{\bu-\mathcal{P}_{\K^\perp}\bu}{L^2}^2\leq  c_2^2 \,\epsilon \,\Sob{\bu}{\h^1}^2	\quad \forall\  \bu \in \h^1, \label{T1_PKp}
\end{equation}
where $\epsilon$ is a suitable smallness parameter. Note that since in our case, $\K^\perp$ is already a closed subspace in $H$, the first inequality above automatically holds with $c_1=1$ but we retain the general formulation. The inequalities in (\ref{T1_PKp}) quantifies the richness of the space $\K^\perp$ which is finite-dimensional and thus isomorphic to $H_\co$. Therefore, the term {\it observation space} will henceforth be used interchangeably for both $H_\co$ and $\K^\perp$.
Examples of observation operators which are important for geophysical applications, and which satisfy \eqref{T1_PKp} for suitable $\epsilon$  are modal and volume observation operators given below. 

\vspace{.2cm}
\noindent In most applications,  $H_\co$ is finite dimensional, and therefore, we can take it to be $\R^N$. Dealing with infinite-dimensional observation space introduces additional technical complications (on the observational noise process) which seem to us to be purely academic and not relevant in practice. 
Therefore, in case  $H_\co$ is finite dimensional, $\co$ has the form 
\[
\co \bu = (\langle \bu, \psi_1\rangle, \cdots , \langle \bu, \psi_N\rangle), \qquad \{\psi_n\}_{n=1}^N \subset H \;\; \ (\text{or}\ \{\psi_n\}_{n=1}^N \subset H_0).
\]

\smallskip
\item[(2)] \underline{\it Type\,-2 observation operator}\,: For such observation operators  $\co: H_0 \subset H \ra H_\co$ where $H_0$ is a dense subspace, compactly embedded in $H$ and $\co$ is a bounded linear operator satisfying
 $\|\co \bu\|_{H_\co} \le \mathfrak{C}\, \|\bu\|_{H_0}$; i.e., it is generally not a bounded operator on $H$. A typical example of a Type\,-2 observation operator is the nodal observation operator  given below. It also satisfies an approximation property similar to the second inequality in \eqref{T1_PKp}.
\end{itemize}

\medskip
\noindent Some concrete examples of observation operators, important in geophysical applications are 
\begin{itemize}[leftmargin=0cm]
\item[]\underline{\it Type\,-1 observation operators}:

\vspace{.2cm}
\begin{itemize}
\item[(i)] {\it Volume observation operator}. 
	Assume that fluid with velocity $\bu$ occupies 
	 a domain $\mathscr{D}$ partitioned into  a (coarse) grid 	$\{\mathscr{D}_n\}_{n=1}^N$ with
	\begin{equation} \label{h_spatial}
	h :=\underset{n\in \{1,2,\dots, N\}}{\max} \,\text{diam}(\mathscr{D}_n).
	\end{equation}  
		Take $\psi_n=\frac{1}{|\mathscr{D}_n|}\chi_{\mathscr{D}_n}$ and note $\{\psi_n\}_{n=1}^N$ is an orthonormal basis set (which depends implicitly on~$h$). Then, the operator $\co= (\co_1, \dots, \co_N)$ such that $\co_n(\bu) = (\bu, \psi_n)=\frac{1}{|\mathscr{D}_n|}\int_{\mathscr{D}_n} \bu(x)\dr x$ is called the {\it volume observation operator} and $\{\co_n\}_{n=1}^N$ are called volume elements  \cite{AOT, JT92}. The {\it volume interpolant} $\mathfrak{I}:\,H_\co\rightarrow H$ is given by the dual of $\co$, i.e.,  $\mathfrak{I}:=\co^*$, and  
\begin{equation}	\label{h_spectral}	
				\textstyle \I_h:=\co^*\co =  \sum_{i=1}^N \psi_n \otimes \psi_n.
\end{equation}

\vspace{.2cm}		
\item[(ii)] {\it Modal observation operator}.  Let $\{\psi_n\}_{n=1}^N$ be an orthonormal basis of the spectral subspace $H_N$ of the Stokes operator defined in Section \ref{sec:nse}. Then, the operator $\co = (\co_1, \dots, \co_N)$ such that $\co_n(\bu) = (\bu, \psi_n)$ is called the {\it modal observation operator}, the modal interpolant is $\mathfrak{I} = \co^*$, and 
\begin{equation}
\textstyle \I_h = \sum_{n=1}^N \psi_n \otimes \psi_n=P_{H_N}, \qquad h = \lambda_N^{-1/2}.
\end{equation}
\end{itemize}
Observe that for both the modal and volume observation operators, $P_{\K^\perp} = \I_h$ where 
$\K = {\rm Kernel}\  \co$ and $\K^\perp = {\rm Span}\ \{\psi_n\}_{n=1}^N$. The above two Type-1 operators satisfy the {\em approximation of the identity property} due to the estimates \cite{AOT}
\be
\textrm{Type\,-1}:\  \|\I_h \bu\|^2_{L^2} = \|\PKp \bu\|^2_{L^2} \le c^2_1\|\bu\|_{L^2},\qquad \Sob{\bu-\I_h\bu}{L^2}^2 = \|P_\K\bu\|_{L^2}^2\leq c_2^2 \,h^2\Sob{\bu}{\h^1}^2	\quad \forall\  \bu \in \h^1, \label{interpolant:est_1}
\ee
where $L^2(\mathscr{D})=W^{0,2}(\mathscr{D})$,  $\mathbb{H}^k(\mathscr{D})=W^{k,2}(\mathscr{D})$ and $W^{k,2}$ is the $(k,p)$\,- Sobolev space with $p=2$, and here,
$c_1,c_2> 0$ are  absolute constants.  Thus, the modal and volume observation operators satisfy \eqref{T1_PKp} with $c_1=1$ and $\epsilon = h^2$.
\vspace{.2cm}	 

\item[] \underline{\it Type\,-2 observation operator}:

\vspace{.2cm}
\begin{itemize}
\item[(iii)] {\it Nodal observation operator}. Let $\{\mathscr{D}_n\}_{n=1}^N$ be as in (i) and let  $\co(\bu)=(\bu_1,\cdots, \bu_N)$,
where  $\bu_n=\bu(x_n),\ x_n \in \Omega_n$. In this case the linear operator $\co$ is called the {\it nodal observation operator} \cite{AOT}, the {\it nodal interpolant}  $\mathfrak{I}: \,(\R^d)^N \rightarrow H$ is $\mathfrak{I} = \co^*$ is given by $\mathfrak{I} \mbf x =\sum_{n=1}^N \mbf x_n \frac{1}{|\Omega_n|} \chi_{\Om_n}$ where $\mbf x =(\mbf x_1, \cdots, \mbf x_N), \mbf x_i \in \R^d, i=1, \cdots , N$.  Thus,
\begin{equation}
	\textstyle   \I_h\bu=\J\circ \co(\bu)=\sum_{n=1}^N \bu_n \frac{1}{|\Omega_n|} \chi_{\Om_n}.
\end{equation}	  
Though arguably the most  important for applications, this type of operator  is also the most difficult to analyze since  $\co$ is not defined as a bounded operator on all of $H$ but rather as an unbounded operator on a dense subspace $H_0=\h^2 (\mathscr{D})\cap H$. The nodal observation operator satisfy the {\em approximation of identity property} analogous to \eqref{T1_PKp} due to the estimate \cite{AOT}
\be
	    \textrm{Type\,-2}:\  \Sob{\bu-\I_h\bu}{L^2}^2\leq c_3^2 h^2\Sob{ \bu}{\h^1}+c_4^2 h^4\Sob{\De \bu}{L^2}^2\quad  \forall\  \bu \in \h^2,\label{interpolant:est_2}
 \ee
  where $c_1,  c_2,c_3,c_4> 0$ are absolute constants. 
\end{itemize}
\end{itemize}
The nodal and volume observation operators represent  \textit{physical observational data} collected in many geophysical applications while the modal observations may also represent physical observables in certain cases \cite{kt1, kt2}.

\subsection{Filter Accuracy and Stability and Covariance Inflation and Localization.}\label{secn:cov_localisation_inflation}\hspace{\fill}\\

\underline{\bf Observation inequality and observation space:} The analysis of the performance  of  DA approaches has been a developing area over the last decade and the papers \cite{AOT, BLSZ2013, CGTU08, deVRS18,deVT20, HM08, TK97,PVL09,PVL10}, amongst many others,  contain finite- and infinite-dimensional theory and numerical experiments in a variety of finite and discretized infinite dimensional systems. 
Many of the techniques used in this analysis originate from the study of  accuracy and stability results for control systems  \cite{TR76}.  In particular,  it is known (e.g.,  \cite{CGTU08, TP11, frank18})  that, in lay terms, the state space `directions' where solutions of the filter estimates diverge must be observed and constrained by observations, and that the quality of filter estimates can be improved by avoiding assimilation of `stable directions', where the solutions converge. This requirement leads to the notion of  {\it assimilation in the unstable subspace/manifold} which must be finite-dimensional in order for it to be `captured' through finite-dimensional observations. Equivalently,  the number of nonnegative Lyapunov exponents of the underlying attractor must be at less than the rank of the observation operator. In our setting of dissipative PDEs, this is reflected in the richness of observations condition \eqref{T1_PKp} and,  in line with \cite{AOT, BLSZ2013}, achieving filter accuracy and stability for the modal and volume element observations requires observations of low modes or coarse-scales in the state space. In our unified setting for both EnKF/EnSRKF and 3DVar,  we  address the issue of filter divergence and its control via {\it covariance inflation} and {\it covariance localization}; our interpretation of the meaning of these often ambiguous notions is outlined below.  Unsurprisingly, in our rigorous analysis it is critical that the unstable directions are contained in the observation space $\K^\perp$.  Moreover, it turns out that while covariance localisation is not essential in  3DVar, localization onto $\K^\perp$ it is {\it critical for stability and accuracy} of EnKF/EnSRKF.  The rigorous recognition of this empirically observed effect is new to the best of our knowledge.  Whether or not  localization onto $\K^\perp$  (as opposed to some larger space containing  $\K^\perp$) is necessary for stability and accuracy of EnKF/EnSRKF remains an open question at this point.

\smallskip
\underline{\bf Covariance inflation}: The minimisation step (\ref{opt}) in the general Bayesian DA procedure computes an update which is a compromise between the model predictions and the data. This compromise is weighted by the covariance on the model and the fixed noise covariance on the data. The underlying dynamics and its model typically allow for unstable divergence of trajectories and the potential amplification of errors in the model prediction, whilst the empirical data tends to stabilize the outcome, albeit at the expense of the corrupting observation noise. Variance inflation is a technique of adding stability to the algorithm by increasing the size of the model covariance in order to weight the empirical data more heavily. The form of variance inflation that we will study is found by shifting the covariance $C$ by some positive definite linear operator $C^+$. That is, one sets $C\rightarrow C+C^+$; the simplest such a procedure which is typically used in applications is to set $C^+ = \mu I$, $\mu>0$ so that $C\rightarrow C+\mu I$ becomes increasingly diagonally dominant for increasing $\mu$. As we highlight later on, the so-called {multiplicative covariance inflation}, $C\rightarrow \mu C$, $\mu>1$, does not help achieve accuracy especially in the important case of the EnKF filter.

\smallskip
\underline{\bf Covariance localization} implies the reduction of the bandwidth and the rank of the covariance operator. In practice covariance localization  relies on removing correlations involving components of the state that are deemed to introduce spurious instabilities in the DA estimates. Coavariance localization depends on the type of the observation operator and, consequently, the choice of the orthonormal basis $\{\psi_n\}_{n\in \mathbb{Z}}$  in $H$.  In case of both the modal and volume element observations, localization entails removing correlations between coarse and fine scales in the covariance. This generally ad hoc procedure is intimately related and constrained by the need to retain correlations between the `unstable' degrees of freedom so that  {\it assimilation in the unstable subspace}  (e.g.,  \cite{CGTU08, TP11, frank18}) is retained. In other words, covariance localization removes spurious correlations introduced by the observational noise between observed and unobserved directions leading to destabilization of the filter equations.  We make this procedure precise and theoretically justified in the sequel.

\smallskip
As already mentioned in the Introduction, the use of covariance localization and inflation has been studied in applied literature (e.g., \cite{bishop07,anderson12,menetrier15,roh15,bolin16,smith18,yoshida18,gharamti18,lopez21}) but their  theoretical importance in the stability and accuracy of EnKF and EnSRKF remained unclear to this point.

\section{Dynamics}
In the subsequent sections  we will consider accuracy and stability of a class of data assimilation methods aimed at estimating the state of a nonlinear, dissipative dynamical system with a quadratic, energy-preserving nonlinearity. There is a large class of finite- and infinite-dimensional dynamical systems with such a structure that are relevant in geophysical applications. In \S\ref{sec:finds} we introduce several finite-dimensional systems of that type, while \S\ref{sec:nse} outlines the well-known and important benchmark dynamics of the Navier-Stokes equations (NSE) that we focus on in the subsequent sections. Other infinite-dimensional, dissipative systems with the energy-preserving nonlinearity include, e.g., the Barotropic or the Boussinesq equation (e.g., \cite{majdapdewaves}); we postpone the analysis of such dynamics to a separate publication.     
 
\subsection{Finite-dimensional forced, dissipative dynamical systems}\label{sec:finds}
A well known example of a three-dimensional, deterministic, forced dissipative dynamical system is given by the Lorenz 63 model (e.g.,~\cite{Keller}) in the form 
\begin{align}\label{Lorenz1}
\begin{cases}
\dot{x} = -\bar\alpha \hspace{.04cm}x+\bar\alpha \hspace{.04cm} y,\\
\dot{y} = -\bar\alpha \hspace{.04cm} x -\bar\beta \hspace{.04cm} y-xz,\\
\dot{z} = -\bar\gamma \hspace{.04cm}z+xy-\bar\gamma \bar\beta^{-2}\hspace{.04cm}(\bar\varrho+\bar\alpha),
\end{cases}
\end{align}
with parameters $ \bar\alpha, \bar\beta,\bar\gamma,  \bar\varrho>0$. We set  $\bv=(v_1,v_2,v_3) := (x,y,z)\in\R^3$  and consider the  stochastically  perturbed version of (\ref{Lorenz1})  for $t\in \R$ in the form
\begin{align}\label{Lorenz}
d\bv_t = \mathbf{b}(t,\bv_t)\dr t+\sigma(\bv_t)\dr \bW_t = \big[-A\bv_t-B(\bv_t, \bv_t)+\mathbf{F}(t)\big]\dr t+\sigma(\bv_t)\dr \bW_t, \qquad \bv_0\in \R^3,
\end{align}
where
\begin{align} \label{Lcoefs}
A&= \begin{bmatrix}
\bar\alpha& -\bar\alpha& 0\\
\bar\alpha & \bar\beta& 0\\
0 & 0 & \bar\gamma
\end{bmatrix}, \;\; B(\bv,\bv)= \begin{bmatrix}
\;\;\;0\\ \;\;\,v_1v_3\\ -v_1v_2
\end{bmatrix}, 
\quad  \;\sigma(v) = \bar\sigma_1+ \bar\sigma_2 \begin{bmatrix}
v_1 & 0& 0\\
0& v_2&0\\ 0& 0& v_3
\end{bmatrix},
\end{align}
 and $\bW_t =(W_t^1, W_t^2, W_t^3)$ is an independent Wiener process in $\R^3$. It is well-known that for $\bar \sigma_1, \bar\sigma_2 =0$ the system (\ref{Lorenz}) has an absorbing ball  (and an attractor) for all values of the parameters, since for $V(t,\bv) = \|\bv\|^2$ we have 
\begin{align*}
\frac{1}{2}\frac{d V}{dt}=\langle \mathbf{b}(\bv),\bv\rangle&= -\bar\alpha\left(v_1-\frac{F_1}{2\bar\alpha}\right)^2- \bar\beta v_2^2-\bar\gamma\left(v_3 +\frac{\bar\varrho+\bar\alpha}{2\bar\beta^2}\right)^2 +\frac{\bar\alpha \bar\gamma\bar\beta^{-2}(\bar\varrho+\bar\alpha)^2+F_1^2}{4\bar\alpha}.
\end{align*}
Next, note that the linear part in (\ref{Lorenz}) 
\begin{align*}
\langle A\bv, \bv\rangle\geqslant \mathfrak{C}_A \| \bv\|^2, \qquad \mathfrak{C}_A =\min\{\bar\alpha,\bar\beta, \bar\gamma\},
\end{align*}
and the quadratic term  $ B(\bv,\bw)$  satisfies  satisfies for any $\bv,\bw,\bu\in \R^3$
\begin{align}\label{DLor}
\langle B(\bv,\bw), \bw\rangle  = 0, \qquad \langle B(\bv,\bw), \bu\rangle =-\langle B(\bv,\bu), \bw\rangle, \qquad \| B(\bv,\bw)\| \leqslant \| \bv\|\| \bw\|.
\end{align}
Existence of global solutions of (\ref{Lorenz}) was shown in \cite{BU20}. 

\smallskip
Another well known forced dissipative system with the quadratic nonlinearity is given by the Lorenz~96 model which is frequently used as a toy test bed in atmosphere-ocean science; for the state $\bv\in \R^N$, $N\geqslant 4$ this dynamics is given component-wise by
\begin{align}\label{L96}
\dr v^{(i)} =\big( (v^{(i+1)}-v^{(i-2)})v^{(i-1)}-v^{(i)}+F^{(i)}\big)\dr t+\sigma_i(\bv)\dr W^{(i)}_t, 
\end{align} 
where we set $v^{(-1)} = v^{(N-1)}$, $v^{(0)}=v_{(N)}$, $v^{(N+1)} = v^{(1)}$, and the quadratic term in the drift satisfies~(\ref{DLor}).  The deterministic setup with $\sigma\equiv 0$ and $F =8$ is commonly used in numerical considerations for dynamics with a chaotic attractor. Existence of global solutions of (\ref{L96}) with the diffusion $\sigma$ as in $(\ref{Lcoefs})$ can be shown in a way analogous to that for (\ref{Lorenz}) (see \cite{BU20}). 

\subsection{Incompressible Navier-Stokes equations (NSE)} \label{sec:nse}
 The NSE is given by
	\begin{align}\label{nse}
		\bdy_t\bu-\nu\De \bu+\bu\cdotp\nabla \bu+\nab p=\mbf{f},\quad\nab \cdotp \bu=0,\quad \bu(x,0)=\bu_0(x), \;x\in\mathscr{D} \subset \R^d.
	\end{align}
Here $\bu$ represents the fluid velocity, $\nu>0$ is the kinematic viscosity, $p$ is the scalar pressure field, $\mbf f$ is a given body force, $\bu_0$ is the initial velocity, and incompressibility is expressed by $\nab \cdotp\bu=0$. For simplicity, we assume $\bu$ is space-periodic and employing Galilean invariance,
we assume $\int_{\mathscr{D}} \bu \dr x =0$.
 The state  space $H$ for the NSE is defined to be
 \[
 H=\left\{\bu \in L^2(\mathscr{D}): \;\nabla \cdot \bu=0, \;\;
 \int_{\mathscr{D}} \bu\, \dr x =0\ (\mbox{periodic})\;\; \mbox{or}\;\; u|_{\partial \mathscr{D}}=0\ (\mbox{Dirichlet})\right\}
 \]
 with the inner product $(\bu,\bv) = \int_\mathscr{D}  \bu\cdot\bv = \int_\mathscr{D} \sum  u_i v_i \,\dr x$, where $u_i, v_i$ are the components of $\bu, \bv$ respectively.
 The Hilbert space $V$ is defined by
 $
 V=\{\bu \in H: \|\bu\|_V:=\|\nabla \bu\|_{L^2} < \infty\},
 $
 with respect the inner-product 
 $$
 \bvl \bu,\bv \bvr= \int_\mathscr{D} \nabla \bu \!::\!\nabla \bv\,\dr x
 = \int_\mathscr{D} \sum \partial_{j} u_i \partial_jv_i \,\dr x.
 $$
Henceforth,  following \cite{cf, Temambook1995}, we will denote
 \be  \label{frenchschool}
 \|\bv\|_H = \|\bv\|_{L^2}\;\; (\bv \in H)\quad  \textrm{and}\ 
 \|\bv\|_V=\|\nabla \bv\|_{L^2}\ (\bv \in V).
 \ee
 The \emph{Leray-Hopf} (orthogonal) projection operator is $\Pi:L^2(\mathscr{D}) \lra H$. For periodic boundary conditions the Stokes operator $A$ is defined to be $A=- \Pi\De \,|_{V \cap H^2(\mathscr{D})}$ which  is positive and self-adjoint with a compact inverse.  Denoting by $A^{1/2}$ the square root of $A$,  we have $D(A^{1/2})=V$ \cite{cf}. Due to the Poincar\'{e} inequality, $\|\bu\|_{\h^1} \sim \|\bu\|_V$ for $\bu\in V$.  Finally, applying the Leray projection,  one may express the NSE in its functional form \cite{Temambook1995}
 \begin{align}\label{dissip:eqn}
        \frac{\dr \bu}{\dr t}=-\nu A\bu-B(\bu,\bu)+\mbf{f}, \qquad 
        B(\bu,\bv)=\Pi (\bu \cdot \nabla)\bv.
    \end{align}
We  denote the eigenvalues of the Stokes operator by
 $0 < \lambda_1 < \lambda_2 < \cdots $ with $\lambda_n \ra \infty$ as $n\rightarrow \infty$, and the associated orthonormal eigenbasis $\{e_n\}_{n=1}^\infty$. In the space $\mathscr{D}$\,--periodic case, the eigenvalues are given by $\{\frac{4\pi^2}{\mathscr{D}^2}|\mbf n|^2\}_{\mbf n \in \mathbb{Z}^d}$. The \emph{spectral subspace} $H_N (\subset H)$ is the finite-dimensional space spanned by eigenvectors corresponding to eigenvalues $\lambda$ such that $\lambda \le \lambda_N$;
 the corresponding orthogonal projection onto $H_N$ is denoted by $P_N$.
 
It is well known that the bilinear operator $B(\cdot,\cdot): V\times V \rightarrow V^*$ (with $V^*$ the dual of $V$) satisfies a number of useful identities; namely
\begin{align}
\langle B(\bu,\bv),\bz\rangle_{V^*,V}\simeq (B(\bu,\bv),\bz) = -(B(\bu,\bz),\bv), \qquad (B(\bu,\bv),\bv) = 0, \qquad \forall\; \bu,\bv,\bz \in V 
\end{align}
and 
\begin{align}
|( B(\bu,\bv),\bz)|\leqslant \|\bu\|_{L^4}\|\bv\|_V\|\bz\|_{L^4}.
\end{align}
Furthermore, the two-dimensional Ladyzhenskaya interpolation inequality 
\begin{equation}\label{ladineq}
\|\bu\|^2_{L^4}\leqslant c_L\|\bu\|_H\|\bu\|_V
\end{equation}
implies that 
\begin{equation}
|( B(\bu,\bv),\bz)|\leq c_L \|\bu\|_{H}^{1/2} \|\bu\|^{1/2}_V \|\bv\|_V \|\bz\|_{H}^{1/2}\|\bz\|_{V}^{1/2}.
\end{equation} 
 It is well-known that a \emph{Leray-Hopf weak solution}  $\bu(t)$ of \eqref{dissip:eqn} exists on $[0,\infty)$ (see \cite{cf, Temambook1995} for definition) and it satisfies the (energy) inequalities
 \be  \label{Lenergyineq}
 \begin{array}{l}
 \frac12 \|\bu(t)\|_H^2 + \nu \int_{t_0}^t \|\bu(s)\|_V^2\, ds \le 
 \frac12 \|\bu(t_0)\|_H^2 + \int_{t_0}^t (f,\bu(s))\,ds,\quad  t \ge t_0
 \  ({\rm a.a.}\, t_0),\\[.5cm]
  \|\bu(t)\|_H^2 \le e^{-\nu \lambda_1t}\|\bu_0\|_H^2 + \frac{1}{\nu^2 \lambda_1^2}\|{\bf f}\|_H^2\lf (1- e^{-\nu \lambda_1t}\rg ), \quad t \ge 0.
  \end{array}
 \ee
The Leray-Hopf weak solutions $\bu(t)$ belong to  $L^\infty([0,T]; H) \cap L^2([0,T]; V) \cap C([0,T]; V^*)$,  and $\frac{du}{dt} \in L^1((0,T);V^*)$ for all $T>0$, where $V^*$ is the dual of $V$ (see \cite{cf, Temambook1995}). 
 A weak solution is said to be a {\em strong solution} if it also belongs to $L^\infty((0,T); V) \cap L^2((0,T); D(A))$. In two-dimensions (\ref{dissip:eqn}) is well-posed and it has a global attractor.  In particular, we have the following well-known results (see, e.g., \cite{cf, Temambook1995}):
 \begin{theorem}\label{NSE_thm_1}
 Let $\bu_0\in V$ and $f\in H$. Then (\ref{dissip:eqn}) has a unique strong solution that satisfies 
 \begin{equation}
 \bu(t)\in C\big([0,\,T];V\big)\cap L^2\big([0,\,T];D(A)\big) \qquad \textrm{for any} \quad T>0.
 \end{equation}
 Moreover, the solution $\bu(t)$ depends continuously on $\bu_0$ in the $V$ norm. 
 \end{theorem}
 \begin{theorem}\label{NSE_thm_2}
 Given the solution $\bu(t)$ of (\ref{dissip:eqn}) as in Theorem \ref{NSE_thm_1}, there exists a time $t_0$ which depends on $\bu_0$, such that for all $t\geqslant t_0$ we have 
 \begin{align}\label{uH_bnd}
 \|\bu(t)\|_H^2\leqslant 2\nu^2G^2 \qquad \textrm{and} \qquad \int_t^{t+T}\|\bu(s)\|_V^2\dr s\leqslant 2(1+T\nu\lambda_1)\nu G, 
 \end{align}
 where $G = \|{\bf f}\|_H/(\nu^2\lambda_1)$ is the {\it Grashof} number. Furthermore
 \begin{align}\label{uV_bnd}
 \|\bu(t)\|_V^2\leqslant 2\nu^2\lambda_1G^2 \qquad \textrm{and} \qquad \int_t^{t+T}\|A\bu(s)\|_H^2\dr s\leqslant 2(1+T\nu\lambda_1)\nu\lambda_1 G, 
 \end{align}
 and 
 \begin{align}\label{AuH_bnd}
 \|A\bu(t)\|_H^2\leqslant 2 c \hspace{.04cm}\nu^2\lambda_1^2(1+G)^4.
 \end{align}
 \end{theorem}

\section{Accuracy and Stability of Approximate Gaussian Filters}  \label{sec:unified}
As seen in Section \ref{cont_approx_G_list}, a variety of commonly used DA algorithms in continuous time outlined in Section \ref{cont_approx_G_list}, including 3DVar, Nudging, EnKF, and EnSRKF, lead to equation of the type as in \eqref{3dvareq}   with different choices of the covariance operator $C$. Thus, the analysis below  provides a unified framework for the study of  the continuous time limit  of these commonly used DA algorithms.  The {main objective}  is  to illuminate the interplay between the observation operator $\co$, the covariance operator ($C$ in \eqref{3dvareq}) and to clarify the role of the deterministic equation corresponding to \eqref{3dvareq}. Accordingly, we aim to provide appropriate conditions for stability and accuracy of the class of approximate Gaussian filters listed in \S\ref{cont_approx_G_list};  that is for a solution $\bm$ of \eqref{3dvareq} (or $\bm^{(k)}$ in (\ref{enkf_t}) or in (\ref{ensrkf_t})) with initial data $\bm_0$ and the solution $\bu(t)$ of (\ref{dissip:eqn}) we aim to establish estimates of the type

\be  \label{goodfilter}
\left.
 \begin{array}{l}
\text{(i) Filter Accuracy:}\ \limsup_{t \ra \infty} \E \|\bm (t) - \bu(t)\|^2  \lesssim\mathfrak{C}\hspace{.03cm}(\sigma^2),\\[.5cm]
\text{(ii) Filter Stability:}\ 
\E \|\bm_1(t) - \bm_2(t)\|^2 \lesssim 
e^{-kt}\E\|\bm_1(0)-\bm_2(0)\|^2\\[.3cm]
\qquad \qquad \qquad \qquad \qquad \qquad \qquad \qquad \qquad + \sup_t \E \|\co(\bu_1(t)-\bu_2(t))\|^2+\mathfrak{C}\hspace{.03cm}(\sigma^2)
\end{array}
\right\},
\ee
where $\|\cdot\|$ is a norm on an appropriate Hilbert space,  $\bm_i(\cdot)$, $i=1,2$ are two solutions of \eqref{3dvareq} with initial (distribution) data $\bm_i(0)$ and trajectories $\{\bu_i(t)\}$, $t\in\bbr$, and  $\mathfrak{C}\hspace{.03cm}(\sigma^2)\rightarrow 0$ as $\sigma\rightarrow 0$, where $\sigma>0$ denotes the amplitude of the observational noise.

\medskip
Our analysis for the 3DVar filter is carried out by first asserting the accuracy and stability of the 3DVar algorithm (\ref{3dvareq}) for a general covariance $C$ for Type\,-1  observation operators (see \S\ref{acc_3dvar}). The relationship to the nudging algorithm is outlined in the process of the analysis.  We then focus on the ensemble based algorithms (EnKF and EnSRKF in (\ref{enkf_t}) and (\ref{ensrkf_t})) whose accuracy and stability can be asserted by an appropriate manipulation of the state-dependent covariance (via the so-called `inflation' and `localization'; see \S\ref{intro} and/or \S\ref{acc_3dvar}) and reducing the problem to (technically) resemble the 3DVar configuration (see \S\ref{enkf_sec}).

The derived bounds on the accuracy of these DA algorithms are not tight and it is likely that there exist choices of the covariance $C$ that lead to more accurate filters.  The issues concerned with the use of machine learning aimed at  fine-tuning these algorithms and empirically improving  the accuracy by optimizing the choice of the covariance operator $C$ (state dependent in case of EnKF and EnSRKF) will be discussed in a separate publication. We note that it  has been shown recently such an approach can  perform remarkably well in computations \cite{PAS}.

\subsection{Accuracy and Stability of 3DVar filter and Covariance Inflation and Localization.}\label{acc_3dvar}
In this section we  consider the problem of stability and accuracy (\ref{goodfilter}) of the filter expressed via (\ref{3dvareq}) with a state-independent covariance $C$; i.e., the 3DVar filter outlined in \S\ref{cont_approx_G_list}. The accuracy and stability of the continuous-time 3DVar filter has been studied before with some added constraints. In particular, the rigorous analysis of 3DVar for the 2D NSE in \cite{BLSZ2013} corresponds to a choice  $C$ as   a diagonal state-independent operator (a fractional power of the Stokes operator), and the  observation operator  $\co$ is restricted to be a modal projection (see (ii) in \S\ref{sect:interpolantops}). In our general setting we will also address the issue of filter divergence and its control via {\it covariance inflation} and {\it covariance localization} (see \S\ref{secn:cov_localisation_inflation} for the definitions).

\medskip
Our goal is to prove that the solution $\bm(t)$  of \eqref{3dvareq} approximates the true solution $\bu(t)$ of~NSE~(\ref{dissip:eqn}) when $t\rightarrow \infty$ to within a tolerance determined by the observation noise/error for  covariance operators other than the diagonal ones. We show this by considering the dynamics of the error $\bee(t) = \bm(t)-\bu(t)$ which, based on  \eqref{dissip:eqn} and \eqref{3dvareq}, satisfies 
\be  \label{3dverr}
\dr \bee=\scaleobj{1.4}{[}-A\bee -(B(\bu, \bee)+B(\bee,\bu)-B(\bee,\bee))-\sigma^{-2}C\co^*\co \bee\scaleobj{1.4}{]} \dr t
+\sigma^{-1}C\co^*\dr \bW_t,
\ee
where  $ \bee(t)$  satisfies a bound of the form 
\be  \label{ee_bnds}
 \limsup_{t \ra \infty} \E \|\bee(t)\|^2  \lesssim \mathfrak{C}\hspace{.03cm}(\sigma^2).
\ee
Here $\|\cdot\|$ is a norm on an appropriate Hilbert space ($H$ or $V$ depending on the observation operator) and  $\mathfrak{C}\hspace{.03cm}(\sigma^2)\rightarrow 0$ as $\sigma\rightarrow 0^+$.

\begin{theorem}[Existence of 3DVar filter solutions]\label{3dvar_exist}
Let $\bu$ be the strong solution of (\ref{dissip:eqn}) in two-dimensions,  where $\bu_0\in V$ and ${\bf f}\in H$.  Assume that $\co : \;H\rightarrow L^2$ satisfies (\eqref{T1_PKp}), and 
\begin{align}  
&(C\I_h \mathcal{P}_{\K^\perp}\bm, \mathcal{P}_{\K^\perp}\bm) \ge \beta \|\mathcal{P}_{\K^\perp}\bm\|_H^2\ \qquad \;\forall\ \bm \in H, \qquad \mbox{where}\ \I_h = \co^*\co, \; \beta>0.\label{cov_cnd_exist}
\end{align}
Then, for any $\bm_0\in V$ and $T>0$, and $2{\hat c_2}^2\beta h^2/\nu\sigma^2$ sufficiently small  there exists a unique solution of (\ref{3dvareq}) in the sense that $\p${\rm -a.s.} ${\bm\in C\big([0,\,T];H\big)\cap L^2\big([0,\,T];V\big)}$ and 
\begin{align}
&(\bm(t),\varphi)+\int_0^t(\bm(s),A\varphi)\hspace{.03cm}{\rm d} s- \int_0^t (B(\bm(s),\varphi),\bm(s))\hspace{.03cm}{\rm d} s  = (\bu_0,\varphi)\notag\\
&\hspace{2.5cm} +\int_0^t({\bf f}(s),\varphi)\hspace{.03cm}{\rm d} s -\sigma^{-2}\int_0^t\big(C\co^*\co(\bm(s)-\bu(s)),\varphi\big)\hspace{.03cm}{\rm d} s+\sigma^{-1}\int_0^t(C\co^*{\rm d} \bW_s,\varphi)
\end{align}
for all $t\in [0,\,T]$ and for all $\varphi\in D(A)$. Moreover, 
\begin{align}\label{m_int}
\E\left(\sup_{0\leqslant t \leqslant T} \|\bm(t)\|_H^2+\nu\int_0^T\|\bm(t)\|_V^2\hspace{.03cm}{\rm d} t \right)<\infty.
\end{align}
\end{theorem}
\noindent {\it Proof}. The proof of this theorem is  similar to \cite{BOT} and \cite{flandoli94} and is based on a path-wise argument. The proof specific to our setup is presented in Appendix \ref{3dvar_exist_proof} in order to make the presentation relatively self-contained. 

\begin{remark}
Note that if $\I_h = \co^*\co$ is an orthogonal projection, then $P_{\K^\perp}=\I_h$ and the condition simplifies to 
 \begin{equation}
 (C\I_h \bm, \I_h\bm)\ge \beta \|\I_h\bm\|_H^2. 
 \end{equation}
\end{remark}

\medskip
Now we state the following two main results for the general 3DVar algorithm in continuous time with Type\,-1 observation operator. For simplicity of presentation and in order to illustrate the main ideas, we   start from a simpler setup in Theorem \ref{3dv_acc_1}), namely when  $\I=\co^*\co = P_{\K^\perp}$ is an orthogonal projection on $\K^\perp$ where $\K = \text{Kernel}\, \co$. This result is subsequently generalized in Corollaries \ref{3dv_acc_2} and \ref{3dv_acc_3}.

\begin{theorem}[Accuracy of filter solutions for  $\co^*\co = \I_h$ an orthogonal projection]\label{3dv_acc_1}
Let $\bu$ be the strong solution of (\ref{dissip:eqn})  where $\bu_0\in V$, ${\bf f}\in H$, and $\bm(t)$ satisfies (\ref{3dvar_exist}). Denote the error  in the 3DVar estimates  as $\bee(t) = \bm(t)-\bu(t)$.  Finally, assume that $\co : \;H\rightarrow L^2$ satisfies (\ref{interpolant:est_1}), 
$\I_h = \co^*\co=P_{\K^\perp}$ is an orthogonal projection on $\K^\perp$, where $ \K = {\rm Kernel}\, \co$ and $C$ is such that  
\begin{align}  
&(C\I_h\bee, \I_h\bee) \ge \beta \|\I_h\bee\|_H^2\qquad  \forall\;\bee \in H,
\label{cov_cnd_1}\\[.2cm]
& \frac{\nu}{4c_2^2h^2}\geqslant \frac{\beta}{\sigma^2} \ge \frac{c_1^2c_2^2h^2}{\sigma^4\nu}
\|P_\K CP_{\K^\perp}\|^2+c_L^2\nu\lambda_1G^2.  \label{betalocal}
\end{align}
Then, for any $\bu_0,\bm_0\in V$ and $T>0$ there exists a unique solution of (\ref{3dverr}) such that 
\begin{align} \label{avg_stochineq_thm_3dvar}
 \limsup_{t \ra \infty} \E \|\bee\|^2_H\leqslant \mathfrak{E}(\sigma^2):=\frac{1}{\gamma\sigma^{2}}\hspace{.04cm}{\rm Tr}(C\I_h C),
\end{align}
where
\begin{align}
\gamma &=  \frac{2\beta}{\sigma^2}- \frac{2c_1^2c_2^2h^2}{\sigma^4\nu}\|P_\K CP_{\K^\perp}\|^2
-2c_L^2\nu\lambda_1G^2. \label{gam_1}
\end{align}
Moreover, 
\begin{align}\label{enst_bnd_3dvar}
 \limsup_{t \ra \infty} \,\frac{\nu}{T}\int_t^{t+T}\|\bee(s)\|_V\dr s\leqslant \frac{1}{\kappa}\left(\frac{1}{\gamma \,T}+1\right)\frac{1}{\sigma^2}{\rm Tr}(C \I_h C),
\end{align}
where 
\begin{align}
\kappa = \frac{\nu}{2}-\frac{2\beta c_2^2h^2}{\sigma^2} .
\end{align}
\end{theorem}

\begin{remark}\rm The following observations are in order: 

\begin{itemize} [leftmargin = .7cm]
\item[(i)] Note that choosing $C=\beta \I_h=\beta P_{\K^{\perp}}$ satisfies \eqref{cov_cnd_1} and it also ensures that ${\rm Tr}(C\I_hC)=\beta^2 \sim  (\nu \lambda_1 G^2)^2 \sigma^4$ which in turn, due to $\gamma \gtrsim \sigma^2$ (see \eqref{gam_1}), implies that  $\mathfrak{E}(\sigma^2) \sim \sigma^4 \ra 0$ as $\sigma \ra 0$ in (\ref{avg_stochineq_thm_3dvar}) which is in line with (\ref{ee_bnds}). Other bespoke choices of $C$ might provide better accuracy. 

\item[(ii)] The conditions in Theorem \ref{3dv_acc_1} enforce a relationship between the \emph{richness of the observation space} which increases with decreasing $h$ and the  {\it large covariance property} in (\ref{cov_cnd_1}) controlled by $\beta$ which can be ensured by multiplicative covariance inflation if $C$ is strictly positive definite on $\text{Ran}\, \I_h= \K^\perp$, or else, by additive covariance inflation.  For 3DVar, under the setting of Theorem \ref{3dv_acc_1} where $\I_h$ is an orthogonal projection on $\K^\perp$, one can ensure \eqref{cov_cnd_1} by simply choosing $C=\mathfrak{c}I$ or more parsimoniously, $C=\mathfrak{c}P_{\K^\perp}$ for $\mathfrak{c}$ sufficiently large. The second choice leads to a localised covariance operator as discussed in (iii) below. Other less trivial choices of $C$ satisfying (\ref{cov_cnd_1}) are likely to lead to improved accuracy; machine learning procedures for learning the background covariance $C$ will be discussed in a separate publication.   
	
\item[(iii)] Formally, for $C=\mathfrak{c} I$ the bound (\ref{avg_stochineq_thm_3dvar}) is analogous to that derived in~\cite{BOT} in the setting of  `nudging'  in DA (\ref{nudg_W}). However, setting $C=\mathfrak{c} I$ is not always desirable (e.g., \cite{BM13}). 

\item[(iv)] When $\I_h$ is an orthogonal projection, i.e. $\I_h = \PKp$, employing {\it covariance localization}  allows one to obtain a tighter bound in (\ref{avg_stochineq_thm_3dvar}). This can be seen from the bound  (\ref{avg_stochineq_thm_3dvar}) which for $C\rightarrow \widetilde{C}=\PKp CP_{\K^\perp}$  holds with $\gamma =  \frac{2\beta}{\sigma^2}-2c_L^2\nu\lambda_1G^2$. The covariance localization, where the correlations between the observed and the unobserved components removed,  can be introduced a priori in the case of 3DVar and we refine and generalize this observation in Corollary \ref{3dv_acc_3}.   In the case of EnKF/EnSRKF, the localization procedure is more intricate but it proves crucial and provides ways to stabilise the algorithm.	

\end{itemize}
\end{remark}

The above setting leads to the following important observation for a much more general observation operator which will also be exploited later in the context of EnKF (see \S\ref{enkf_IhPK}):
\begin{theorem}[Accuracy of filter solutions for  $\co^*\co=\I_h$ of finite rank]\label{3dv_acc_2}
Let $\bu$ be the strong solution of (\ref{dissip:eqn})  where $\bu_0\in V$ and ${\bf f}\in H$ and assume that $\co : \;H\rightarrow L^2$ satisfies (\ref{interpolant:est_1}), $\I_h = \co^*\co$ has a finite rank, and  $C$ is such that 
\begin{align}  
&(C\I_h \mathcal{P}_{\K^\perp}\bee, \mathcal{P}_{\K^\perp}\bee) \ge \beta \|\mathcal{P}_{\K^\perp}\bee\|_H^2 \qquad \forall \;\bee \in H,\label{cov_cnd_2}\\[.2cm]
& \frac{\nu}{4c_2^2h^2}\geqslant \frac{\beta}{\sigma^2} \ge \frac{c_1^2c_2^2h^2}{\sigma^4\nu}
\|P_\K CP_{\K^\perp}\|^2+c_L^2\nu\lambda_1G^2.  \label{betacond_gen_loc}
\end{align}
Then, the assertions of Theorem \ref{3dv_acc_1} hold. 
\end{theorem}
\begin{remark}
\rm Note that if $\I_h$ is an orthogonal projection, then  $P_{\K^{\perp}} = \I_h$ and Theorem \ref{3dv_acc_2} reduces to Theorem \ref{3dv_acc_1}. Under the setting of Corollary \ref{3dv_acc_2}, i.e., when $\I_h=\co^*\co$ is of finite rank but not necessarily an orthogonal projection one can choose $C$ as follows in order to ensure~\eqref{cov_cnd_2}. Using the spectral theorem $\I_h$  can be represented as $\I_h = \sum_{i=1}^N \lambda_{i,\co} \,\psi_i \otimes \psi_i$, where $  \lambda_{1,\co} \ge \lambda_{2,\co} \ge \cdots \ge \lambda_{N,\co} >0$ are the strictly positive eigenvalues of $\I_h$ with corresponding orthonormal eigen-basis  $\{\psi_i\}_{i=1}^N$. We can then choose $C$  to be a localised diagonal matrix with respect to the eigen-basis of  $\co $ given by  $C= \sum_{i=1}^N \mu_i \psi_i \otimes \psi_i$ where $\mu_i\, \lambda_{i,\co} =\beta, i =1, \cdots , N$ with $\beta $ as in \eqref{betacond_gen_loc}.
\end{remark}
\begin{corollary}[Accuracy of filter solutions with covariance localization; $\co^*\co=\I_h$ finite rank]\label{3dv_acc_3}
Let the conditions of Theorem \ref{3dv_acc_2} hold and set $\widetilde C= P_{\K^\perp} C P_{\K^\perp}$ and assume that  
\begin{align}  
& \frac{\nu}{2c_2^2h^2}\geqslant \frac{\beta}{\sigma^2} \ge c_L^2\nu\lambda_1G^2. 
\end{align}
Then, for any $\bu_0,\bm_0\in V$ and $T>0$ there exists a unique solution of (\ref{3dverr}) such that 
\begin{align} \label{avg_stochineq_thm_3dvar_3}
 \limsup_{t \ra \infty} \E \|\bee\|^2_H\leqslant \mathfrak{E}(\sigma^2)=\frac{1}{\gamma\sigma^{2}}\hspace{.04cm}{\rm Tr}(C\I_h C), \qquad \gamma =  \frac{2\beta}{\sigma^2}- 2c_L^2\nu\lambda_1G^2.
\end{align}
Moreover, 
\begin{align}\label{enst_bnd_3dvar}
 \limsup_{t \ra \infty} \,\frac{\nu}{T}\int_t^{t+T}\|\bee(s)\|_V\dr s\leqslant \frac{1}{\kappa}\left(\frac{1}{\gamma \,T}+1\right)\frac{1}{\sigma^2}{\rm Tr}(C \I_h C), \qquad \kappa = \nu-2 c_2^2h^2/\sigma^2.
\end{align}
\end{corollary}

\smallskip
\begin{remark}\mbox{}\rm
\begin{itemize}[leftmargin = .8cm]
\item[(i)] Note that the covariance localization, $C\rightarrow  \widetilde{C}:=P_{\K^\perp} C P_{\K^\perp}$, in Corollary \ref{3dv_acc_3} allows one to recover the tighter bound from Theorem \ref{3dv_acc_1} which is obtained for the diagonal covariance and $\I_h$ an orthogonal projection. Thus, neglecting the cross-correlations between the observed and unobserved components of the state has the effect of improving the accuracy bound to the best possible one in our hierarchy. 

\item[(ii)]  Analogously to the more restrictive setup of Theorem \ref{3dv_acc_1},  a  choice of $C$ satisfying \eqref{cov_cnd_2}, e.g., $C\I_h=\beta P_{\K^{\perp}}$, ensures that ${\rm Tr}(C\I_hC)=\beta^2 \sim  (\nu \lambda_1 G^2)^2 \sigma^4$ which in turn, due to $\gamma \gtrsim \sigma^2$ (see~\eqref{gam_1}), implies that $\mathfrak{E}(\sigma^2) \sim \sigma^4 \ra 0$ as $\sigma \ra 0$ which is in line with (\ref{ee_bnds}).
\end{itemize}
\end{remark}

\begin{theorem}[Stability of filter solutions]\label{3dv_stab_1}
Assume that conditions of Theorem \ref{3dv_acc_1} hold and $\mathbf{m}^{(n)}(t)$ be two solutions of \eqref{3dvareq} with initial data $\mathbf{m}^{(n)}(0)$. Then, $$\limsup_{ t \ra \infty} \E \|\mbf{m}^{(1)} (t)-\mbf{m}^{(2)}(t)\|_H^2 \le 2 \mathfrak{E}(\sigma^2).$$
\end{theorem}
\begin{proof}
Follows immediately from Theorem \ref{3dv_acc_1} and the triangle inequality
\[
\|\mbf m^{(1)}- \mbf m^{(2)}\|_H \le \|\mbf m^{(1)} - \bu\|_H +  \|\mbf m^{(2)} - \bu\|_H
=\|\mbf e^{(1)}\|_H + \|\mbf e^{(2)}\|_H.
\]
\end{proof}

\subsubsection{Proofs }
The proofs of Theorems \ref{3dv_acc_1}, \ref{3dv_acc_2} and \ref{3dv_stab_1} are provided below. 

\medskip
\noindent{\it Proof of Theorem \ref{3dv_acc_1}.} First, note that given the dynamics in (\ref{3dverr}) and the fact that $\co$ is finite rank and $C$  is prescribed so that $\|C\co^*\|_\textsc{hs}<\infty$ and It\^o formula holds (Theorem 7.4 in~\cite{daprato92}). Then, by Ito's formula \cite{daprato92}, and noting that $C$ is a positive, self-adjoint operator, we have 
\[
\dr \|\bee\|_H^2=2(\bee, \dr\bee)+ \sigma^{-2} \,\textrm{Tr}(C\co^*\co C)\, \dr t.
\]
From equation \eqref{3dverr},  using  $(B(\bv, \bee), \bee)=0$ and $(A\bee , \bee)=\|\bee\|_V^2$,
we  obtain
\bea
\lefteqn{\dr \|\bee\|_H^2+ 2\nu \|\bee\|_V^2 \, \dr t= } \nn \\[.2cm]
 & & \scaleobj{1.4}{[}-2 (B(\bee,\bu),\bee)-2\sigma^{-2}(C\co^*\co \bee, \bee)
+ \sigma^{-2}\,\textrm{Tr}(C\co^*\co C)\scaleobj{1.4}{]}\, \dr t+ 2\sigma^{-1}(C\co^*\dr \bW_t,\bee).
\label{energystoch}
\eea
Note that in 2D, the Ladyzhenskaya inequality yields
\be  \label{lady2d}
-2|B(\bee,\bu),\bee)| \le 2c_L \|\bee\|_H\|\bee\|_V\|\bu\|_V 
\le \nu \|\bee\|_V^2 + \frac{c_L^2}{\nu}\|\bu\|_V^2\|\bee\|_H^2.
\ee
Next, recall that, given Theorems \ref{NSE_thm_1} and \ref{NSE_thm_2},  the solution $\bu(t)$ of (\ref{dissip:eqn})  satisfies the  uniform bound 
\be  \label{unifwbd}
\sup_{t\in [0,\,T]} \|\bu\|^2_V  = M^2_\bu < \infty, \quad {\rm} \quad \exists \, t_0\geqslant 0\quad {\rm s.t.} \quad M^2_\bu\leqslant 2\nu^2\lambda_1 G^2\quad {\rm for}\quad t\geqslant t_0.
\ee
Thus, using \eqref{energystoch}, \eqref{lady2d}, and (\ref{unifwbd}) we arrive at 
\be  \label{stochineq}
\dr  \|\bee\|_H^2 + \nu \|\bee\|_V^2\, \dr t \le
\scaleobj{1.4}{[}-2\sigma^{-2}(C\co^*\co \bee, \bee) +c_L^2\nu^{-1}M^2_\bu\|\bee\|_H^2
+\sigma^{-2}\text{Tr}(C\co^*\co C)\scaleobj{1.4}{]}\dr t+ 2\sigma^{-1}(C\co^*\dr \bW_t,\bee).
\ee
Special Case for Type\,-1 operator when  $\I_h=\co^*\co$ is an orthogonal projection. This simplified case already covers two very important classes of examples, namely when $\co$ is either a modal projection or a volume interpolant; the general case of a finite rank operator is discussed in the proof of theorem \ref{3dv_acc_2} below. For any $\|C\|<\infty$ we have 
\begin{align}\label{sigest_1}
-2\sigma^{-2}(C\co^*\co \bee, \bee)&= -2\sigma^{-2}(C\I_h \bee, \I_h\bee)-2\sigma^{-2}(C\I_h\bee, (I-\I_h)\bee)\notag \\[.1cm]
&  \le - 2\sigma^{-2}\beta \|\bee\|_H^2+ 2\sigma^{-2}\beta \|(I-\I_h)\bee\|_H^2 +2\sigma^{-2}\,\|P_\K CP_{\K^\perp}\|
\|\I_h\bee\|_H\|(I-\I_h)\bee\|_H\notag\\[.1cm]
&  \le - 2\sigma^{-2}\beta \|\bee\|_H^2+ 2\sigma^{-2}\beta c_2^2h^2\|\bee\|_V^2 +2\sigma^{-2}\,c_1c_2h\|P_\K CP_{\K^\perp}\|\|\bee\|_H\|\bee\|_V\notag\\[.1cm]
&  \le - 2\sigma^{-2}\beta \|\bee\|_H^2+ 2\sigma^{-2}\beta c_2^2h^2\|\bee\|_V^2 +\frac{2c_1^2c_2^2h^2\|P_\K CP_{\K^\perp}\|^2}{\sigma^{4}\nu}\|\bee\|^2_H+\frac{\nu}{2}\|\bee\|^2_V
\end{align} 
so that the conditions for accuracy (for $t\geq t_0$) are 
\be  \label{largebeta}
 \frac{\nu}{4c_2^2h^2}\geqslant \frac{\beta}{\sigma^2} \ge \frac{c_1^2c_2^2h^2}{\sigma^4\nu}\|P_\K CP_{\K^\perp}\|^2+c_L^2\nu\lambda_1G^2.
\ee
Note that \eqref{largebeta} highlights, in particular,  the {interplay} between the richness of the observation space ($\sim h^{-2}$) and the large covariance property (\ref{cov_cnd_1}).

Moreover,  note that when $C = \mathfrak{c}I$ we have $P_\K C P_{\K^\perp}=0$ and following steps as above, we obtain a simpler estimate 
\be \label{sigest_2}
-2\sigma^{-2}(\co^*\co \bee, \bee)
 \le - 2\sigma^{-2}\beta \|\bee\|_H^2+ 2\sigma^{-2}\beta c_2^2h^2\|\bee\|_V^2. 
 \ee

Inserting the above in \eqref{stochineq},  taking the expectation and noting that  ${\E \big(\int_0^t C\co^*\dr \bW_s,\bee\big)=0}$ (since $\bW_t$ is independent of the adapted process $\bee$), we get 
\be  \label{avg_stochineq_3dv}
\dt \E \|\bee\|_H^2 +\kappa\|\bee\|^2_V
\le - \gamma \hspace{.03cm}\E \|\bee\|_H^2+\sigma^{-2}\hspace{.04cm}\text{Tr}(C\co^*\co C), \quad \gamma>0.
\ee
where 
\begin{align}
\gamma =  \frac{2\beta}{\sigma^2}- \frac{2c_1^2c_2^2h^2}{\sigma^4\nu}\|P_\K CP_{\K^\perp}\|^2-\frac{c_L^2}{\nu}M^2_\bu, \qquad \kappa = \frac{\nu}{2}-\frac{2\beta c_2^2h^2}{\sigma^2}.
\end{align}
This readily yields 
\begin{align}\label{err_est_1}
 \limsup_{t \ra \infty} \,\E \|\bee(t)\|^2_H\leqslant \frac{1}{\gamma\hspace{.03cm}\sigma^{2}}\hspace{.04cm}{\rm Tr}(C\co^*\co C)
\end{align}
  via the Gronwall's lemma, where 
 \begin{align}
\gamma =  \frac{2\beta}{\sigma^2}- \frac{2c_1^2c_2^2h^2}{\sigma^4\nu}\|P_\K CP_{\K^\perp}\|^2-2c_L^2\nu\lambda_1G^2.
\end{align}
 
 \medskip
 In order to show (\ref{enst_bnd_3dvar}), we integrate (\ref{avg_stochineq_3dv}) from $t$ to $t+T$ which yields  
\begin{align}
&\E \|\bee(t+T)\|_H^2 +\kappa\int_t^{t+T}\|\bee(s)\|^2_V\,\dr s\notag\\
&\hspace{2cm}\le \E \|\bee(t)\|_H^2- \gamma \int_t^{t+T}\hspace{.03cm}\E \|\bee(s)\|_H^2\,\dr s+\sigma^{-2}\hspace{.04cm}T\,\text{Tr}(C\co^*\co C)\notag\\
&\hspace{2cm}\le \frac{1}{\gamma\sigma^{2}}\hspace{.04cm}{\rm Tr}(C\co^*\co C)- \gamma \int_t^{t+T}\hspace{.03cm}\E \|\bee(s)\|_H^2\dr s+\frac{1}{\sigma^{2}}\hspace{.04cm}T\,\text{Tr}(C\co^*\co C)\notag
\end{align}
Thus
\begin{align}
 \limsup_{t \ra \infty} \,\int_t^{t+T}\|\bee(s)\|^2_V\,\dr s\leqslant \frac{1}{\kappa}\left(\frac{1}{\gamma}+T\right)\frac{1}{\sigma^{2}}\text{Tr}(C\co^*\co C).
\end{align}

\qed

\medskip
\noindent {\it Proof of Theorem \ref{3dv_acc_2}}. In this more general setup, we  consider the case of a Type\,-1 observation operator $\co:H \ra \R^N$ such  that  $\I_h:=\co^*\co$ is a finite rank, bounded, non-negative, self-adjoint operator on~$H$ (but not necessarily an orthogonal projection). Therefore, using the spectral theorem, $\I_h$  can be represented as 
$\I_h = \sum_{i=1}^N \lambda_{i,\co} \,\psi_i \otimes \psi_i$, where
$  \lambda_{1,\co} \ge \lambda_{2,\co} \ge \cdots \ge
\lambda_{N,\co} >0$ are the strictly positive eigenvalues of $\I_h$ with corresponding orthonormal eigenbasis  $\{\psi_i\}_{i=1}^N$. Clearly,  
$\K:= \text{ker}\, \I_h=\text{ker}\, \co = \{ \text{span}\, \psi_i\}^\perp$ so that  $\{\psi_i\}_{i=1}^N$ is an orthonormal basis of $\K^\perp$. Then, we have 
\begin{align}\label{sigest_3}
\lefteqn{-2\sigma^{-2}(C\co^*\co \bee, \bee) } \nn \\
&\quad = -2\sigma^{-2}(C \I_h \bee, \bee)=-2\sigma^{-2}(C \I_h P_{\K^\perp}\bee, \bee)\notag\\[.1cm]
&\quad=-2\sigma^{-2}(C\I_hP_{\K^\perp} \bee, P_{\K^\perp}\bee)-2\sigma^{-2}(C\I_h P_{\K^\perp}\bee, (I-P_{\K^\perp})\bee)\notag \\[.1cm]
&\quad \le-2\sigma^{-2}\beta \|P_{\K^\perp} \bee\|^2_H-2\sigma^{-2}(C\I_h\bee, (I-P_{\K^\perp})\bee)\notag \\[.1cm]
& \quad \le - 2\sigma^{-2}\beta \|\bee\|_H^2+ 2\sigma^{-2}\beta \|(I-P_{\K^\perp})\bee\|_H^2 +2\sigma^{-2}\,\|C\|\|\I_h\bee\|_H\|(I-P_{\K^\perp})\bee\|_H\notag\\[.1cm]
&\quad  \le - 2\sigma^{-2}\beta \|\bee\|_H^2+ 2\sigma^{-2}\beta c_2^2h^2\|\bee\|_V^2 +2\sigma^{-2}\,c_1c_2h\|P_\K CP_{\K^\perp}\|\|\bee\|_H\|\bee\|_V\notag\\[.1cm]
& \quad \le - 2\sigma^{-2}\beta \|\bee\|_H^2+ 2\sigma^{-2}\beta c_2^2h^2\|\bee\|_V^2 +\frac{2c_1^2c_2^2h^2\|P_{\K}CP_{\K^\perp}\|^2}{\sigma^{4}\nu}\|\bee\|^2_H+\frac{\nu}{2}\|\bee\|^2_V.
\end{align} 
Next, one follows the same steps as in the proof of Theorem \ref{3dv_acc_1}. (Here, the constant $c_2$ arises from condition (\ref{interpolant:est_1}) associated with $P_{\K^\perp}$ not with $\I_h$ but we abuse the notation.)

\qed

\medskip
\noindent {\it Proof of Corollary \ref{3dv_acc_3}}. Note that $\widetilde C=P_{\K^\perp} C P_{\K^\perp}$ has the bloc operator form; namely 
\be  \label{cform}
\widetilde C = \left[ \begin{array}{rr}
                         C & 0\\
                         0 & 0
                         \end{array}
                         \right]
  \ee
 with respect to the decomposition $H= \K^\perp \oplus \,\K$.
 In this case, for all $\bee \in  H_0$, noting that $\co^*\co \hspace{.03cm}\bee \in \K^\perp$ and using \eqref{cform} and \eqref{cov_cnd_2},
 we have
 \begin{align}
 - (\widetilde C\co^*\co\bee,\bee) =  - (\tilde C\I_h\bee,\bee)&= - (C \I_h P_{\K^\perp}\bv,P_{\K^\perp}\bv) \le - \beta \|P_{\K^\perp}\bv\|^2_{H} \nn \\[.2cm]
 & = - \beta \|\bv\|^2_{H} + \beta \|(I-P_{\K^\perp}) \bv\|^2_{L^2}
 \le - \beta \|\bv\|^2_{H}  + \beta \epsilon_h^2 \|\bv\|_{H_0}^2,
 \label{positivity}
\end{align}
where $\|\bv\|_{H_0}=\|\bv\|_V, \;\bv \in V$ for NSE. Now we can just follow the same steps as in the proof of Theorem~\ref{3dv_acc_1}.

\qed

\subsection{Accuracy of EnKF and EnSRKF for modal and volume based observations.}\label{enkf_sec}
In contrast to  the case of 3DVar, where we considered the general case of a finite-rank Type\,-1 observation operator, here we confine our analysis to the case where $\co$ is of Type\,-1 and  $\co^*\co$ is an orthogonal projection. Recall  that {\em the modal and volume observation operators} $\co$ (\S\ref{sect:interpolantops}) already satisfy $\co^*\co=P_{\K^\perp}$. It turns out that dealing with a more general case of $\co^*\co$ given by an arbitrary finite-rank Type\,-1 or Type\,-2 operator requires considering correlated noise in the observations and we postpone this analysis to a subsequent publication. 

The most prominent examples of such  observation operators are modal and volume element observation described in Section \ref{sect:interpolantops}. This assumption, namely $\co^*\co=P_{\K^\perp}, \K =\text{kernel}\, \co$ will be enforced in Section \ref{enkf_IhPK}.

For EnKF the  equation that is analogous to (\ref{3dvareq}) for 3DVar is given by (see also (\ref{enkf_t}))
 \begin{align}\label{enkf_eq}
\dr \bmk = \big\{F(\bmk)- \sigma^{-2}\hspace{.04cm}C(\bm) \co^* \co (\bmk - \bu)\big\}\dr t+ \sigma^{-1}\hspace{.04cm}C(\bm)\co^* (\dr \bW^{(k)}_t+\dr {\bf B}_t), 
\end{align}
where  
\begin{equation}
\bm = \{\bmk\}_{k=1}^K, \quad \bmk(0)=\bmk_{0}, \quad  \bar\bm = {\textstyle \frac{1}{K}\sum_{k=1}^K \bmk}
\end{equation}
and $\bW_t^{(k)}, {\bf B}_t$ are pairwise independent Wiener processes in the observation space $H_\co$. 

\medskip
Similarly, the continuous-time dynamics of EnSRKF is given by (see (\ref{ensrkf_t}))
 \begin{align}  \label{ensrkf_eq}
\textstyle \dr \bmk = \Big\{F(\bmk)- \frac{1}{2}\sigma^{-2}\hspace{.04cm}C(\bm) \co^* \co \big( \bmk-\bu\big)-\frac{1}{2}\sigma^{-2}\hspace{.04cm}C(\bm) \co^* \co \big(\bar\bm-\bu\big) )\Big\}\dr t+ \sigma^{-1}\hspace{.04cm}C(\bm)\co^* \dr \bW_t,
 \end{align}
  where  ${\bf W}_t$ is the Wiener processes in the (finite dimensional) observation space $\cal H_\co$. 

Note that the state dependent covariance operator $C(\bm)$  is quadratic in $\bm$ and consequently, if $\bm \in V\ ~a.e.\; \P \times \text{dt}$, then $C(\bm)\, H \subset H\; a.e.\; \P \times \dr t$ due to the Sobolev inequality and the fact that $\|\bm\|_V \sim \|\bm\|_{\h^1}$.  Consequently,  equations (\ref{enkf_eq}) and (\ref{ensrkf_eq}) have a cubic nonlinearity and establishing the existence of global solutions requires some care. In order to study the global existence, uniqueness and regularity of solutions of these two  systems with $\bu(t)$ solving the  2D NSE, we need to consider the properties of the state-dependent covariance $C(\bm)=C(\{\bmk\}_{k=1}^K)$ which leads to the main complication of the problem as compared to 3DVar or  nudging DA. First,  based on standard techniques and $\bm_0\in V$, one can  assert the existence of local mild solutions on bounded sets $\mathcal{B}_{\bm_0}\subset V$ up to some stopping time  $\tau_{\bm_0}:=\big\{\inf\,t\in [0,\,T]:\; \|\bm(t)\|_V\notin\mathcal{B}_{\bm_0}\big\}$. This is obtained in a standard way by modifying Theorem 7.4 in \cite{daprato92} (existence and uniqueness of global solutions for Lipshitz coefficients) and  considering  (\ref{enkf_eq}) and (\ref{ensrkf_eq}) on  $\mathcal{B}_{\bm_0}\subset V$ so that the respective drift and diffusion operators are locally Lipshitz (and Lipshitz on $\mathcal{B}_{\bm_0}\times \tau_{\bm_0}$). Then we obtain that  $\E \big[\int_0^{\tau_{\bm_0}} \|C(\bm)\|_{\textsc{hs}}^2\, \dr t\big] = \E\big[ \int_0^{\tau_{\bm_0}}  \text{Tr}\, (C(\bm)^*C(\bm))\, \dr t\big]< \infty$ (\cite[Theorem 7.4]{daprato92} restricted to $\mathcal{B}_{\bm_0}\times \tau_{\bm_0}$), which implies that\footnote{\,Specifically, we show that $\frac{1}{K}\sum_{k=1}^K\E \|\bee^{(k)}(t)\|_H^2<\infty$, $t\in [0,\,T]$ (Theorems \ref{enkf_acc_thm} and \ref{ensrkf_acc_thm}) which implies  $\int_0^{T} \|\widetilde C(\bm)\|_{\textsc{hs}}^2\, \dr t <\infty \; a.e.\; \P$ based on (\ref{enkfcov_e}).}  $\int_0^{\tau_{\bm_0}} \|C(\bm)\|_{\textsc{hs}}^2\, \dr t <\infty \; a.e.\; \P$, so that It\^o's formula holds on $[0,\,\tau_{\bm_0}]$; see \cite[Theorem~4.17]{daprato92}.  Then, we show in the process of the  analysis of the dynamics of the error in the estimates $\bmk$ that, for appropriately localized and inflated covariance $C(\bm) \mapsto \widetilde C(\bm)$ in (\ref{enkf_eq}) and (\ref{ensrkf_eq}) one has  $\int_0^{T} \|\widetilde C(\bm)\|_{\textsc{hs}}^2\, \dr t <\infty \; a.e.\; \P$ for any $\|\bm_0\|\in V$ due to a very useful and non-trivial cancellation of terms in the estimates,   which allows one to assert the existence of global solutions $\bmk\in C\big([0,\,T];H\big)\cap L^2\big([0,\,T];V\big)$ by extending  the local solutions of (\ref{enkf_eq}) and (\ref{ensrkf_eq}) with $\widetilde C$  to the whole interval $[0,\, T]$ in a standard fashion (by noticing that the stopping time $\tau_{\bm_0} = T$ for any local solution starting at any $\|\bm_0\|\in V$).  It is worth noting, given the above discussion,  that there is a potential inconsistency in the results of \cite{KLS14} concerning  accuracy of EnKF for modal observations, since it is assumed there that the observation operator $\co$ and the noise covariance in the observations are of full rank (specifically, $\co = I$ and $\Gamma = \sigma^2 I$ on $H$) in which case $\bmk(t)\notin L^2(\mathscr{D})$ and the It\^o formula does not hold (\cite[Theorem~4.17]{daprato92}).  The importance of these facts becomes  clear upon inspection of Theorems \ref{enkf_acc_thm} and~\ref{ensrkf_acc_thm} below which are devoid of these issues. 
  
  \medskip
 First, we define error $\bek$ in the $k$-th ensemble member, and the average error $\beb$ as 
 $$\bek = \bmk-\bu, \qquad \beb = \overline\bm-\bu, \qquad \bee = \{\bek\}_{k=1}^K,$$
and note that $C(\bm) = C(\bee)$ since 
 \begin{align}  \label{enkfcov_e}
 C(\bm) =  \frac{1}{K} \sum_{k=1}^K (\bmk - \overline{\bm})\otimes (\bmk - \overline{\bm}) =  \frac{1}{K} \sum_{k=1}^K (\bek- \beb)\otimes (\bek - \beb)  =C(\bee).
 \end{align}
 Then, the evolution of the error $\bek$ in EnKF is obtained from (\ref{enkf_eq}) as 
\vspace{.1cm}\begin{align}\label{e_enkf}
\dr \bek &= \Big[-\nu A\bek-B(\bmk,\bmk)+B(\bu,\bu)\Big]\dr t- \sigma^{-2}C(\bee) \co^* \co\, \bek \dr t \notag\\ 
&\hspace{7.1cm}+\sigma^{-1}C(\bee)\mathcal{O}^*(\dr \bW^{(k)}_t+\dr {\bf B}_t),
\end{align}
where we used \eqref{enkfcov_e}. Analogously,  the evolution of the error $\bek$ in EnSRKF is obtained from (\ref{ensrkf_eq}) as 
{\rm 
\vspace{.1cm}\begin{align}\label{e_ensrkf}
\dr \bek &= \Big[-\nu A\bek-B(\bmk,\bmk)+B(\bu,\bu)\Big]\dr t- \textstyle{\frac{1}{2}}\sigma^{-2}\hspace{.04cm}C(\bee) \co^* \co \, \bek\dr t\notag\\[.2cm]
&\hspace{7.1cm}-\textstyle{\frac{1}{2}}\sigma^{-2}\hspace{.04cm}C(\bee) \co^* \co \,\beb\,\dr t\notag\\[.2cm] 
&\hspace{7.1cm}+\sigma^{-1}C(\bee)\mathcal{O}^*\dr \bW^{(k)}_t.
\end{align}}

 We now have the following results: 
 \begin{proposition}\label{enkf_e2_prop}
 Let $\bu$ be the strong solution of (\ref{dissip:eqn})  where $\bu_0\in V$, ${\bf f}\in H$, and let $\bmk(t)$ satisfy~(\ref{enkf_eq}) and  (\ref{e_enkf}).  Then there exists $T>0$ such that $\E \int_0^T \|C(\bee)\|_{\textsc{hs}}^2\, dt < \infty$ and the following holds 
{\rm
\begin{align}\label{e2_enkf}
\dr \|\bek\|^2_H +2\nu\|\bek\|^2_V\dr t =& -  2\big(B(\bek,\bu),\bek\big) \dr t \notag\\
&-2\sigma^{-2}\Big[\big(C(\bee)\co^*\co \bek,\bek\big)-\frac{1}{2}\textrm{Tr}\big(C(\bee)\co^*\co C(\bee)\big)\Big] \dr t \notag\\[.2cm]
& +2\sigma^{-1}\big(C(\bee)\mathcal{O}^*(\dr \bW^{(k)}_t+\dr {\bf B}_t),\bek\big),
\end{align}}
where $\|\bek\|^2_H = (\bek,\bek)$, $\|\bek\|^2_V = (A\bek,\bek) = (A^{1/2}\bek,A^{1/2}\bek)$.
\end{proposition}
\noindent{\it Proof.} 
Since $\E \int_0^T \|C(\bee)\|_{\textsc{hs}}^2\, dt < \infty$, It\^o's lemma is applicable  on $[0,\,T]$ (see Theorem 4.17 in \cite{daprato92})).
Thus, from It\^o lemma we have 
\begin{align}\label{e_enkf2}
\dr \|\bek\|^2_H = 2 (\bek, \dr \bek) +\sigma^{-2}\hspace{.04cm}\textrm{Tr}\big(C(\bee)\co^*\co C(\bee)\big) \dr t.
\end{align}
Using (\ref{e_enkf}) in (\ref{e_enkf2}) and noting that   $(B(\bu,\bek),\bek)=0$ and the  identity 
\begin{align}
B(\bmk,\bmk)-B(\bu,\bu) &= B(\bmk,\bmk)-B(\bmk,\bu)+B(\bmk,\bu)-B(\bu,\bu)\notag\\
& = B(\bmk,\bek)+B(\bek,\bu)\notag
\end{align}
leads to 
\begin{align}
\dr \|\bek\|^2_H +2\nu\|\bek\|^2_V\dr t =& - 2\big(B(\bek,\bu),\bek\big) \dr t \notag\\
&-2\sigma^{-2}\Big[\big(C(\bee)\co^*\co \bek,\bek\big)-\frac{1}{2}\textrm{Tr}\big(C(\bee)\co^*\co C(\bee)\big)\Big] \dr t \notag\\[.2cm]& +2\sigma^{-1}\big(C(\bee)\mathcal{O}^*(\dr \bW^{(k)}_t+\dr {\bf B}_t),\bek\big)
\end{align}
which, as claimed, yields (\ref{e2_enkf}). \qed

\begin{proposition}\label{ensrkf_e2_prop}
 Let $\bu$ be the strong solution of (\ref{dissip:eqn})  where $\bu_0\in V$, ${\bf f}\in H$, and let $\bmk(t)$ satisfy~(\ref{ensrkf_eq}) and  (\ref{e_ensrkf}).  Then there exists $T>0$ such that $\E \int_0^T \|C(\bm)\|_{\textsc{hs}}^2\, dt < \infty$ and the following holds
{\rm
\begin{align}\label{e2_ensrkf}
\dr \|\bek\|^2_H +2\nu\|\bek\|^2_V\dr t =& -  2\big(B(\bek,\bu),\bek\big) \dr t \notag\\
&-\sigma^{-2}\Big[\big(C(\bee)\co^*\co  \hspace{.04cm}\bek,\bek\big)-\textrm{Tr}\big(C(\bee)\co^*\co C(\bee)\big)\Big] \dr t \notag\\[.2cm]
&-\sigma^{-2}\big(C(\bee)\co^*\co \hspace{.04cm}\beb,\bek\big) \dr t \notag\\[.2cm]
& +2\sigma^{-1}\big(C(\bee)\mathcal{O}^*\dr \bW^{(k)}_t,\bek\big),
\end{align}}
where $\|\bek\|^2_H = (\bek,\bek)$, $\|\bek\|^2_V = (A\bek,\bek) = (A^{1/2}\bek,A^{1/2}\bek)$.
\end{proposition}
\noindent{\it Proof.} This follows trivially the proof of Theorem \ref{enkf_e2_prop}.

\medskip
Note that, in contrast to the 3DVar filter (see, e.g.,  Theorem \ref{3dv_acc_1}), the first term on the right-hand-side of (\ref{e2_enkf}) and (\ref{e2_ensrkf}) cannot be always controlled by the second term  
\begin{equation}\label{cntrl_term}
\big(C(\bee)\co^*\co \bek,\bek\big)-\alpha\,\textrm{Tr}\big(C(\bee)\co^*\co C(\bee)\big), \qquad \alpha = \frac{1}{2}, \,1
\end{equation}
which is not is not necessarily positive-definite. Crucially, however, it turns out that an appropriate localization and an additive inflation of the state-dependent covariance $C$ in (\ref{e2_enkf}) and in (\ref{e2_ensrkf}) allows one to assure  the accuracy of the estimates in EnKF and EnSRKF, as discussed below. Here, we consider the case of a general Type\,-1 operator;  Type\,-2 operators require additional technical considerations and will be dealt with in a subsequent work.  The crucial technical observation that allows one to assert the accuracy is derived in Lemma \ref{lemma_0} below.
  
\medskip

\subsubsection{Covariance localization and inflation when $\co$ is Type\,-1 observation operator and ${\co^*\co = P_{\K^\perp}}$}\label{enkf_IhPK}\mbox{}

Let $\co:H \lra H_\co$ be a finite-rank (hence bounded) Type\,-1 observation operator and, as before, denote 
\[
\K = {\rm Kernel}\, \co= {\rm Kernel}\ (\co^*\co), \quad \K^\perp = {\rm Span}\ \{\psi_n\}_{n=1}^N\quad
 {\rm and}\quad \I_h=\co^*\co. 
 \]
Throughout the remainder of this section, we will assume that $\I$ is an orthogonal projection, i.e., 
\be \label{orthoproj}
\I_h = \co^*\co = \PKp\ \mbox{and}\  \{\psi_n\}_{n=1}^N \ \mbox{is an orthonormal basis of}\ \K^\perp.
\ee 
Recall  that {\em the modal and volume observation operators} $\co$ (\S\ref{sect:interpolantops}) already satisfy $\co^*\co=P_{\K^\perp}$.  
As mentioned earlier,  dealing with a more general case of $\co^*\co$ given by an arbitrary finite rank Type\,-1 or Type\,-2 operator requires considering correlated noise in the observations and we postpone this analysis to a subsequent publication.

Next, consider the {\it localised covariance operator} obtained by replacing  $C$ in~(\ref{e2_enkf}) and (\ref{e2_ensrkf})  with  $\widetilde C=P_{\K^\perp} C P_{\K^\perp}$ which is positive definite since, using \eqref{orthoproj}, we have
\begin{align}
(\widetilde C(\bee)\co^*\co \bv,\bv) = (\widetilde C(\bee)P_{\K^\perp} \bv,\bv) = ( C(\bee)P_{\K^\perp} \bv,P_{\K^\perp}\bv) \qquad \forall\;\bv\in H.
\end{align}
Note that, depending on the nature of the observation operator, the localization can take place either in the `physical' state space or in the spectral domain (\S\ref{sect:interpolantops}).  
In both cases we have 
\begin{align}\label{tldC}
\widetilde C(\bee)= P_{\K^\perp} C(\bee)P_{\K^\perp} =  \frac{1}{K}\sum_{k=1}^KP_{\K^\perp}(\bek-\beb\,)\otimes P_{\K^\perp}(\bek-\beb\,)=C(P_{\K^\perp}\bee) .
\end{align}

Given the localised covariance $\widetilde C$ (\ref{tldC}), we obtain the following intermediate but instructive result: 

 \begin{proposition}\label{enkf_e2_prop_2}
 Let $\bu$ be the strong solution of (\ref{dissip:eqn})  where $\bu_0\in V$, ${\bf f}\in H$, and let $\bmk(t)$ satisfy~(\ref{enkf_eq}) with the covariance $C$ replaced by $\widetilde C=P_{\K^\perp} C P_{\K^\perp}$.  Then
{
\rm \begin{align}\label{e_enkf33}
\frac{1}{K}\sum_{k=1}^K\dr  \|\bek\|_H^2 &+\frac{1}{K}\sum_{k=1}^K2\nu\|\bek\|_V^2\dr t +
\sigma^{-2}{\rm Tr}\big(\widetilde C(\bee)\co^*\co\widetilde  C(\bee)\big) \nn \\
 &\hspace{3cm}=-\frac{1}{K}\sum_{k=1}^K2\big(B(\bek,\bu),\bek\big) \dr t -2\sigma^{-2}\big( C(\bee)\PKp  \beb,\PKp\beb\big) \dr t\notag\\[.2cm]
&\hspace{3.3cm}+\frac{1}{K}\sum_{k=1}^K2\sigma^{-1}\left(\widetilde C(\bee)\mathcal{O}^*\big(\dr \bW^{(k)}_t+\dr {\bf B}_t\big),\bek\right). 
\end{align}}
\end{proposition}
The proof of Proposition \ref{enkf_e2_prop_2} is immediate given (\ref{e2_enkf}), with $C$ replaced by $\widetilde{C}$, and Proposition~\ref{lemma_0} given below.

\begin{proposition}\label{lemma_0}
The following holds for  $\co^*\co = P_{\K^\perp}$ and  $\widetilde C=P_{\K^\perp} C P_{\K^\perp}${\rm :} 
\begin{align}\label{avg_err_term}
\frac{1}{K}\sum_{k=1}^K\big(\widetilde C(\bee)\co^*\co \bek,\bek\big)-{\rm Tr}\big(\widetilde C(\bee)\co^*\co\widetilde  C(\bee)\big)=\big( \widetilde C(\bee)  \beb,\beb\big).
\end{align}
\end{proposition}
\noindent {\it Proof}: See Appendix \ref{lemma_0_app}.

\medskip
An analogous statement to that in Proposition \ref{enkf_e2_prop_2} holds for the EnSRKF algorithm in (\ref{ensrkf_eq}) which can be trivially verified; namely 
{
\rm \begin{align}\label{e_ensrkff33}
\frac{1}{K}\sum_{k=1}^K\dr  \|\bek\|_H^2 &+\frac{1}{K}\sum_{k=1}^K2\nu\|\bek\|_V^2\dr t  \nn \\
 &\hspace{2cm}=-\frac{1}{K}\sum_{k=1}^K2\big(B(\bek,\bu),\bek\big) \dr t -2\sigma^{-2}\big( C(\bee)\PKp  \beb,\PKp\beb\big) \dr t\notag\\[.2cm]
&\hspace{2.3cm}+\frac{1}{K}\sum_{k=1}^K2\sigma^{-1}\left(C(\bee)\mathcal{O}^*\dr \bW^{(k)}_t,\bek\right). 
\end{align}}
Note that the evolution in (\ref{e_enkf33}) and (\ref{e_ensrkff33}) still does not provide the control of the evolution of $\frac{1}{K}\sum_{k=1}^K\|\bek\|_H^2$. Thus, in order to achieve the desired control of the dynamics in (\ref{e2_enkf}) and (\ref{e2_ensrkf}), we consider both the localization and an additive covariance inflation of $C(\bee)$  in the form 
\begin{align}\label{Cmu}
C_\mu(\bee) = \frac{1}{4}\PKp C(\bee)\PKp+\mu\PKp =\frac{1}{4}\widetilde C(\bee)+\mu\PKp,
\end{align}
where $\mu$ is the additive inflation parameter. It turns out that localization followed by {\em multiplicative} inflation of the covariance, i.e., $C'_\mu = \mu\hspace{.05cm}\widetilde C$, cannot in general enforce the desired control of the dynamics in (\ref{e2_enkf}) and (\ref{e2_ensrkf});  see Remark \ref{enkf_rm1} below. The additive covariance inflation (\ref{Cmu}) leads to the following:
\begin{theorem}[Accuracy of EnKF]\label{enkf_acc_thm}
Let $\bu$ be the strong solution of (\ref{dissip:eqn}), where $\bu_0\in V$ and $f\in H$ and assume that $\co : \;H\rightarrow H_\co$ satisfies (\ref{T1_PKp}) and $\I = \co^*\co=P_{\K^\perp}$ and let $C_\mu$ as in \eqref{Cmu}. Assume that
\begin{align}\label{kappagammacond_enkf}
\frac{3}{2}\frac{\mu  \hspace{.05cm} c_2^2 \hspace{.05cm}\epsilon^2}{\sigma^2}\leqslant \nu \quad and \quad \mu\geqslant \frac{2}{3} \frac{\hspace{.05cm}c_L^2 \hspace{.05cm}\sigma^2}{\nu}M_{\bu}^2,
\end{align}
where $M_\bu=\sup_t \|\bu\|_H< \infty$ and $\epsilon$ and $c_2$ are as in \eqref{interpolant:est_1} and (\ref{T1_PKp}), and $c_L$ as in (\ref{ladineq}). 
Then, for any $\bu_0,\bm_0\in V$ and $T>0$ there exists a unique solution of (\ref{e_enkf}) satisfying
\be \label{henkfbd}
\frac{1}{K}\sum_{k=1}^K\E \|\bee^{(k)}\|_H^2 \le e^{-\gamma t} \frac{1}{K}\sum_{k=1}^K \|\bee^{(k)}_0\|_H^2 + \frac{\mu^2}{\gamma\sigma^2}\hspace{.03cm} {\rm dim}(\mathcal{K}^\perp)
\ee
and
\be \label{enkftracebd}
 \frac{\kappa}{K}\sum_{k=1}^K\, \E\, \int_0^T\|\bee^{(k)}\|^2_V\, dt
+\frac{7}{16 \sigma^2} \E \int_0^T {\rm Tr}(\widetilde C(\bee)^2\,\PKp)\, dt
\le \frac{\mu^2}{\sigma^2}\,T\,{\rm dim}(\mathcal{K}^\perp),
\ee
where
\begin{align} 
\kappa = \nu-\frac{3}{2}\frac{\mu  \hspace{.05cm} c_2^2 \hspace{.05cm}\epsilon^2}{\sigma^2},
\hspace{.5cm} \gamma = \frac{3}{2}\frac{\mu}{\sigma^2}-\frac{c_L^2}{\nu}M_{\bu}^2.
\end{align}

\end{theorem}

 The accuracy result in Theorem \ref{enkf_acc_thm} is the first one that establishes the accuracy of EnKF, at least for the Navier-Stokes dynamics.  Earlier results of Kelly et al. \cite{KLS2014} proved well-posedness of EnKF for the Navier-Stokes dynamics but with a fully observed (infinite-dimensional) state, i.e., $ \K^\perp =H$.  Although  the statement of Theorem  \ref{enkf_acc_thm} necessitates  ${\rm dim}\ \K^\perp < \infty$, we can include cases where $\K^\perp$ is infinite-dimensional (and in particular the case $\K^\perp=H$ as in \cite{KLS2014}) by requiring that the covariance matrix observational noise  is trace class and the noise process $W$ is a (cylindrical) $Q-$Wiener process.

\begin{remark}\rm
 Note that \eqref{henkfbd} is of the form \eqref{goodfilter} with $\mathfrak{C}(\sigma^2)=\frac{\mu^2}{\gamma \sigma^2}\textrm{dim}\ \K^\perp$.  
We want to show that one can choose parameters such that $\mathfrak{C}(\sigma^2) \ra 0$ as $\sigma \ra 0$. In case of Type 1 interpolants (modal and volume), it is known \cite{AOT} that
$\textrm{dim}\ \K^\perp \sim \frac{1}{\epsilon^2}$. Choose $\frac{\mu}{\sigma^2}$ sufficiently large such that $\gamma \ge \frac{\mu}{\sigma^2}$. Now choose $\mu=\sigma^{3/2}$. Consulting \eqref{henkfbd} and \eqref{kappagammacond_enkf}, it readily follows that $\mathfrak{C}(\sigma^2) \sim \sigma \ra 0$ as $\sigma \ra 0$. Similar, procedure applies to the case of EnSRKF considered below. 
\end{remark}

The result below concerned with the accuracy of EnSRKF has not been treated before elsewhere, but it turns out that the bounds on the accuracy are essentially the same; namely we have the following: 

\begin{theorem}[Accuracy of EnSRKF]\label{ensrkf_acc_thm}
Let $\bu$ be the strong solution of (\ref{dissip:eqn}), where $\bu_0\in V$ and $f\in H$ and assume that $\co : \;H\rightarrow H_\co$ satisfies (\ref{T1_PKp}) and $\I = \co^*\co=P_{\K^\perp}$ and let $C_\mu$ as in \eqref{Cmu}. Assume that
\begin{align}\label{kappagammacond_ensrkf}
\frac{1}{2}\frac{\mu c_2^2\epsilon^2}{\sigma^2}\leqslant \nu \quad and \quad \mu\geqslant  \frac{2c_L^2\sigma^2}{\nu}M_{\bu}^2,
\end{align}
where $M_\bu=\sup_t \|\bu\|_H< \infty$, $\epsilon$ and $c_2$  are as in \eqref{interpolant:est_1} and (\ref{T1_PKp}), and $c_L$ as in (\ref{ladineq}). 
Then, for any $\bu_0,\bm_0\in V$ and $T>0$ there exists a unique solution of (\ref{e_ensrkf}) satisfying
\be \label{hensrkfbd}
\frac{1}{K}\sum_{k=1}^K\E \|\bee^{(k)}\|_H^2 \le e^{-\gamma t} \frac{1}{K}\sum_{k=1}^K \|\bee^{(k)}_0\|_H^2 + \frac{\mu^2}{\gamma \sigma^2}\hspace{.03cm} {\rm dim}(\mathcal{K}^\perp)
\ee
and
\be \label{ensrkftracebd}
\frac{\kappa}{K}\sum_{k=1}^K\, \E\, \int_0^T\|\bee^{(k)}\|^2_V\, dt
+\frac{3}{16 \sigma^2} \E \int_0^T {\rm Tr}(\widetilde C(\bee)^2\,\PKp)\, dt
\le \frac{\mu^2}{\sigma^2}\,T\,{\rm dim}(\mathcal{K}^\perp),
\ee
where 
\begin{align}
 \kappa = \nu-\frac{1}{2}\frac{\mu  \hspace{.05cm} c_2^2 \hspace{.05cm}\epsilon^2}{\sigma^2},
\hspace{.5cm} \gamma = \frac{1}{2}\frac{\mu}{\sigma^2}-\frac{c_L^2}{\nu}M_{\bu}^2.
\end{align}

\end{theorem}

\bigskip
\begin{center}
\textsc{Proofs of Theorems ( \ref{enkf_acc_thm}) and (\ref{ensrkf_acc_thm}) }
\end{center}
\noindent {\it Proof of Theorem \ref{enkf_acc_thm}}.  Consider $C_\mu$ given in (\ref{Cmu}) and  set $C\mapsto C_\mu$ in  the evolution of the error of EnKF given by  (\ref{e2_enkf}). Recall that $\I=\co^*\co=P_{\K^\perp}$ due to our assumption.  Note that
\be \label{tr_add_infl_1}
(C_\mu(\bee)\,\co^*\co \bek,\bek)= (C_\mu(\bee)\,\PKp \bek,\bek)=\frac{1}{4}(\widetilde C(\bee)\,\PKp \bek,\bek)+\mu\|\PKp\bek\|_H^2
\ee
and moreover,
\begin{align}
\textrm{Tr}\big(C_\mu\,\co^*\co C_\mu\big)&= \textrm{Tr}\big(C_\mu\, P_{\K^\perp} C_\mu\big)
=\textrm{Tr}(C^2_\mu\,\PKp)\notag\\[.2cm]
&  = \frac{1}{16}\textrm{Tr}(\widetilde C^2\,\PKp )+\mu^2 \hspace{.05cm}\textrm{Tr}(\PKp)+\frac{1}{2}\mu \hspace{.04cm} \textrm{Tr}(\widetilde C\,\PKp), \label{tr_add_infl}
\end{align}
where $C_\mu = C_\mu(\bee)$ is defined in (\ref{Cmu}), $\widetilde C = \widetilde C(\bee)$, and  $\textrm{Tr}(\PKp) = \textrm{dim}(\mathcal{K}^\perp)$. Using (\ref{tldC}) and the fact that ${\rm Tr}(\bu \otimes \bu)= \|\bu\|_H^2\quad  \forall\;  \bu \in H$, and noting  that $\widetilde{C} \PKp =\widetilde{C}$, we obtain
\begin{align} \label{trcompute}
 \textrm{Tr}(\widetilde C \PKp )=\textrm{Tr}(\widetilde C) &= \frac{1}{K}\sum_{k=1}^K\|\PKp(\bek-\beb)\|_H^2 \notag\\
& = \frac{1}{K}\sum_{k=1}^K\left(\|\PKp\bek\|_H^2-2(\PKp\bek,\beb)+\|\PKp \beb\|_H^2\right)\notag\\
 &= \frac{1}{K}\sum_{k=1}^K\|\PKp\bek\|_H^2-\|\PKp \beb\|_H^2.
\end{align}
  The evolution of  $\bek$ in  (\ref{e2_enkf}) with $C\mapsto C_\mu$ and $\co^*\co = \PKp$ is given by 
\begin{align*}
\dr \|\bek\|_H^2 +2\nu\|\bek\|^2_V\dr t &=  -2(B(\bek,\bu),\bek) \dr t -2\sigma^{-2}(C_\mu(\bee)\PKp \bek,\bek)\dr t\notag\\[.1cm]&\qquad +\sigma^{-2}\textrm{Tr}(C_\mu(\bee)\PKp C_\mu(\bee)) 
\dr t\\[.2cm]
&\qquad +2\sigma^{-1}\big(C_\mu(\bee)\mathcal{O}^*(\dr \bW^{(k)}_t+\dr {\bf B}_t),\bek\big).
\end{align*}
Then, using \eqref{tr_add_infl_1}. \eqref{tr_add_infl},  \eqref{trcompute}, and after dropping the last (non-positive) term on the right hand side of \eqref{trcompute}, we get
\begin{align} \label{e2_enkf_1}
\dr \|\bek\|_H^2 +2\nu\|\bek\|^2_V\dr t &\leq -2(B(\bek,\bu),\bek) \dr t -\frac{1}{2\sigma^2}\left((\widetilde C(\bee)\,\PKp \bek,\bek) -\frac{1}{8}\textrm{Tr}(\widetilde C^2(\bee)\,\PKp)\right) \dr t\notag\\[.1cm]&
\qquad +\frac{1}{\sigma^2}\left(\mu^2\textrm{dim}(\mathcal{K}^\perp)+\frac{\mu}{2K}\sum_{k=1}^K\|\PKp\bek\|_H^2 - 2\mu\|\PKp\bek\|_H^2\right)\dr t\notag\\[.2cm]
&\qquad +\frac{2}{\sigma}\big(C_\mu(\bee)\mathcal{O}^*(\dr \bW^{(k)}_t+\dr {\bf B}_t),\bek\big).
\end{align}
Thus,
\begin{align} \label{e_enkf333}
\dr \|\bek\|_H^2 &+2\nu\|\bek\|^2_V\dr t \, + \frac{7}{16 \sigma^2} \textrm{Tr}(\widetilde C(\bee)^2\,\PKp)  \nn \\[.1cm]
&\hspace{1cm}\leq-2(B(\bek,\bu),\bek) \dr t -\frac{1}{2\sigma^2}\left((\widetilde C(\bee)\,\PKp \bek,\bek) -\textrm{Tr}(\widetilde C(\bee)^2\,\PKp)\right) \dr t\notag\\[.1cm]&
\hspace{1.5cm} +\frac{1}{\sigma^2}\left(\mu^2\textrm{dim}(\mathcal{K}^\perp)+\frac{\mu}{2K}\sum_{k=1}^K\|\PKp\bek\|_H^2 - 2\mu\|\PKp\bek\|_H^2\right)\dr t\notag\\
&\hspace{1.5cm} +\frac{2}{\sigma}\big(C_\mu(\bee)\mathcal{O}^*(\dr \bW^{(k)}_t+\dr {\bf B}_t),\bek\big).
\end{align}
Finally, we have from (\ref{e_enkf333}) and  Proposition \ref{lemma_0} that 
\begin{align}\label{e_enkf_K}
\frac{1}{K}\sum_{k=1}^K\dr \|\bek\|^2 &+2\nu\frac{1}{K}\sum_{k=1}^K\|\bek\|_V^2\dr t +\frac{7}{16 \sigma^2} \textrm{Tr}(\widetilde C(\bee)^2\,\PKp) \dr t\nn \\
&\hspace{2cm}\leq-\frac{1}{K}\sum_{k=1}^K2(B(\bek,\bu),\bek)\dr t -\frac{1}{2\sigma^2}( C(\bee)\PKp  \,\beb,\PKp\beb)\dr t\notag\\
&\hspace{2.4cm} +\frac{1}{\sigma^2}\left(\mu^2\textrm{dim}(\mathcal{K}^\perp)-\mu\,\frac{3}{2}\frac{1}{K}\sum_{k=1}^K\|\PKp\bek\|_H^2 \right)\dr t\notag\\[.2cm]
&\hspace{2.4cm}+\frac{2}{\sigma}\frac{1}{K}\sum_{k=1}^K\big(C_\mu(\bee)\mathcal{O}^*(\dr \bW^{(k)}_t+\dr {\bf B}_t),\bek\big).
\end{align}
where the last equality is due to Proposition \ref{lemma_0}.

\noindent Given the form of (\ref{e_enkf_K}), we proceed with the proof in an analogous fashion to the proof of accuracy of the 3DVar filter in Theorem \ref{3dv_acc_1} using a large value of $\mu$ so that the damping term $-\mu\sigma^{-2}\frac{3}{2K}\sum_{k=1}^K\|\PKp\bek\|_H^2$ controls the nonlinear term $-2\frac{1}{K}\sum_{k=1}^K(B(\bek,\bu),\bek)$.

\medskip
From the properties of $B$ and the Ladyzhenskaya inequality we have 
\be  
-2|B(\bek,\bu),\bek)| \le 2c_L \|\bek\|_H\|\bek\|_V\|\bu\|_V 
\le \nu \|\bek\|_V^2 + \frac{c_L^2}{\nu}\|\bu\|_V^2\|\bek\|_H^2
\ee
and the second term in (\ref{e_enkf_K}) is negative-definite. Thus we have 
\begin{align}\label{e_enkf_K2}
\frac{1}{K}\sum_{k=1}^K\dr \|\bek\|^2 &+\nu\frac{1}{K}\sum_{k=1}^K\|\bek\|_V^2\dr t +\frac{7}{16 \sigma^2} \textrm{Tr}(\widetilde C(\bee)^2\,\PKp) \dr t\nn \\
&\hspace{2cm}\leq\frac{1}{K}\sum_{k=1}^K\left(\frac{c_L^2}{\nu}\|\bu\|_V^2\|\bek\|_H^2-\frac{3}{2}\frac{\mu}{\sigma^2} \|\PKp\bek\|_H^2\right)\dr t \notag\\[.1cm]
&\hspace{2.4cm} +\frac{1}{2\sigma^2}\left(2\mu^2\textrm{dim}(\mathcal{K}^\perp)- (C(\bee)\PKp  \,\beb,\PKp\beb)  \right)\dr t\notag\\[.1cm]
&\hspace{2.4cm}+\frac{2}{\sigma}\frac{1}{K}\sum_{k=1}^K\big(C_\mu(\bee)\mathcal{O}^*(\dr \bW^{(k)}_t+\dr {\bf B}_t),\bek\big)\notag%\\[.1cm]
\end{align}
\begin{align}
&\hspace{3.cm}\leq  \frac{1}{K}\sum_{k=1}^K \Big(\frac{c_L^2}{\nu}\|\bu\|_V^2\|\bek\|_H^2-\frac{3}{2}\frac{\mu}{\sigma^2} \|\bek\|_H^2 \notag\\[.1cm]
&\hspace{2.4cm}\hspace{2.6cm}+\frac{3}{2}\frac{\mu}{\sigma^2} \|(I-P_{\K^\perp})\bek\|_H^2\Big)\dr t\notag\\[.1cm]
&\hspace{3.5cm} +\frac{1}{2\sigma^2}\left(2\mu^2\textrm{dim}(\mathcal{K}^\perp)- (C(\bee)\PKp  \,\beb,\PKp\beb)  \right)\dr t\notag\\[.2cm]
&\hspace{3.5cm}+\frac{2}{\sigma}\frac{1}{K}\sum_{k=1}^K\big(C_\mu(\bee)\mathcal{O}^*(\dr \bW^{(k)}_t+\dr {\bf B}_t),\bek\big)\notag\\[.1cm]
&\hspace{3.cm}\leq  \frac{1}{K}\sum_{k=1}^K \Big(\frac{c_L^2}{\nu}\|\bu\|_V^2- \frac{3}{2}\frac{\mu}{\sigma^2}\Big)\|\bek\|_H^2\dr t \notag\\[.1cm]
&\hspace{3.5cm}+ \frac{3}{2}\frac{\mu \hspace{.05cm}c_2^2 \hspace{.05cm}\epsilon^2}{\sigma^2} \frac{1}{K}\sum_{k=1}^K \|\bek\|_V^2\dr t\notag\\[.1cm]
&\hspace{3.5cm} +\frac{1}{2\sigma^2}\left(2\mu^2\textrm{dim}(\mathcal{K}^\perp)- (C(\bee)\PKp  \,\beb,\PKp\beb)  \right)\dr t\notag\\[.2cm]
&\hspace{3.5cm}+\frac{2}{\sigma}\frac{1}{K}\sum_{k=1}^K\big(C_\mu(\bee)\mathcal{O}^*(\dr \bW^{(k)}_t+\dr {\bf B}_t),\bek\big),
\end{align}
where we used \eqref{T1_PKp}  to obtain the last inequality. Taking the expectation in (\ref{e_enkf_K2}) and noting that  ${\E \big(\int_0^t C_\mu\co^*\dr \bW_s,\bee\big)=0}$ (since $\bW_t$ is independent of the adapted process $\bee(t)$), we get 
\begin{align}  \label{avg_stochineq}
\frac{1}{K}\sum_{k=1}^K\dt \E \|\bee^{(k)}\|_H^2 &+\kappa\frac{1}{K}\sum_{k=1}^K\, \E\, \|\bee^{(k)}\|^2_V +\frac{7}{16 \sigma^2} \textrm{Tr}(\widetilde C(\bee)^2\,\PKp)\nn \\
&\hspace{2cm} \le - \gamma \hspace{.03cm}\frac{1}{K}\sum_{k=1}^K\E \|\bee^{(k)}\|_H^2+\frac{\mu^2}{\sigma^2}\textrm{dim}(\mathcal{K}^\perp),
\end{align}
where
\begin{align}
 \kappa = \nu-\frac{3}{2}\frac{\mu  \hspace{.05cm} c_2^2 \hspace{.05cm}\epsilon^2}{\sigma^2},
\hspace{.5cm} \gamma = \frac{3}{2}\frac{\mu}{\sigma^2}-\frac{c_L^2}{\nu}M_{\bu}^2.
\end{align}
and  $M_\bu=\sup_t \|\bu\|< \infty$.
Thus, provided \eqref{kappagammacond_enkf} holds, 
by applying Gronwall inequality, we obtain \eqref{henkfbd}.
Now integrating  both sides of \eqref{avg_stochineq}, and by dropping negative terms on the right hand side, we obtain \eqref{enkftracebd}. 

\qed

\medskip
\noindent {\it Proof of Theorem \ref{ensrkf_acc_thm}}. The proof follows closely the steps of the proof of Theorem \ref{enkf_acc_thm}, except that equations (\ref{e2_enkf_1}), (\ref{e_enkf333}), (\ref{e_enkf_K})  are replaced by 
\begin{align}\label{e2_ensrkf_1}
& \dr \|\bek\|_H^2 +2\nu\|\bek\|^2_V\dr t =  -2(B(\bek,\bu),\bek) \dr t -\sigma^{-2}(C_\mu(\bee)\PKp \bek,\bek)\dr t\notag\\[.1cm]
&\quad +\sigma^{-2}\textrm{Tr}(C_\mu(\bee)\PKp C_\mu(\bee)) \dr t
-\sigma^{-2}\big(C_\mu(\bee)\PKp \hspace{.04cm}\beb,\bek\big) \dr t \notag
+2\sigma^{-1}\big(C_\mu(\bee)\mathcal{O}^*\dr \bW^{(k)}_t,\bek\big)\notag\\[.2cm]
&\leq-2(B(\bek,\bu),\bek) \dr t 
-\frac{1}{4\sigma^2}\left((\widetilde C(\bee)\,\PKp \hspace{.03cm}\bek,\bek) +\widetilde C(\bee)\,\PKp \hspace{.03cm}\beb,\bek)-\frac{1}{4}\textrm{Tr}(\widetilde C^2(\bee)\,\PKp)\right) \dr t\notag\\[.1cm]
&\qquad +\frac{1}{\sigma^2}\Big(\mu^2\textrm{dim}(\mathcal{K}^\perp)+\frac{\mu}{2K}\sum_{k=1}^K\|\PKp\bek\|_H^2 - \mu\|\PKp\bek\|_H^2 -\mu\big(\PKp \beb,\PKp\bek\big) \notag \\[.2cm]
&\qquad +\frac{2}{\sigma}\big(C_\mu(\bee)\mathcal{O}^*\dr \bW^{(k)}_t,\bek\big).
\end{align}
Thus
\begin{align}\label{e2_ensrkf_2}
& \frac{1}{K}\sum_{k=1}^K\dr \|\bek\|_H^2 +2\nu\frac{1}{K}\sum_{k=1}^K\|\bek\|^2_V\dr t  +\frac{3}{16\sigma^2}\textrm{Tr}(\widetilde C^2(\bee)\,\PKp) \dr t  \notag\\
&\hspace{1cm}\leqslant -2\frac{1}{K}\sum_{k=1}^K(B(\bek,\bu),\bek) \dr t  -\frac{1}{2\sigma^2}(\widetilde C(\bee)\,\PKp \beb,\PKp\beb) \dr t\notag\\[.1cm]
&\hspace{1cm}+\frac{1}{\sigma^2}\left(\mu^2\textrm{dim}(\mathcal{K}^\perp)-\mu\frac{1}{2}\frac{1}{K}\sum_{k=1}^K\|\PKp\bek\|_H^2 \right)\dr t
+\frac{2}{\sigma}\frac{1}{K}\sum_{k=1}^K\big(C_\mu(\bee)\mathcal{O}^*(\dr \bW^{(k)}_t+\dr {\bf B}_t),\bek\big).
\end{align}
Note that the inequality in (\ref{e2_ensrkf_2}) is essentially the same  (with a few  different coefficients but no sign changes) as that in (\ref{e_enkf_K}) for $\mathbf{B}_t\equiv 0$. 
Thus, the rest of the proof is essentially the same as that of Theorem \ref{enkf_acc_thm}. 

\qed

\begin{remark}\label{enkf_rm1}\rm \mbox{}
\begin{itemize}[leftmargin=0.7cm]
\item[(i)] Note that the additive covariance inflation in the form as in (\ref{Cmu}) is necessary in the assertions of Theorems~\ref{enkf_acc_thm} and \ref{ensrkf_acc_thm}. Using the multiplicative covariance inflation and localization in (\ref{e2_enkf}) and (\ref{e2_ensrkf}), i.e., setting $C\mapsto C_\mu = \mu\PKp C \PKp$,  cannot guarantee the accuracy of EnKF  and EnSRKF since, in such a case, the trace term  becomes $\textrm{Tr}(C_\mu\,\co^*\co C_\mu) = \mu^2\textrm{Tr}(\widetilde C^2\,\co^*\co)$ while the damping term remains linear in $\mu$  and it cannot control $\mu^2$ even if $\widetilde C = \PKp C\PKp$ has a very low condition number. Moreover, rank($C$) may be much smaller than the dimension of the state $N$ (case of covariance collapse) in which  case the additive covariance inflation is indespensible.

\item[(ii)] Note also that for the additive covariance inflation we are essentially `throwing away' the state-dependent covariance matrix $C(\bee)$ in (\ref{enkf_e2_prop}) computed from the coupled ensemble, which effectively implies bounding the error in EnKF/EnSRKF estimates by the error obtained from DA using an ensemble of independent 3DVar filters. Thus, it is   interesting to note that the theoretical accuracy of EnKF/EnSRKF does not require the covariance estimates computed from the evolving ensemble of particles although, in practice,  retaining the state-dependent covariance is likely to improve the accuracy to a degree better than than that asserted in Theorem \ref{enkf_acc_thm} and Theorem \ref{ensrkf_acc_thm}. \end{itemize}
\end{remark}

\newpage
\appendix
\section{Formal derivation of continuous-time equations of the Ensemble Square-Root Kalman Filter (EnSRKF) (\ref{ensrkf_t}).}\label{ensqkf_app}
Here, we outline the formal derivation of  the continuous-time limit for the state estimate arising through the EnSRKF dynamics (\ref{ensrkf_t}) from the discrete-time setting (\ref{ensrkf_j_beg})-(\ref{ensrkf_j_end}). Similar to \cite{BLSZ2013}, we do not embark on rigorous derivations of this dynamics. In this general setting, assuming that $\widehat C_j$ arises as an approximation of a continuous process $\widehat C(t)$ evaluated at $t = jh$, we have  that $\widehat C_j = \widehat C(jh)$, and we assume that $h \ll 1$. Furthermore, we define the process generated by the observations $(\by_j)_{j\in \mathbb{Z}}$, $\by_j\in \cal H_\co$ via 
\begin{equation}
\pmb{\eta}_{j+1} = \pmb{\eta}_j+h\by_{j+1}.
\end{equation}
 Throughout we will assume that $\Gamma= h^{-1}\Gamma_0$, $\Gamma_0>0$. This scaling implies that the noise variance is inversely proportional to the time between observations and is the relationship which gives a nontrivial  limit as $h \rightarrow 0$.
With these assumptions the dynamics  (\ref{ensrkf_j_beg})-(\ref{ensrkf_j_end}) becomes
\begin{equation}\label{enkf_disc_h}
 \bmk_{j+1} = \bar{\hbm}_{j+1} +h\hspace{.03cm}\widehat C_{j+1}\mathcal{O}^*(h\mathcal{O} \widehat C_{j+1}\mathcal{O}^*+\Gamma_0)^{-1}\mathcal{O}^*\big(h^{-1}(\pmb{\eta}_{j+1}-\pmb{\eta}_j)-\mathcal{O} \,\bar{\hbm}_{j+1} \big)+S^{(k)}_{j+1}
\end{equation}
and 
\begin{align}
 S_{j+1} = \widehat S_{j+1} T_{j+1} = \widehat S_{j+1} (I-K^{-1}h\hspace{0.03cm}\widehat S_{j+1}^*\mathcal{O}^*(h\mathcal{O}\widehat C_{j+1}\mathcal{O}^*+\Gamma_0)^{-1}\mathcal{O}\widehat S_{j+1})^{1/2}.
 \end{align}
 Next, note that
 \begin{align}
& h\hspace{.03cm}\widehat C_{j+1}\mathcal{O}^*(h\mathcal{O} \widehat C_{j+1}\mathcal{O}^*+\Gamma_0)^{-1}\mathcal{O}^*=h\hspace{.03cm}\widehat C_{j+1}\mathcal{O}^*\Gamma_0^{-1}\mathcal{O}^*+\mathscr{O}(h^2),\\
&S_{j+1} ^{(k)}=\hbmk_{j+1} -\bar{\hbm}_{j+1}-\frac{1}{2}K^{-1}h\hspace{0.03cm} \widehat S_{j+1}\widehat S_{j+1}^*\mathcal{O}^*\Gamma_0^{-1}\mathcal{O}\big(\hbmk_{j+1} -\bar{\hbm}_{j+1}\big)+\mathscr{O}(h^2),
 \end{align} 
 where $\mathscr{O}(h^2)$ denote terms of order $h^2$. With the above formulas at hand (\ref{enkf_disc_h}) can be written as 
\begin{align}\label{enkf_disc_h2}
 \bmk_{j+1} &=  \Psi(\bmk_j)+h\hspace{.03cm}\widehat C_{j+1}\mathcal{O}^*\Gamma_0^{-1}\mathcal{O}^*\big(h^{-1}(\pmb{\eta}_{j+1}-\pmb{\eta}_j)-\mathcal{O}\,\bar{\hbm}_{j+1}\ \big)-h\mathscr{S}^{(k)}_{j+1}+\mathscr{O}(h^2),
\end{align}
where 
\begin{equation}
\mathscr{S}_{j+1}^{(k)}:=\frac{1}{2}(K-1)^{-1}\hspace{0.03cm} \widehat S_{j+1} \widehat S_{j+1}^*\mathcal{O}^*\Gamma_0^{-1}\mathcal{O}\big(\Psi( \bmk_{j})-\bar{\hbm}_{j+1}\big).
\end{equation} 
and we used the fact that $\hbmk_{j+1}=\Psi(\bmk_j)$ with $\bar\hbm_{j+1}= \frac{1}{K} {\textstyle \sum_{k=1}^K\Psi( \bmk_{j} )}$.
Finally, we observe that 
\begin{equation}
\Psi( \bmk_{j} ) =  \bmk_{j}+hF( \bmk_{j})+\mathscr{O}(h^2),
\end{equation}
where the last representation of the flow map arises through the general definition of the continuous-time dynamics in (\ref{gen_ode_ds}). Thus, (\ref{enkf_disc_h}) can be  written as 
\begin{align}\label{enkf_disc_h3}
\frac{\bmk_{j+1}-\bmk_{j} }{h}&=  F(\bmk_j)+\widehat C_{j+1}\mathcal{O}^*\Gamma_0^{-1}\mathcal{O}^*\big(h^{-1}(\pmb{\eta}_{j+1}-\pmb{\eta}_j)-\mathcal{O} \,\bar{\hbm}_{j+1} \big)-\mathscr{S}^{(k)}_{j+1}+\mathscr{O}(h),
\end{align}
so that,  in the limit $h\rightarrow 0$, we formally have 
\begin{align}\label{enkf_disc_h4}
\dr \bmk&=  F(\bmk)\dr t+C\mathcal{O}^*\Gamma_0^{-1}\mathcal{O}^*\big(\dr \pmb{\eta} -\mathcal{O}\, \bar{ \bm}  \dr t \big)-\mathcal{S}^{(k)}\dr t, 
\end{align}
where $\bar{ \bm}=\frac{1}{K} {\textstyle \sum_{k=1}^K \bmk}$,
\begin{equation}
\mathcal{S}^{(k)}:=\frac{1}{2}K^{-1}\hspace{0.03cm}  \mathcal{S}  \mathcal{S}^*\mathcal{O}^*\Gamma_0^{-1}\mathcal{O}\big(\bmk-\bar{ \bm}\big).
\end{equation} 
and 
\begin{equation}
\mathcal{S}(\bm) =   \big[ \bm^{(1)}- \bar{ \bm}, \dots,  \bm^{(K)}- \bar{\bm}\big]
\end{equation}
so that 
\begin{equation}
\mathcal{S}^{(k)}:=\frac{1}{2}\hspace{0.03cm}  C(\bm)\mathcal{O}^*\Gamma_0^{-1}\mathcal{O}\big(\bmk-\bar{ \bm}\big).
\end{equation}
Next, analogously to \cite{BLSZ2013} we observe that 
\begin{equation}\label{y_sde}
h^{-1}(\pmb{\eta}_{j+1}-\pmb{\eta}_j) = \by_{j+1} = \co \bu_{j+1}+\sqrt{h^{-1}\Gamma_0}\;\Delta \mathbf{w}_{j+1}, \qquad \Delta \mathbf{w}_{j+1} \in \cal H_\co,
\end{equation}
where $ (\Delta \mathbf{w}_{j})_{j\in \mathbb{Z}}$ is an i.i.d.~sequence in $\cal H_\co$ such that $\mathbf{w}_{j} \sim \mathcal{N}(0,I)$ for all $j\in \mathbb{Z}$. Note that (\ref{y_sde}) corresponds to the Euler-Maruyama discretisation of the SDE given by 
\begin{equation}\label{y_sde1}
\dr \pmb{\eta}=  \co \bu \,\dr t+\sqrt{\Gamma_0}\;\dr  \mathbf{W}_{t}.
\end{equation}
Gathering the above results leads to (with $\Gamma_0 = \sigma^2$)
 \begin{align}  \label{ensrkf_t_app}
\dr \bmk &= \Big\{F(\bmk)+\sigma^{-2}\hspace{.04cm}C(\bm) \co^* \co (\bu-\bar{ \bm})-\textstyle{\frac{1}{2}}\sigma^{-2}\hspace{.04cm}C(\bm) \co^* \co\big(\bmk-\bar{ \bm}\big)\Big\}\dr t+ \sigma^{-1}\hspace{.04cm}C(\bm)\co^* \dr \bW_t\notag\\
&=\Big\{F(\bmk)-\textstyle{\frac{1}{2}}\sigma^{-2}\hspace{.04cm}C(\bm) \co^* \co (\bmk-\bu)-\textstyle{\frac{1}{2}}\sigma^{-2}\hspace{.04cm}C(\bm) \co^* \co\big(\bar{ \bm}-\bu\big)\Big\}\dr t+ \sigma^{-1}\hspace{.04cm}C(\bm)\co^* \dr \bW_t
 \end{align}

\section{Proof of Theorem \ref{3dvar_exist}}\label{3dvar_exist_proof}
The proof is based on the classical change of variable $\bm = \tilde\bu +\bz  $ and a path-wise argument. 
First, consider the auxiliary process $\bz(t)$ which satisfies
\begin{align}\label{z_eq}
\dr \bz +\nu A\bz\hspace{.04cm}\dr t = \sigma^{-1}C\co^*\dr \bW_t,\qquad z(0) = 0
\end{align}
with $\bW_t$ the standard Wiener process in the (finite-dimensional) observation space $\cal H_\co\equiv\mathbb{R}^N$ and  the solution (see \cite{daprato92})
\begin{align}\label{z_sol}
\bz(t) = \sigma^{-1}\int_0^te^{-\nu A(t-s)}C\co^*\dr \bW_s
\end{align}
given by the stationary $D(A^{1/2})$ -- valued  process with continuous trajectories. Note that in our setup $\bz(t)$ might not be ergodic. The ergodicity (or lack thereof) of the filtering measure in this setup will be studied in a separate publication.

\smallskip
Moreover, for  $\bz(t)\in D(A^{1/2})$ in (\ref{z_sol}) we have 
\begin{align}
\E\|\bz(t)\|^2_{D(A^{1/2})}\leqslant \frac{1}{2\nu}\textrm{Tr}(C\co^*\co C).
\end{align}
The above estimate is obtained as follows:
\begin{align}
\bz(t) = \sum_{j=1}^\infty z_j(t)e_j \qquad \textrm{and} \qquad C\co^*\bW_t =  \sum_{j=1}^\infty \big(C\co^*\bW_t\big)_j\,e_j=  \sum_{j=1}^\infty \left(\sum_{n=1}^N(C\co^*)_{j,n}\,W^{(n)}_t\right)e_j, 
\end{align}
where $z_j(t) = (\bz(t),e_j)$, with $\{e_j\}_{j=1}^\infty$ the orthonormal eigenbasis  of the Stokes operator $A$, and  $\bW_t  = (W^{(1)}_t, \dots, W^{(N)}_t)$ with $N=\textrm{dim} \,\cal H_\co<\infty$ - the dimension of the observation space. 

Then, 
\begin{align}
z_j(t) =  \sigma^{-1}\sum_{n=1}^N(C\co^*)_{j,n}\int_0^te^{-\nu \lambda_j(t-s)}\dr W^{(n)}_s. 
\end{align} 
Using the independence of the components of ${\bf W}_t$ and the \^Ito isometry, we have 
\begin{align}
\E|z_j(t)|^2 &= \sigma^{-2} \sum_{n=1}^N(C\co^*)^2_{j,n}\,\E\left |\int_0^te^{-\nu \lambda_j(t-s)}\dr W^{(n)}_s\right|^2\notag\\
& = \sum_{n=1}^N(C\co^*)^2_{j,n}\,\left |\int_0^te^{-2\nu \lambda_j(t-s)}\dr s\right|^2\leqslant \frac{1}{2\nu\lambda_j} \sum_{n=1}^N(C\co^*)^2_{j,n}
\end{align}
Furthermore, provided that $2\alpha\leqslant 1$ we have 
\begin{align}
\E\|\bz(t)\|^2_{D(A^\alpha)} = \sum_{j=1}^\infty \lambda_j^{2\alpha} \,\E |z_j|^2\leqslant \frac{1}{2\nu}\sum_{j=1}^\infty\sum_{n=1}^N\frac{1}{\lambda_j^{1-2\alpha}}(C\co^*)^2_{j,n}\leqslant \frac{1}{2\nu\lambda_1^{1-2\alpha}}\textrm{Tr}(C\co^*\co C).
\end{align}
Now, using the change of variable $ \bm =\tilde\bu+\bz$ in (\ref{3dvareq}), we find that $\tilde \bu$ solves the following differential equation with random coefficients
\begin{align}\label{utilde}
\frac{\dr}{\dr t} \tilde \bu+\nu A\tilde \bu+B(\tilde \bu+{\bf z}, \tilde \bu+{\bf z})+\sigma^{-2} C\co^*\co(\tilde \bu+{\bf z}) = \tilde {\bf f}, \qquad \tilde {\bf f} = {\bf f}+\sigma^{-2} \Pi C\co^*\co \,\bu,
\end{align} 
where $\tilde \bu(0) = \tilde \bu_0 = \bu_0$. Theorem \ref{NSE_thm_1} implies that $\bu\in C\big([0,\,T];V\big)$. Thus, using the approximation of the identity (\ref{interpolant:est_1}) and the Poincar\'e inequality we have 
\begin{align}
\|C\co^*\co \,\bu\|_H\leqslant \|C\|\|\co^*\co \,\bu\|_H\leqslant c_1\|C\|\|\bu\|_H,
\end{align}
which implies that $C\co^*\co \,\bu\in C([0,\,T];H)$ and, consequently, $\tilde {\bf f} \in C([0,\,T];H)$. Similar to \cite[Theorem 3.1]{BOT} or \cite[Theorem 3.1]{flandoli94}, for every $\om\in \Om$ and $T>0$ there exists a unique weak solution $\tilde \bu$ of (\ref{utilde}) which depends continuously in the $C\big([0,\,T];H\big)\cap L^2\big([0,\,T];V\big)$ norms on the initial condition $\tilde\bu_0$ in $H$. The rigorous proof os this statement is rather long but  it relies on the same procedure as that in \cite[Theorem 3.1]{flandoli94} for the stochastically forced Navier-Stokes equation; namely it is based on the Galerkin approximation procedure and then passing to the limit using the appropriate compactness theorems. We do not repeat these steps and  just state the necessary a priori estimates which are different in our setup.  

\smallskip
To this end  consider 
\begin{align}\label{3dvar_NS_est1}
\left(\frac{\dr }{\dr t}\tilde \bu,\tilde\bu\right)+\nu(A\tilde\bu,\tilde\bu) = -\big(B(\tilde \bu+{\bf z}, \tilde \bu+{\bf z}),\tilde\bu\big) - \sigma^{-2}\big( C\co^*\co(\tilde \bu+{\bf z}),\tilde\bu\big) +\big(\tilde{\bf f},\tilde\bu\big)
\end{align}
for a fixed $\om\in \Omega$. 
Using the properties of $B$ in (\ref{dissip:eqn}), the Ladyzhenskaya interpolation inequality and Young's inequality, we obtain the following bound on  the first term on the right-hand-side of~(\ref{3dvar_NS_est1})  
\begin{align}
|\big(B(\tilde \bu+{\bf z}, \tilde \bu+{\bf z}),\tilde\bu\big)| &= |\big(B(\tilde \bu,\tilde\bu),{\bf z}\big) +\big(B({\bf z}, \tilde \bu),{\bf z}\big)|\notag\\
&\leqslant \|\tilde \bu\|_{L^4}\|\tilde \bu\|_V\|{\bf z}\|_{L^4}+\|\tilde \bu\|_V\|{\bf z}\|^2_{L^4}\notag\\
&\leqslant c_L^{1/2}\|\tilde \bu\|_V^{3/2}\|\tilde \bu\|^{1/2}_{H}\|{\bf z}\|_{L^4}+\|\tilde \bu\|_V\|{\bf z}\|^2_{L^4}\notag\\
&\leqslant \frac{3}{4}\nu\|\tilde\bu\|^2_V+\frac{c_L^2}{4\nu^3}\|\tilde\bu\|^2_H\|{\bf z}\|^4_{L^4}+\|\tilde \bu\|_V\|{\bf z}\|^2_{L^4}\notag\\
&\leqslant \left(\frac{3}{4}+\epsilon\right)\nu\|\tilde\bu\|^2_V+\frac{c_L^2}{4\nu^3}\|\tilde\bu\|^2_H\|{\bf z}\|^4_{L^4}+\frac{1}{4\nu\epsilon}\|{\bf z}\|^4_{L^4}\notag\\
&\leqslant \left(\frac{3}{4}+\epsilon\right)\nu\|\tilde\bu\|^2_V+\frac{c_L^4}{4\nu^3}\|\tilde\bu\|^2_H\|{\bf z}\|^2_{H}\|{\bf z}\|^2_{V}+\frac{c_L^2}{4\nu\epsilon}\|{\bf z}\|^2_{H}\|{\bf z}\|^2_{V}.
\end{align}

\noindent For the second term on the right-hand-side of (\ref{3dvar_NS_est1}) we use (\ref{interpolant:est_1}) and (\ref{T1_PKp}) with $\co^*\co = \I_h$ to  obtain 
\begin{align} \label{sigest_exist}
- 2\sigma^{-2}\big( C \I_h(\tilde \bu+{\bf z}),\tilde\bu\big) &=- 2\sigma^{-2}\big( C\I_h\,{\bf z},\tilde\bu\big)- 2\sigma^{-2}\big( C \I_h\,\tilde \bu,\tilde\bu\big)\notag\\[.2cm]
&\hspace{-0.8cm}=- 2\sigma^{-2}\big( C\I_h\,{\bf z},\tilde\bu\big)- 2\sigma^{-2}\big( C \I_h\PKp\,\tilde \bu,\tilde\bu\big)\notag\\[.2cm]
&\hspace{-0.8cm}\leqslant 2\sigma^{-2}\|C \I_h\,{\bf z}\|_H\|\tilde\bu\|_H-2\sigma^{-2}(C\I_h \PKp\tilde\bu, \PKp\tilde\bu)-2\sigma^{-2}(C\I_h\PKp\tilde\bu, (I-\PKp)\tilde\bu)  \notag\\[.2cm]
&\hspace{-.8cm}\leqslant 2\sigma^{-2}\|C\|\|\I_h\,{\bf z}\|_H\|\tilde\bu\|_H - 2\sigma^{-2}\beta \|\PKp\tilde\bu\|_H^2 +2\sigma^{-2}\|C\|\|\I_h\tilde\bu\|_H\|(I-\PKp)\tilde\bu\|_H\notag\\[.2cm]
&\hspace{-.8cm}\leqslant 2\sigma^{-2}c_1^2\|C\|\|{\bf z}\|_H\|\tilde\bu\|_H - 2\sigma^{-2}\beta \|\tilde\bu\|_H^2+ 2\sigma^{-2}\beta \|(I-\PKp)\tilde\bu\|_H^2 \notag\\[.2cm]
&+2\sigma^{-2}c_1\|C\|\|\tilde\bu\|_H\|(I-\PKp)\tilde\bu\|_H\notag\\[.2cm]
&\hspace{-.8cm}\leqslant \sigma^{-2}c_1^2\|C\|^2 \|{\bf z}\|^2_H+\sigma^{-2}\|\tilde\bu\|^2_H-2\sigma^{-2}\beta \|\tilde\bu\|_H^2+ 2\sigma^{-2}\beta {\hat c_2}^2h^2\|\tilde\bu\|_V^2 \notag\\[.2cm]
&+2\sigma^{-2}c_1\hat c_2h\|C\|\|\tilde\bu\|_H\|\tilde\bu\|_V\notag\\[.2cm]
&\hspace{-.8cm}\leqslant \sigma^{-2}c_1^2\|C\|^2 \|{\bf z}\|^2_H+\sigma^{-2}\|\tilde\bu\|^2_H-2\sigma^{-2}\beta \|\tilde\bu\|_H^2+ 2\sigma^{-2}\beta {\hat c_2}^2h^2\|\tilde\bu\|_V^2 \notag\\[.2cm]
&+\frac{c_1^2{\hat c_2}^2h^2\|C\|^2}{4\sigma^{4}\nu\epsilon}\|\tilde\bu\|^2_H+\epsilon \nu\|\tilde\bu\|^2_V.
\end{align}

\noindent Finally, 
\begin{align}
\big(\tilde{\bf f},\tilde\bu\big)\leqslant \|\tilde{\bf f}\|_H \|\tilde \bu\|_H\leqslant \lambda_1^{-1/2}\|\tilde{\bf f}\|_H \|\tilde \bu\|_V\leqslant \frac{1}{4\nu\epsilon\lambda_1}\|\tilde{\bf f}\|_H^2+\epsilon\nu\|\tilde \bu\|_V^2.
\end{align}
Hence, using the above estimates in (\ref{3dvar_NS_est1}) we have that 
\begin{align}
\frac{\dr }{\dr t}\|\tilde \bu\|^2_H +\left(\frac{1}{4}-3\epsilon-\frac{2\beta{\hat c_2}^2h^2}{\nu\sigma^2} \right)\nu \|\tilde\bu\|_V&\leqslant \frac{c_L^2}{4\nu\epsilon}\|{\bf z}\|^2_{H}\|{\bf z}\|^2_{V}+  \frac{c_1^2}{\sigma^2}\|C\|^2 \|{\bf z}\|^2_H+\frac{1}{4\nu\epsilon\lambda_1}\|\tilde{\bf f}\|_H^2\notag\\[.2cm]
&\hspace{-.2cm}+\left(\frac{c_L^4}{4\nu^3}\|{\bf z}\|^2_{H}\|{\bf z}\|^2_{V}+\frac{1}{\sigma^2}\left(1+2\beta+\frac{c_1^2{\hat c_2}^2h^2\|C\|^2}{4\sigma^{2}\nu\epsilon}\right)\right)\|\tilde\bu\|^2_H.
\end{align}
Since $\tilde {\bf f}\in C([0,\,T];H)$ and $\bz\in C([0,\,T];V)$, by Gronwall's lemma we have that for $\frac{2\beta{\hat c_2}^2h^2}{\nu\sigma^2}<\frac{1}{4}-3\epsilon$
\begin{align}
\underset{t\in [0,\,T]}{\sup} \|\tilde \bu(t)\|^2_H\leqslant \mathfrak{C}, \qquad \int_0^T\|\tilde \bu(s)\|_V^2\,\dr s\leqslant  \mathfrak{C}.
\end{align}
Finally, since $\bm = \tilde \bu +\bz $, we have that 
\begin{align}
\bm\in C([0,\,T];H)\cap L^2([0,\,T];V) \quad \p\textrm{-a.s.}
\end{align}
One can show that $\bm(t)$ is adapted (to the Wiener $\sigma$-algebra) from the limiting procedure of adapted processes via the Galerkin approximation. 

\medskip

In order to outline the proof of (\ref{m_int}), we carry out the calculations for $\bm(t)$ (the proper and standard procedure would be preceded by first considering the Galerkin approximation and then apply the limiting procedure to consider $\bm$ solving (\ref{3dvareq})); analogous steps are carried out in \cite{BOT, flandoli94}. Since, $\int_0^T {\rm Tr}(C\co^* \co C) \, \dr t < \infty $ as $C$ and $\co $ are either time-independent or not singular in space-time  the Ito formula holds (e.g., \cite{daprato92}) and we have that    
\begin{align}\label{dm2}
\dr \|\bm(t)\|_H^2 = 2(\bm(t),\dr \bm(t))+\sigma^{-2}\textrm{Tr}(C\co^*\co C).
\end{align}
Then, integrating (\ref{dm2})  over $[0,\,T]$ yields 
\begin{align}\label{3dvar_m_int_exist}
\sup_{0\leqslant t\leqslant T} \|\bm(t)\|_H^2+2\nu\int_0^T\|\bm(\tau)\|_V\dr \tau \leqslant& \|\bm_0\|_H^2+\int_0^T({\bf f},\bm(\tau))\dr \tau\notag\\
&- 2\sigma^{-2}\int_0^T\big(C\co^*\co(\bm(\tau)-\bu(\tau)),\bm(\tau)\big)\dr\tau\notag\\
&+2\sigma^{-1}\sup_{0\leqslant t\leqslant T} \int_0^t(\bm(\tau),C\co^*\dr {\bf W}_t)+\sigma^{-2}T\,\textrm{Tr}(C\co^*\co C),
\end{align}
where the last term can be made time-dependent but we simplify the setup.

The martingale term in (\ref{3dvar_m_int_exist}) can be bounded with the help of the Burkholder-Gundy-Davis inequality as 
\begin{align}
2\sigma^{-1}\E\left[\sup_{0\leqslant t\leqslant T} \int_0^t(\bm(\tau),C\co^*\dr {\bf W}_t)\right]&\leqslant 2\sigma^{-1}\sqrt{\textrm{Tr}(C\co^*\co C)}\,\,\E\left[\sqrt{\int_0^t\|\bm(\tau)\|_H^2\dr \tau}\,\right]\notag\\
& \leqslant  2\sigma^{-1}\sqrt{T\,\textrm{Tr}(C\co^*\co C)}\,\,\E\left[\sup_{0\leqslant t\leqslant T} \|\bm(t)\|_H^2\right]\notag\\
&\leqslant  2\sigma^{-2}\,T\,\textrm{Tr}(C\co^*\co C) +\frac{1}{2}\E\left[\sup_{0\leqslant t\leqslant T} \|\bm(t)\|_H^2\right].
\end{align}
For the second term in on the right-hand-side in (\ref{3dvar_m_int_exist}) we use (\ref{interpolant:est_1}) and (\ref{T1_PKp}) (with $\co^*\co = \I_h$ and $\mathcal{K} = \textrm{ker}\, \co$) to obtain 
\begin{align}
- 2\sigma^{-2}\big(C \I_h(\bm-\bu),\bm\big) 
&\leqslant -2\sigma^2(C\I_h \PKp\bm, \PKp\bm)-2\sigma^{-2}(C\I_h\bm, (I-\PKp)\bm)\notag\\[.1cm]
&\quad +2\sigma^{-2}\|C \I_h\bu\|_{H}\|\bm\|_H\notag\\[.2cm]
&\leqslant -2\sigma^{-2} \beta \|\PKp \bm\|^2_H+2\sigma^{-2}\|C\I_h\bm\|_H\| (I-\PKp)\bm\|_H\notag\\[.1cm]
&\qquad +2c_1\sigma^{-2}\|C\| \|\bu\|_H\|\bm\|_H\notag\\[.2cm]
&\hspace{0cm}\leqslant -2\sigma^{-2}\beta\|\bm\|^2_H+2\sigma^{-2}\beta\|(I-\PKp)\bm\|^2_H\notag\\[.1cm]
&\qquad +2c_1\sigma^{-2}\|C\|\|\bm\|_H\| (I-\PKp)\bm\|_H+2c_1\sigma^{-2}\|C\| \|\bu\|_H\|\bm\|_H\notag\\[.2cm]
&\hspace{0cm}\leqslant -2\sigma^{-2}\beta\|\bm\|^{2}_H+2\sigma^{-2}\beta \hat c_2^2h^2\|\bm\|^2_V\notag\\[.1cm]
&\qquad +2c_1\hat c_2\sigma^{-2}\|C\|\|\bm\|_H\| \bm\|_V+2c_1\sigma^{-2}\|C\| \|\bu\|_H\|\bm\|_H\notag\\[.2cm]
&\hspace{0cm}\leqslant -2\sigma^{-2}\beta\|\bm\|^{2}_H+2\sigma^{-2}\beta \hat c_2^2h^2\|\bm\|^2_V\notag\\[.1cm]
&\qquad +2c_1\hat c_2\sigma^{-2}\|C\|\|\bm\|_H\| \bm\|_V+2c_1\sigma^{-2}\|C\| \|\bu\|_H\|\bm\|_H\notag\\[.2cm]
&\hspace{0cm}\leqslant -2\sigma^{-2}\beta\|\bm\|^{2}_H+2\sigma^{-2}\beta \hat c_2^2h^2\|\bm\|^2_V\notag\\[.1cm]
&\qquad +c_1^2\hat c_2^2\sigma^{-4}\nu^{-1}\|C\|^2\|\bm\|_H^2 +\nu\|\bm\|^2_V+2c_1\sigma^{-2}\|C\| \|\bu\|_H\|\bm\|_H\notag\\[.2cm]
\end{align}
Combining the above estimates with the Gronwall lemma one obtains that for $2\beta{\hat c_2}^2h^2/(\nu\sigma^2)$ sufficiently small 
\begin{align}
\E\left[\sup_{0\leqslant t\leqslant T} \|\bm(t)\|_H^2\right]\leqslant C, 
\end{align} 
where $C=C(\|\bm_0\|_H,T,\sigma,\nu,h,\lambda_1, \|\bu\|_H, \textrm{Tr}(C\co^*\co C))$. Using this estimate again in (\ref{3dvar_m_int_exist}) finishes the proof.

\section{Proof of Proposition \ref{lemma_0} \label{lemma_0_app}}
The proof is obtained by combining the following two propositions and it is provided subsequently.
\begin{proposition}\label{lemma1} Let  $\K^\perp={\rm Span}\ \{\psi_n\}_{n=1}^N,
\co^*\co = P_{\K^\perp}$ and  $\widetilde C=P_{\K^\perp} C P_{\K^\perp}${\rm .} Then,
\begin{align}
 \frac{1}{K}\sum_{k=1}^K (\widetilde C\co^*\co \bek,\bek) =\frac{1}{K}\sum_{k=1}^K\sum_{n=1}^N (\bek,\psi_n)(C(\bee)P_{\K^\perp}\bek,\psi_n). 
\end{align}
\end{proposition}
\noindent {\it Proof}\,: This follows by a direct calculation. Note that 
\begin{align*}
\co^*\co \bek=P_{\K^\perp}\bek = \sum_{n=1}^N({\bek,\psi_n})\,\psi_n
\end{align*}
and 
\begin{align*}
\widetilde C \co^*\co\,\bek = \PKp C(\bee)\PKp\bek = \sum_{n=1}^N({\bek,\psi_n})\PKp C(\bee)\,\psi_n. 
\end{align*}
Next, we observe that 
\begin{align}
\hspace{2cm}\frac{1}{K}\sum_{k=1}^K (\widetilde C\co^*\co \bek,\bek) &=  \frac{1}{K}\sum_{k=1}^K\sum_{n=1}^N({\bek,\psi_n})(\PKp C(\bee)\,\psi_n, \bek)\notag\\[.2cm]
 &=  \frac{1}{K}\sum_{k=1}^K\sum_{n=1}^N({\bek,\psi_n})(\psi_n, C(\bee)\,
 \PKp\, \bek).\notag \hspace{3cm}\qed
\end{align}

\begin{proposition}\label{lemma2} The following holds for  $\co^*\co = P_{\K^\perp}$ and  $\widetilde C=P_{\K^\perp} C P_{\K^\perp}${\rm :}
\begin{align}
 \frac{1}{K}\sum_{k=1}^K\sum_{n=1}^N({\beb,\psi_n})(\widetilde C(\bee)\,\bek,\psi_n) = (\widetilde C(\bee)\co^*\co\, \beb,\beb) = (C(\bee)\PKp \beb, \PKp \beb).
 \end{align}
\end{proposition}
\noindent{\it Proof}\,: By direct calculation
\begin{align*}
 \frac{1}{K}\sum_{k=1}^K\sum_{n=1}^N({\beb,\psi_n})(\widetilde C(\bee)\,\bek, \psi_n) &= \sum_{n=1}^N\,({\beb,\psi_n})\left(\widetilde C(\bee)\,{\textstyle \frac{1}{K}\sum_{k=1}^K\bek}  , \psi_n\right) \notag\\
 &= \sum_{n=1}^N\,({\beb,\psi_n})\left(\widetilde C(\bee)\,\beb, \psi_n \right)\notag \\
 &=  \sum_{n=1}^N\,({\beb,\psi_n})\left(\PKp\,\beb, C(\bee) \psi_n \right).
 \end{align*}
 Observe that
 \begin{align*}
 C(\bee)\PKp\beb = C(\bee)\sum_{n=1}^N\,(\beb,\psi_n)\,\psi_n = \sum_{n=1}^N\,(\beb,\psi_n)\,C(\bee)\psi_n.
 \end{align*}
 This completes the proof. \qed

\medskip
 \noindent {\it Proof of Proposition \ref{lemma_0}.}  
In order to derive (\ref{avg_err_term}), recall that $\{\psi_n\}_{n=1}^N$ is an orthonormal basis for $\K^\perp$ and 
\begin{align*}
 \textrm{Tr}(\widetilde C(\bee)\co^*\co \widetilde C(\bee))=\textrm{Tr}(\widetilde C^2(\bee)\PKp) = \sum^N_{n=1}(\widetilde C^2\psi_n,\psi_n) = \sum^N_{n=1}(\widetilde C\psi_n,\widetilde C\psi_n). 
 \end{align*} 
Moreover,  since $\widetilde C(\bee) = \frac{1}{K}\sum_{k=1}^KP_{\K^\perp}(\bek-\beb)\otimes P_{\K^\perp}(\bek-\beb)$, as in (\ref{tldC}), we have 
\begin{align*}
\widetilde C(\bee) \psi_n = \frac{1}{K}\sum_{k=1}^K\,(\bek-\beb,\psi_n) P_{\K^\perp}(\bek-\beb),
\end{align*}
and thus 
\begin{align}\label{trOO}
\textrm{Tr}(\widetilde C(\bee)\co^*\co \widetilde C(\bee)) &= \frac{1}{K}\sum_{n=1}^N\sum_{k=1}^K\,(\bek-\beb,\psi_n)( P_{\K^\perp}(\bek-\beb), \widetilde C(\bee) \psi_n)\notag\\
&=\frac{1}{K}\sum_{n=1}^N\sum_{k=1}^K\,(\bek-\beb,\psi_n)( \widetilde C(\bee)(\bek-\beb),  \psi_n) \notag\\
& = \frac{1}{K}\sum_{n=1}^N\sum_{k=1}^K\,(\bek-\beb,\psi_n)( \widetilde C(\bee)\,\bek,  \psi_n)\notag\\
& = \frac{1}{K}\sum_{n=1}^N\sum_{k=1}^K\,(\bek,\psi_n)( \widetilde C(\bee)\,\bek,  \psi_n)-\frac{1}{K}\sum_{n=1}^N\sum_{k=1}^K\,(\beb,\psi_n)( \widetilde C(\bee)\,\bek,  \psi_n),
\end{align}
where the third equality is due to the fact that 
\begin{align}
\sum_{n=1}^N\,({\textstyle  \frac{1}{K}\sum_{k=1}^K}\bek-\beb,\psi_n)( \widetilde C\,\beb,  \psi_n) =  0. 
\end{align}
Now, we evaluate the terms in the last line of (\ref{trOO}) as follows:
\begin{align}\label{I11}
\sum_{n=1}^N\sum_{k=1}^K\,(\bek,\psi_n)( \widetilde C(\bee)\,\bek,  \psi_n) =& \sum_{n=1}^N\sum_{k=1}^K\,(\bek,\psi_n)( C(\bee)\PKp\,\bek,  \psi_n)  \nn \\
& = \sum_{k=1}^K (\widetilde C(\bee)\co^*\co \bek,\bek),
\end{align}
where the last equality due to Proposition \ref{lemma1}. The second term in the last line of (\ref{trOO})  can be evaluated as follows
\begin{align}\label{I12}
& \frac{1}{K}\sum_{n=1}^N\sum_{k=1}^K\,(\beb,\psi_n)( \widetilde C(\bee)\,\bek,  \psi_n) = \sum_{n=1}^N\,(\beb,\psi_n)( \widetilde C(\bee)\,\beb,  \psi_n) \nn \\
& \qquad \qquad =  (\widetilde C(\bee)\co^*\co \,\beb,\beb) =  ( C(\bee)\PKp  \,\beb,\PKp\beb)\geq 0, 
\end{align}
where the second equality follows from Proposition \ref{lemma2}. Thus, substituting (\ref{I11}) and (\ref{I12}) into (\ref{trOO}) leads to  (\ref{avg_err_term}). \qed


{\small \begin{thebibliography}{1000}


\bibitem{AlbanezNussenzveigLopesTiti2016}  D. Albanez,  H. Nussenzveig-Lopes, E. S.   Titi, Continuous data assimilation for the three-dimensional Navier-Stokes-$\alpha$ model, \emph{Asymptot. Anal.,} {\bf 97} (2016), pp. 139-164.

\bibitem{anderson12} J.L. Anderson, Localization and sampling error correction in ensemble Kalman filter data assimilation, {\it Mon. Weather Rev.}, {\bf 140}, 2359--2371, 2012.

\bibitem{astrom70} K.J. $\mathring{\rm A}$str\"om, {\it Introduction to Stochastic Control Theory}, 1st Edition, Volume 70, Academic Press Inc. (London), 1970.

\bibitem{AOT} A. Azouani, E. Olson and E. S. Titi, \emph{Continuous Data Assimilation Using General Interpolant Observables}, J. Nonlinear Sci., 24 (2014), pp. 277-304.

\bibitem{BC08} A. Bain and D. Crisan, {\it Fundamentals of Stochastic Filtering}, Springer 2008.

\bibitem{BBJ2020} A. Biswas, Z. Bradshaw and M. S. Jolly, 
Data assimilation for the Navier-Stokes equations using local observables, \emph{SIAM J. Appl. Dyn. Syst.} {\bf 20} (2021),  no. 4, 2174--2203.


\bibitem{BLSZ2013} D. Bl\"{o}mker, K. Law, A.M. Stuart and K.C. Zygalakis, Accuracy and stability of the continuous-time 3DVar filter for the Navier-Stokes equation, {\it Nonlinearity} {\bf 26 } (2013), no. 8, pp. 2193--2219.

\bibitem{BCJ11}  A. Beskos, D. Crisan and  A. Jasra, {\it On the stability of sequential Monte Carlo methods in high dimensions}, The Annals of Applied Probability {\bf 24}(4), 2011. 

\bibitem{BOT} H. Bessaih, E. Olson and E. S. Titi, \emph{Continuous data assimilation with stochastically noisy data}, Nonlinearity 28 (2015), pp. 729--753.

\bibitem{bishop07} C.H. Bishop and D. Hodyss, Flow-adaptive moderation of spurious ensemble correlations and its use in ensemble-based data assimilation, {\it Q. J. Roy. Meteor. Soc.}, {\bf 133}, 2029--2044, 2007.

\bibitem{bolin16} D. Bolin and J. Wallin, Spatially adaptive covariance tapering, {\it Spat. Stat.}, {\bf 18}, 163--178,  2016.

\bibitem{BM13} M. Branicki and A.J. Majda, {\it Dynamic Stochastic Superresolution of sparsely observed turbulent systems}, J. Comp. Phys., {\bf 241}(5), 333--363, 2013.

\bibitem{BM14} M. Branicki and A.J. Majda, Quantifying Bayesian filter performance for turbulent dynamical systems through information theory, Commun. Math. Sci. {\bf 12}(5), pp. 901--978,  2014.

\bibitem{BU20} M. Branicki and K. Uda, {\it Time-periodic measures, random periodic orbits, and the linear response for dissipative non-autonomous stochastic differential equations}, Res.  Math. Sci., {\bf 8}, 2020. 

\bibitem{BLLMSS11} C. E. A. Brett, K.F. Lam, K.J.H. Law, D.S. McCormick, M.R. Scott and A. M. Stuart, Accuracy and stability of filters for dissipative PDEs, 
\emph{Phys. D} {\bf 245} (2013), pp. 34--45.

\bibitem{carlson2018} E. Carlson, J. Hudson and A. Larios, Parameter recovery for the 2-dimensional Navier-Stokes equations via continuous data assimilation,  {\it SIAM J. Sci. Comput.} {\bf 42} (2020), no. 1, A250-A270.


\bibitem{CHLM} E. Carlson, J. Hudson, A. Larios, V. R. Martinez, E. Ng and J. P. Whitehead, 
 Dynamically learning the parameters of a chaotic system using partial observations, 
 {\em Discrete Contin. Dyn. Syst.}, {\bf 42} (2022), no. 8, 3809-3839.
 
 \bibitem{CLT} E. Carlson, A. Larios and E. S. Titi, Super-exponential convergence rate of a nonlinear continuous data assimilation algorithm: The 2D Navier-Stokes equations paradigm, 
 {\em arXiv:2304.01128.}

\bibitem{CGTU08} A. Carrassi, M. Ghil, A. Trevisan and F. Uboldi, {\it  Data assimilation as a nonlinear dynamical systems problem: stability and convergence of the prediction-assimilation system},  Chaos: Interdiscip. J. Nonlinear Sci. {\bf 18} 023112, 2008.

\bibitem{ChMT10}  A. Chorin, M. Morzfeld and  X. Tu, {\it Implicit particle filters for data assimilation},  Commun. Appl. Math. Computat. Sci., {\bf 5}, pp. 213--9,  2010. 

\bibitem{CJT97} B. Cockburn, D.A. Jones and E.S. Titi, {\it Estimating the number of asymptotic degrees of freedom for nonlinear dissipative systems},  Math. Comput. {\bf 66},  pp. 1073--87, 1997. 

\bibitem{cf} P. Constantin and C. Foias, \emph{Navier-Stokes Equations}, Chicago Lectures in Mathematics, University of Chicago Press, Chicago, IL, 1988.

\bibitem{CF88} P. Constantin and C. Foias, {\it  Navier–Stokes Equations }(Chicago, IL: University of Chicago Press) 1988.

\bibitem{gharamti18} M. El Gharamti, Enhanced Adaptive Inflation Algorithm for Ensemble Filters, {\it Mon. Weather Rev.}, {\bf 146}, 623--640, 2018.

\bibitem{Evens09}  G. Evensen,  {\it Data Assimilation: the Ensemble Kalman Filter}, 2009. 

\bibitem{FGMW} 
\newblock A. Farhat, N. E. Glatt-Holtz, V. R. Martinez, S. A. McQuarrie, J. P. Whitehead, \newblock Data Assimilation in Large Prandtl Rayleigh-Bénard Convection from Thermal Measurements.
\newblock \emph{ SIAM J. Appl. Dyn. Syst.} {\bf 19} (2020),
no. 1, 510--540.


\bibitem{FLT} A. Farhat, E. Lunasin  and E.S. Titi, \emph{On the Charney conjecture of data assimilation employing temperature measurements alone: the paradigm of 3D planetary geostrophic model}, Mathematics of Climate and Weather Forecasting, 2(1) (2016), pp. 61--74.


\bibitem{flandoli94} F. Flandoli, {\it Dissipativity and invariant measures for stochastic Navier--Stokes equations},  Nonlinear Diff. Eqns Appl.,  {\bf 1}, pp. 403--23, 1994.

\bibitem{FMRT01} C. Foias, O. Manley, R. Rosa and R. Temam,   {\it Navier--Stokes equations and turbulence} in Encyclopedia of Mathematics and Its Applications vol 83 (Cambridge: Cambridge University Press) 2001. 

\bibitem{FoiasMondainiTiti2016} C. Foias, C. F. Mondaini and E. S. Titi, \emph{A discrete data assimilation scheme for the solutions of the two-dimensional Navier-Stokes equations and their statistics}, SIAM J. Appl. Dyn. Syst., 15 (2016), no. 4, pp. 2109--2142.

 \bibitem{FT84} C. Foias C and R. Temam, {\it Determination of the solutions of the Navier--Stokes equations by a set of nodal values},  Math. Comput. {\bf 43}, pp. 117--33, 1984. 
 
 \bibitem{FTT91}  C. Foias and E.S. Titi,  {\it Determining nodes, finite difference schemes and inertial manifolds},  Nonlinearity {\bf 4},  pp. 135--53, 1991.
 
 \bibitem{frank18}  J. Frank and S. Zhuk, A detectability criterion and data assimilation for nonlinear differential equations, {\it Nonlinearity} {\bf 31}, 5235--5257, 2018.
 
 \bibitem{FLV} T. Franz, A. Larios and C. Victor, The bleeps, the sweeps, and the creeps: convergence rates for dynamic observer patterns via data assimilation for the 2D Navier-Stokes equations, {\em Comput. Methods Appl. Mech. Engrg.}, {\bf 392} (2022), Paper No. 114673, 19 pp.
 
 \bibitem{GM2013}
\newblock G. Gottwald and A. Majda.
\newblock A mechanism for catastrophic filter divergence in data assimilation for sparse observation networks.
\newblock \emph{Nonlinear Process. Geophys.} {\bf 20} (2013),
 70--12.


\bibitem{HM08}  J. Harlim and  A.J. Majda,  Filtering nonlinear dynamical systems with linear stochastic models,  {\em Nonlinearity}, {\bf 21}, 1281, 2008.

\bibitem{HOT11} K. Hayden, E. Olson and E.S. Titi, {\it Discrete data assimilation in the Lorenz and 2D Navier--Stokes equations},  Physica D {\bf 240}, pp. 1416--25, 2011.

 \bibitem{HT97} M.J. Holst and E.S. Titi,  {\it Determining projections and functionals for weak solutions of the Navier--Stokes equations},  Contemp. Math. {\bf 204}, pp. 125--38, 1997.
 
  \bibitem{JT92} D.A. Jones and E.S. Titi,  {\it Determining finite volume elements for the 2D Navier--Stokes equations},  Physica D {\bf 60},  pp. 165--74, 1992.
 
  \bibitem{JT93} D.A. Jones and E.S. Titi,  {\it Upper bounds on the number of determining modes, nodes and volume elements for the Navier--Stokes equations},  Indiana Univ. Math. J. {\bf 42} , pp. 875--87,  1993.


\bibitem{kalinpur} G. Kallianpur and C. Striebel, Estimation of stochastic systems: arbitrary system process with additive white noise observation errors, Ann. Math. Stat. {\bf 39}, pp. 785--801, 1968. 

\bibitem{Kal60} R.E. Kalman,  {\it A New Approach to Linear Filtering and Prediction Problems}, Transactions of the ASME--Journal of Basic Engineering, {\bf 82}, pp. 35--45, 1960.

\bibitem{Keller} Keller, H., \emph{Attractors and bifurcations of the stochastic Lorenz system,} Technical Report 389, Institut f\"ur Dynamiche Systeme, Universit\"at Bremen, 1996.


\bibitem{KLS2014} D. T. B. Kelly, K. J. H. Law, and A. M. Stuart,
Well-posedness and accuracy of the ensemble Kalman filter in discrete and continuous time, \emph{Nonlinearity} {\bf 27}, pp.  2579--2603, 2014.

\bibitem{KLS14} D.T.B. Kelly, K.J.H.  Law and  A.M. Stuart, {\it Well--posedness and accuracy of the ensemble Kalman filter in discrete and continuous time}, Nonlinearity {\bf 27}, pp. 2579--2603, 2014. 

\bibitem{KMT15} D. Kelly, A.J. Majda and X.T. Tong, {\it Concrete ensemble Kalman filters with rigorous catastrophic filter divergence}, PNAS, {\bf112} (34),  pp. 10589--10594, 2015.

\bibitem{menetrier15} B. M\'en\'etrier, T. Montmerle, Y. Michel, and L. Berre,  Linear Filtering of Sample Covariances for Ensemble-Based Data Assimilation. Part I: Optimality Criteria and Application to Variance Filtering and Covariance Localization, {\it Mon. Weather Rev.}, {\bf 143}, 1622--1643, 2015.

\bibitem{kt1} 
\newblock W. Kuang, A. Tangborn, Z. Wei and T. Sabaka, 
\newblock Constraining a numerical geodynamo model with 100 years of surface observations, \newblock \emph{Geophys. J. Int.} (2009) {\bf 179}, pp. 1458--1468.
 doi: 10.1111/j.1365-246X.2009.04376.x

 \bibitem{LS12} K.J.H Law and A. Stuart, {\it Evaluating Data Assimilation Algorithms}, Monthly  Weather  Review, {\bf 80}, pp. 3757--82, 2012.

\bibitem{LSZbook2015}  K. Law, A.M. Stuart and K.C. Zygalakis, \emph{Data Assimilation, A Mathematical Introduction}, Springer, 2015.

\bibitem{PVL09}  P.J. Van Leeuwen,  {\it Particle filtering in geophysical systems},  Mon. Weather Rev., {\bf 137}, pp. 4089--114, 2009. 
 
  \bibitem{PVL10}   P.J. Van Leeuwen,  {\it Nonlinear data assimilation in geosciences: an extremely efficient particle filter}, Q. J. R. Meteorol. Soc., {\bf 136}, pp. 1991--9, 2010.  

\bibitem{lipster01} R.S. Lipster and A.N. Shiryaev, {\it Statistics of Random Processes}, Springer Berlin, Heidelberg 2001. 

\bibitem{Lor86} A.C. Lorenc, {\it Analysis methods for numerical weather prediction},  Q. J. R. Meteorol. Soc., {\bf 112}, pp. 1177--94, 1986.

\bibitem{lopez21} S. Lopez-Restrepo,  E.D. Nino-Ruiz, L.G. Guzman-Reyes, A. Yarce, O. L. Quintero, N. Pinel, A. Segers, A.W. Heemink, An efficient ensemble Kalman Filter implementation via shrinkage covariance matrix estimation: exploiting prior knowledge, {\it Computational Geosciences} {\bf 25}, 985--1003, 2021.

\bibitem{Luenberger1971} D. Luenberger, An introduction to observers, {\em IEEE T. Automat. Contr.,}  {\bf 16}, pp. 596--602, 1971.

\bibitem{majdapdewaves} A.J. Majda, Introduction to PDEs and Waves for the Atmosphere and Ocean, Courant Lecture Notes. AMS   2003. 
 \bibitem{MHG10}  A.J. Majda, J. Harlim and  B. Gershgorin,  {\it Mathematical strategies for filtering turbulent dynamical systems},   Dyn. Syst., {\bf 27}, pp. 441--86, 2010.
 
 \bibitem{MH2010}
\newblock A. Majda and J. Harlim.
\newblock Catastrophic filter divergence in filtering nonlinear dissipative systems,
\newblock \emph{Commun. Math. Sci.}  {\bf 8} (2010), 27--43.

\bibitem{Nijmeijer2001} H. Nijmeijer, A dynamic control view of synchronization, \emph{Physica D} {\bf 154}, pp. 219--228, 2001.



\bibitem{OT11} E. Olson and E.S. Titi, {\it Determining modes for continuous data assimilation in 2D turbulence},  J. Stat. Phys. {\bf 113}, pp. 799--84, 2011.

\bibitem{PAS}
\newblock S. Pawar, S. Ahmed, O. San, A. Rasheed and I. M. Navon.
\newblock Long short-term memory embedded nudging schemes for nonlinear data assimilation of geophysical flows,
\newblock \emph{Phys. Fluids} {\bf 32}, 076606, 2020.

\bibitem{daprato92} G. Da Prato  and J. Zabczyk,  {\it }Stochastic Equations in Infinite Dimensions, Cambridge University Press, 1992.

\bibitem{roh15} S. Roh, M. Jun, I. Szunyogh, and M.G. Genton, Multivariate localization methods for ensemble Kalman filtering, {\it Nonlin. Processes Geophys.}, {\bf 22}, 723--735,  2015.

\bibitem{smith18} P.J. Smith, A.S. Lawless, N.K.  Nichols, Treating Sample Covariances for Use in Strongly Coupled Atmosphere Ocean Data Assimilation, {\it Geophys. Res. Lett.}, {\bf 45}, 445--454, 2018.

\bibitem{kt2} 
\newblock A. Tangborn and W. Kuang.
\newblock Geodynamo model and error parameter estimation using
geomagnetic data assimilation,
\newblock \emph{Geophys. J. Int.} {\bf 200}, pp. 664-675,  2015.


\bibitem{TR76} T. Tarn and Y. Rasis, {\it Observers for nonlinear stochastic systems},  IEEE Trans. Autom. Control {\bf 21}, pp. 441--8, 1976. 

\bibitem{Temambook1995} R. Temam, \emph{Navier-Stokes Equations and Nonlinear Functional Analysis}, 2nd ed., CBMS-NSF Regional Conference Series in Applied Mathematics, 66, SIAM, Philadelphia, PA, 1995.

\bibitem{Tem97} R. Temam,  {\it Infinite-Dimensional Dynamical Systems in Mechanics and Physics} (Applied Mathematical Sciences vol 68) 2nd edn (New York: Springer) 1997.

\bibitem{Thau1973} F. E. Thau, Observing the state of non-linear dynamic systems, {\em Int. J. Control,} {\bf 17} (1973), pp. 471--479.

\bibitem{TMK16} X.T. Tong, A.J. Majda and D. Kelly, {\it Nonlinear stability and ergodicity of ensemble based Kalman filters}, {\bf 29},  pp. 657--691, 2016. 

\bibitem{TMK15} X.T. Tong, A.J. Majda and D. Kelly, {\it Nonlinear stability of the ensemble Kalman filter with adaptive covariance inflation}, Comm. Math. Sci., {\bf 14}(5), 2015. 


  \bibitem{TK97}   Z. Toth and  E. Kalnay,  {\it Ensemble forecasting at NCEP and the breeding method},  Mon. Weather Rev., {\bf 125}, 3297, 1997. 

\bibitem{TP11} A. Trevisan and L. Palatella, {\it Chaos and weather forecasting: the role of the unstable subspace in predictability and state estimation problems},  Int. J. Bifurcation Chaos {\bf 21},  3389--415, 2011.

\bibitem{deVT20} J. de Viljes and X.T. Tong,  Analysis of a localised nonlinear ensemble Kalman Bucy filter with complete and accurate observations, Nonlinearity, {\bf 33}, pp. 4752--4782, 2020. 

\bibitem{deVRS18} J. de Viljes, S. Reich and W. Stannat, {\it Long-Time Stability and Accuracy of the Ensemble Kalman--Bucy Filter for Fully Observed Processes and Small Measurement Noise}, SIAM J. Appl. Dyn. Sys. {\bf 17}(2), pp. 1152--1181, 2018. 

 \bibitem{yoshida18} T. Yoshida and E.K. Kalnay, Correlation-Cutoff Method for Covariance Localization in Strongly Coupled Data Assimilation, {\it Mon. Weather Rev.}, {\bf 146}, 2881--2889, 2018.












 
\end{thebibliography}
}
\end{document}